\documentclass[11pt]{article}

\usepackage{amssymb,amsmath,bm}
\usepackage{amsthm}

\usepackage{textcomp}
\usepackage{enumerate}      
\usepackage{graphicx}        
\usepackage{caption}

\usepackage{mathrsfs}

\usepackage{url}

\renewcommand{\setminus}{{\smallsetminus}}

\usepackage{amssymb,amsmath,bm}
\usepackage{amsthm}
\usepackage{hyperref}
\usepackage{mathrsfs}
\usepackage{textcomp}
\usepackage{enumerate}
\usepackage{graphicx}
\usepackage{tikz}
\usepackage{verbatim}

\newtheorem{theorem}{Theorem}[section]
\newtheorem{lemma}[theorem]{Lemma}
\newtheorem{proposition}[theorem]{Proposition}
\newtheorem{definition}[theorem]{Definition}

\theoremstyle{remark}
\newtheorem{remark}[theorem]{Remark}

\theoremstyle{remark}

\numberwithin{equation}{section}

\begin{document}
\title{\bf Adjoint twisted Reidemeister torsion  and  Gram matrices}

\author{Ka Ho Wong and Tian Yang}

\date{}

\maketitle

\begin{abstract}  We compute the adjoint twisted Reidemeister torsion for  closed oriented hyperbolic $3$-manifolds and for hyperbolic $3$-manifolds with toroidal boundary.  In our formula, we consider the manifold as obtained by doing a Dehn-filling along suitable boundary components of a fundamental shadow link complement, and the formula is in terms of the logarithmic holonomy of the meridians of the boundary components. As an important special case, we also write down a formula of the adjoint twisted Reidemeister torsion for the double of a hyperbolic $3$-manifold with totally geodesic boundary in terms of the edge lengths of a geometric ideal triangulation of the manifold. These unexpected formulas were inspired by, and played an important role in, the study of the asymptotic expansion of  quantum invariants\,\cite{WY}.

\end{abstract}

\section{Introduction}

We compute the \emph{adjoint twisted Reidemeister torsion}  (see Section \ref{TRT})  for closed orientable hyperbolic $3$-manifolds and for orientable hyperbolic $3$-manifolds with toroidal boundary with a representation of the fundamental group into $\mathrm{PSL}(2;\mathbb C)$ for which the adjoint twisted Reidemeister torsion is defined. 

To present the $3$-manifolds, we use a $3$-dimensional analogue of the pair-of-pants decompositions for surfaces, known as the \emph{fundamental shadow link complements} (\cite{CT}, see also Section \ref{fsl}). The fundamental shadow link complements form a universal family of $3$-manifolds with toroidal boundary in the sense that all  orientable $3$-manifolds with empty or toroidal boundary can be obtained from one of them by doing a Dehn-filling along suitable boundary components\,\cite{CT}. Then in Theorem \ref{main1}, we obtain an explicit formula of the adjoint twisted Reidemeister torsion of the fundamental shadow link complements, which turns out to be a product of the square root of the determinant of the values of the \emph{Gram matrix function} (see Section \ref{Gram})  at the logarithmic holonomy of the meridians.  As a consequence, in the main result of this paper, Theorem \ref{main3}, we obtain an explicit formula of the adjoint twisted Reidemeister torsion of  hyperbolic $3$-manifolds obtained by doing a Dehn-filling along suitable boundary components of a fundamental shadow link complement.  By \cite{CT, HK2}, these manifolds contain most closed and cusped orientable hyperbolic $3$-manifolds in the sense explained in Remark \ref{most}

To the best of our knowledge, this is  by far the only explicit formula of the adjoint twisted Reidemeister torsion for most hyperbolic $3$-manifolds. It is worth mentioning that the  $1$-loop Conjecture\,\cite{DG} suggests another formula of this quantity for cusped hyperbolic $3$-manifolds in terms of the shape parameters. 
 
In a setting dual to that in  Theorem \ref{main1} and Theorem \ref{main3}, we in Theorem \ref{main2} compute the adjoint twisted Reidemeister torsion of the double of a geometrically ideally triangulated hyperbolic $3$-manifold with totally geodesic boundary, in terms of the edge lengths of the triangulation.

The relationship between the two intensively studied geometric quantities in our formulas, the adjoint twisted Reidemeister torsion and the Gram matrix, is completely unexpected, and is suggested by the asymptotic expansion  of various quantum invariants of $3$-manifolds proposed by the authors in \cite{WY}. This is one of the few examples where ideas from the study of quantum invariants shed light on a solution of purely geometric problems. In return, these formulas also play an essential role in the study of the asymptotic expansion of quantum invariants\,\cite{WY}.

\subsection{Fundamental shadow link complements}

\begin{theorem} \label{main1}  Let $M=\#^{d+1}(S^2\times S^1)\setminus L_{\text{FSL}}$ be the complement of a fundamental shadow link $L_{\text{FSL}}$ with $n$ components $L_1,\dots,L_n,$ which is the orientable double of the union of truncated tetrahedra $\Delta_1,\dots, \Delta_d$ along pairs of the triangles of truncation (see Section \ref{fsl}).

\begin{enumerate}[(1)] 
\item  Let $\boldsymbol m=(m_1,\dots, m_n)$ be the system of the meridians of a tubular neighborhood of the components of $L_{\text{FSL}}.$ For an $\boldsymbol m$-regular $\mathrm{PSL}(2; \mathbb C)$-character $[\rho]$ of $M$ (see Definition \ref{reg}), let $( u_1,
\dots, u_n)$ be the logarithmic holonomies of $\boldsymbol m$ in $\rho.$ For each $k\in\{1,\dots,d\},$ let $L_{k_1},\dots,L_{k_6}$ be the components of $L_{\text{FSL}}$ intersecting $\Delta_k,$ and let $\mathbb G_k=\mathbb G\Big(\frac{ u_{k_1}}{2},\dots,\frac{ u_{k_6}}{2}\Big)$ be the value of the Gram matrix function at $\Big(\frac{ u_{k_1}}{2},\dots,\frac{ u_{k_6}}{2}\Big).$ Then the adjoint twisted Reidemeister torsion of $M$ with respect to $\boldsymbol m$ (see Definition \ref{ATRT}) at $[\rho]$ is 
$$\mathbb T_{(M,\boldsymbol m)}([\rho])=\pm2^{3d}\prod_{k=1}^d \sqrt{\det\mathbb G_k}.$$

\item In addition to the conditions of (1), let $\boldsymbol \mu=(\mu_1,\dots,\mu_n)$ be a system of simple closed curves on $\partial M,$ and let $( u_{\mu_1},\dots,  u_{\mu_n})$ be their logarithmic holonomies which are functions of $( u_1,
\dots, u_n)$ near $[\rho].$ If $[\rho]$ is $\boldsymbol\mu$-regular, then the adjoint twisted Reidemeister torsion of $M$ with respect to $\boldsymbol\mu$ at $[\rho]$ is
 $$ \mathbb T_{(M,\boldsymbol\mu)}([\rho])=\pm2^{3d}\det\bigg(\frac{\partial ( u_{\mu_1},\dots, u_{\mu_n})}{\partial ( u_1,\dots, u_n)}\bigg)\prod_{k=1}^d \sqrt{\det\mathbb G_k},$$
 where $\frac{\partial ( u_{\mu_1},\dots, u_{\mu_n})}{\partial ( u_1,\dots, u_n)}$ is the Jacobian matrix of $( u_{\mu_1},\dots,  u_{\mu_n})$ with respect to $( u_1,
\dots, u_n)$ evaluated at $[\rho].$ 
\end{enumerate} 
\end{theorem}

\begin{remark} By  (\ref{cosh}) and the analyticity of both sides, the logarithmic  holonomies of the system of longitudes, and hence of any system of simple closed curves on $\partial M,$ can be explicitly written in terms of the $(u_1,\dots,  u_n).$ Therefore, the formula in (2) can be written explicitly in terms of $(u_1,\dots,  u_n).$
\end{remark}

\begin{remark} By \cite{P, NZ}, all the characters near that of the holonomy representation of the complete hyperbolic structure of $M$ are $\boldsymbol\mu$-regular for any  system of simple closed curves $\boldsymbol\mu$ on $\partial M.$
\end{remark}

\subsection{Hyperbolic  $3$-manifolds}

 As a consequence of Theorem \ref{main1}, we obtain in Theorem \ref{main3} a  formula of the adjoint twisted Reidemeister torsion for hyperbolic $3$-manifolds  with empty or toroidal boundary obtained by doing a  Dehn-filling along suitable boundary components of a fundamental shadow link complement, with a technique assumption on the holonomy representation of the hyperbolic structure. Recall from \cite{CT} that every orientable hyperbolic $3$-manifolds with empty or toroidal boundary can be obtained in this way, and from \cite{HK2} and as explained in Remark \ref{most} for most closed and cusped hyperbolic $3$-manifolds the technical assumption is satisfied.

Let $M$ be a fundamental shadow link complement as in Theorem \ref{main1}. 
For $m$ with $0\leqslant m\leqslant n,$ let $\boldsymbol\mu=(\mu_1,\dots,\mu_m)$ be a system of simple closed curves on $\partial M$ such that $\mu_i\subset T_i,$ and let $\boldsymbol\nu=(\nu_{m+1},\dots,\nu_n)$  be a system of simple closed curves on $\partial M$ such that $\nu_j\subset T_j.$  Let $M_{\boldsymbol\mu}$ be the  $3$-manifold  obtained from $M$ by doing the Dehn-filling along $\boldsymbol\mu.$ Then $\boldsymbol\nu$ can be considered as a system of simple closed curves on $\partial M_{\boldsymbol\mu}.$ If $m=n,$ then $\boldsymbol \nu=\emptyset$ and $M_{\boldsymbol\mu}$ is a closed  $3$-manifold,

 \begin{theorem} \label{main3}  Suppose $M_{\boldsymbol\mu}$ is hyperbolic. 
Let $[\rho_{\boldsymbol\mu}]$ be a $\boldsymbol\nu$-regular character of $M_{\boldsymbol\mu}$  and let $\rho$ be the restriction of $\rho_{\boldsymbol\mu}$ on $M.$ 
Let $( u_1, \dots, u_n)$ be the logarithmic holonomies of the system of meridians $\boldsymbol m$  in $[\rho]$ and for each $k\in\{1,\dots,d\},$ let $L_{k_1},\dots,L_{k_6}$ be the components of $L_{\text{FSL}}$ intersecting $\Delta_k$ and let $\mathbb G_k=\mathbb G\Big(\frac{ u_{k_1}}{2},\dots,\frac{ u_{k_6}}{2}\Big)$ be the value of the Gram matrix function at $\Big(\frac{ u_{k_1}}{2},\dots,\frac{ u_{k_6}}{2}\Big).$ Let $( u_{\mu_1},\dots, u_{\mu_m})$ and $( u_{\nu_{m+1}},\dots, u_{\nu_n})$ respectively be the logarithmic holonomies of $\boldsymbol\mu$ and $\boldsymbol\nu$ considered as functions of $( u_1, \dots, u_n)$ near $[\rho].$  Let $(\gamma_1,\dots,\gamma_m)$ be a system of  simple closed curves on $\partial M$ that are isotopic to the core curves of the solid tori filled in and let $( u_{\gamma_1},\dots, u_{\gamma_m})$ be their logarithmic holonomies in $[\rho].$ 
If $[\rho]$ is in the distinguished component of the $\mathrm{PSL}(2;\mathbb C)$-character variety of $M,$ then the adjoint twisted Reidemeister torsion of $M_{\boldsymbol\mu}$ with respect to $\boldsymbol \nu$ at $[\rho_{\boldsymbol\mu}]$ is 
 $$ \mathbb T_{(M_{\boldsymbol\mu},\boldsymbol\nu)}([\rho_{\boldsymbol\mu}])=\pm2^{3d-2m}\det\bigg(\frac{\partial ( u_{\mu_1},\dots, u_{\mu_m}, u_{\nu_{m+1}},\dots, u_{\nu_n})}{\partial ( u_1,\dots, u_n)}\bigg)\prod_{k=1}^d \sqrt{\det\mathbb G_k}\prod_{i=1}^m\frac{1}{\sinh^2\frac{ u_{\gamma_i}}{2}},$$
 where $\frac{\partial ( u_{\mu_1},\dots, u_{\mu_m}, u_{\nu_{m+1}},\dots, u_{\nu_n})}{\partial ( u_1,\dots, u_n)}$ is the Jacobian matrix of $( u_{\mu_1},\dots,  u_{\mu_m}, u_{\nu_{m+1}},\dots, u_{\nu_n})$ with respect to $( u_1,\dots, u_n)$ evaluated at $[\rho].$
 
 In particular, if $M_{\boldsymbol\mu}$ is closed, $\rho_{\boldsymbol\mu}$ is the holonomy representation of the hyperbolic structure and $[\rho]$ is in the distinguished component of the $\mathrm{PSL}(2;\mathbb C)$-character variety of $M,$ then the adjoint twisted Reidemeister torsion of $M_{\boldsymbol\mu}$ is 
  $$ \mathrm {Tor}(M_{\boldsymbol\mu};\mathrm{Ad}_{\rho_{\boldsymbol\mu}})=\pm2^{3d-2n}\det\bigg(\frac{\partial ( u_{\mu_1},\dots, u_{\mu_n})}{\partial ( u_1,\dots, u_n)}\bigg)\prod_{k=1}^d \sqrt{\det\mathbb G_k}\prod_{i=1}^n\frac{1}{\sinh^2\frac{ u_{\gamma_i}}{2}}.$$
 
\end{theorem}

\begin{remark}\label{most} From \cite{CT, HK2}, the manifolds $M_{\boldsymbol\mu}$ in Theorem \ref{main3} cover most closed and cusped orientable hyperbolic $3$-manifolds in the sense that for each boundary component $T_i$ of $M,$  except for at most $114$ simple closed curves $\mu_i,$  the complete hyperbolic metric on $M_{\boldsymbol\mu}$ can be connected to the complete hyperbolic metric on $M$ by a one-parameter family of hyperbolic cone metrics on $M.$ As a consequence, $[\rho]$  lies in the distinguished component of the $\mathrm{PSL}(2;\mathbb C)$-character variety of $M$ satisfying the condition in Theorem \ref{main3}.  We believe that this  condition could be removed and the formula holds for all the closed hyperbolic $3$-manifolds and hyperbolic $3$-manifolds with toroidal boundary with a $\mathrm{PSL}(2;\mathbb C)$-representation for which the adjoint twisted Reidemeister torsion is defined. 
\end{remark}


\subsection{Double of hyperbolic polyhedral $3$-manifolds}

\begin{theorem} \label{main2} Let $N$ be a hyperbolic polyhedral $3$-manifold which is the union of truncated tetrahedra $\Delta_1,\dots,\Delta_d$ along pairs of hexagonal faces, and let $M$ be the double of  $N$ with the double of the edges $e_1,\dots,e_n$ removed  (see Section \ref{dhp}). 
\begin{enumerate}[(1)] 
\item  For $i\in \{1,\dots,n\},$ let $l_i$ be the lengths of $e_i.$ Let $\boldsymbol l$ be the system of the preferred longitudes of $M$ with the logarithmic holonomies $(2l_1,\dots,2l_n).$ For each $k\in\{1,\dots,d\},$ let $e_{k_1},\dots,e_{k_6}$ be the edges intersecting $\Delta_k,$ and let $\mathbb G_k=\mathbb G( l_{k_1},\dots, l_{k_6})$ be the value of the Gram matrix function at $( l_{k_1},\dots, l_{k_6}).$ Let $\rho$ be the holonomy representation of the hyperbolic cone metric on $M$ obtained by doubling the hyperbolic polyhedral metric of $N.$ Then
 $$\mathbb T_{(M,\boldsymbol l)}([\rho])=\pm2^{3d}\prod_{k=1}^d \sqrt{\det\mathbb G_k}.$$

\item Let $\boldsymbol m$ be the system of meridians of a tubular neighborhood of the double of the edges, and let $(\theta_1,\dots,\theta_n)$ be the cone angles at the edges which are functions of the lengths $(l_1,\dots,l_n)$ of the edges of $N.$  Then
$$ \mathbb T_{(M,\boldsymbol m)}([\rho])=\pm \mathbf i^n2^{3d-n}\det\bigg(\frac{\partial (\theta_1,\dots,\theta_n)}{\partial (l_1,\dots,l_n)}\bigg)\prod_{k=1}^d \sqrt{\det\mathbb G_k},$$
where $\frac{\partial (\theta_1,\dots,\theta_n)}{\partial (l_1,\dots,l_n)}$ is the Jacobian matrix of $(\theta_1,\dots,\theta_n)$ with respect to $(l_1,\dots,l_n)$ evaluated at $[\rho].$

\item Suppose $\overline M$ is the double of a geometrically ideally triangulated hyperbolic $3$-manifold $N$ with totally geodesic boundary (which is $M$ with the removed double of edges filled back). Let $\rho$ and $\overline\rho$ respectively be the holomony representations of $M$ and $\overline M.$ Let $(l_1,\dots,l_n)$ be the lengths of the edges of $N$ and let  $(\theta_1,\dots,\theta_n)$ be the cone angles considered as functions of $(l_1,\dots,l_n).$ For each $k\in\{1,\dots,d\},$ let $e_{k_1},\dots,e_{k_6}$ be the edges intersecting $\Delta_k$ and let $\mathbb G_k=\mathbb G( l_{k_1},\dots, l_{k_6})$ be the value of the Gram matrix function at $( l_{k_1},\dots, l_{k_6}).$ Then 
 $$ \mathrm{Tor}(\overline M;\mathrm{Ad}_{\overline\rho})=\pm \mathbf i^n2^{3d-3n}\det\bigg(\frac{\partial (\theta_1,\dots,\theta_n)}{\partial (l_1,\dots,l_n)}\bigg)\prod_{k=1}^d \sqrt{\det\mathbb G_k}\prod_{i=1}^n\frac{1}{\sinh^2l_i},$$
 where $\frac{\partial (\theta_1,\dots,\theta_n)}{\partial (l_1,\dots,l_n)}$ is the Jacobian matrix of $(\theta_1,\dots,\theta_n)$ with respect to $(l_1,\dots,l_n)$ evaluated at $[\rho].$ 

\end{enumerate} 
\end{theorem}

\begin{remark} Since the cone angles $(\theta_1,\dots,\theta_n)$ are the sums of the dihedral angles which by  (\ref{cos}) can be explicitly written as functions of $(l_1,\dots,l_n),$ both of the formulas in (2) and (3) can be written explicitly in terms of the edge lengths $(l_1,\dots,l_n).$ \end{remark}

\begin{remark} We believe that a similar formula of the adjoint twisted Reidemeister torsion of a geometrically ideally triangulated cusped hyperbolic $3$-manifold and of a geometrically triangulated closed hyperbolic $3$-manifold  should also exist, respectively in terms of the decorated edge lengths and the edge lengths.\end{remark}



\subsection{Outline of the proof} The main tool in the computation is the Mayer-Vietoris formula stated in Theorem \ref{MaV}.  To use this formula, we in Sections \ref{TP} and \ref{TD} respectively compute the adjoint twisted Reidemeister torsion of the pairs of pants and of the $D$-blocks, and  in Section \ref{TMV} compute the Reidemeister torsion of the Mayer-Vietoris sequence. Then the results follow from Theorem \ref{MaV}.\\

\noindent\textbf{Acknowledgements.}  The authors would like to thank Francis Bonahon, Giulio Belletti, Yi Liu, Feng Luo, Tushar Pandey, Hongbin Sun, Zhizhang Xie and Seokbeom Yoon for helpful discussions. The authors are also grateful to the referees' invaluable suggestions, both in the mathematics and in the writing. The second author is supported by NSF Grants DMS-1812008 and DMS-2203334.


\section{Preliminaries}

\subsection{Reidemeister torsions}\label{TRT}

Let $\mathrm C_*$ be a finite chain complex 
$$0\to \mathrm C_d\xrightarrow{\partial}\mathrm C_{d-1}\xrightarrow{\partial}\cdots\xrightarrow{\partial}\mathrm C_1\xrightarrow{\partial} \mathrm C_0\to 0$$
of $\mathbb C$-vector spaces, and for each $\mathrm C_k$ choose a basis $\mathbf c_k.$ Let $\mathrm H_*$ be the homology of $\mathrm C_*,$ and for each $\mathrm H_k$ choose a basis $\mathbf h_k$ and a lift $\widetilde{\mathbf h}_k\subset \mathrm C_k$ of $\mathbf h_k.$ We also choose a basis $\mathbf b_k$ for each image $\partial (\mathrm C_{k+1})$ and a lift $\widetilde{\mathbf b}_k\subset \mathrm C_{k+1}$ of $\mathbf b_k.$ Then $\mathbf b_k\sqcup \widetilde{\mathbf b}_{k-1}\sqcup \widetilde{\mathbf h}_k$ form a new basis of $\mathrm C_k.$ Let $[\mathbf b_k\sqcup \widetilde{\mathbf b}_{k-1}\sqcup \widetilde{\mathbf h}_k;\mathrm c_k]$ be the determinant of the transition matrix from the  basis $\mathbf c_k$ to the new basis $\mathbf b_k\sqcup \widetilde{\mathbf b}_{k-1}\sqcup \widetilde{\mathbf h}_k.$
 Then the Reidemeister torsion of the chain complex $\mathrm C_*$ with the chosen bases $\{\mathbf c_k\}$ and $\{\mathbf h_k\}$ is defined by 
\begin{equation}
\mathrm{Tor}(\mathrm C_*, \{\mathbf c_k\}, \{\mathbf h_k\})=\pm\prod_{k=0}^d[\mathbf b_k\sqcup \widetilde{\mathbf b}_{k-1}\sqcup \widetilde{\mathbf h}_k;\mathbf c_k]^{(-1)^{k+1}}\in\mathbb C^*/\{\pm 1\}.
\end{equation}
It is easy to check that $\mathrm{Tor}(\mathrm C_*, \{\mathbf c_k\}, \{\mathbf h_k\})$ depends only on the choices of $\{\mathbf c_k\}$ and $\{\mathbf h_k\},$ and does not depend on the choices of $\{\mathbf b_k\}$ and the lifts $ \{\widetilde{\mathbf b}_k\}$ and $\{\widetilde{\mathbf h}_k\}.$

We recall the twisted Reidemeister torsion of a CW-complex following the conventions in \cite{P2}. Let $K$ be a finite CW-complex and let $\rho:\pi_1(K)\to\mathrm{SL}(N;\mathbb C)$ be a representation of its fundamental group. Consider the twisted chain complex 
$$\mathrm C_*(K;\rho)= \mathbb C^N\otimes_\rho \mathrm C_*(\widetilde K;\mathbb Z)$$
where $\mathrm C_*(\widetilde K;\mathbb Z)$ is the simplicial complex of the universal covering of $K$ and $\otimes_\rho$ means the tensor product over $\mathbb Z$ modulo the relation
$$\mathbf v\otimes( \gamma\cdot\mathbf c)=\Big(\rho(\gamma)^T\cdot\mathbf v\Big)\otimes \mathbf c,$$
where $T$ is the transpose, $\mathbf v\in\mathbb C^N,$ $\gamma\in\pi_1(K)$ and $\mathbf c\in\mathrm C_*(\widetilde K;\mathbb Z).$ The boundary operator on $\mathrm C_*(K;\rho)$ is defined by
$$\partial(\mathbf v\otimes \mathbf c)=\mathbf v\otimes \partial(\mathbf c)$$
for $\mathbf v\in\mathbb C^N$ and $\mathbf c\in\mathrm C_*(\widetilde K;\mathbb Z).$ Let $\{\mathbf e_1,\dots,\mathbf e_N\}$ be the standard basis of $\mathbb C^N,$ and let $\{c_1^k,\dots,c_{d^k}^k\}$ denote the set of $k$-cells of $K.$ Then we call
$$\mathbf c_k=\big\{ \mathbf e_i\otimes c_s^k\ \big|\ i\in\{1,\dots,N\}, s\in\{1,\dots,d^k\}\big\}$$
the standard basis of $\mathrm C_k(K;\rho).$ Let $\mathrm H_*(K;\rho)$ be the homology of the chain complex $\mathrm C_*(K;\rho)$ and let $\mathbf h_k$ be a basis of $\mathrm H_k(K;\rho).$ Then the Reidemeister torsion of $K$ twisted by $\rho$ with the basis $\{\mathbf h_k\}$ is 
$$\mathrm{Tor}(K, \{\mathbf h_k\}; \rho)=\mathrm{Tor}(\mathrm C_*(K;\rho),\{\mathbf c_k\}, \{\mathbf h_k\}).$$

By \cite{P}, $\mathrm{Tor}(K, \{\mathbf h_k\}; \rho)$ depends only on the conjugacy class of $\rho.$ By for example \cite{T}, the Reidemeister torsion is invariant under elementary expansions and elementary collapses of CW-complexes; and by \cite{M}  it is invariant under subdivisions, hence defines an invariant of PL-manifolds and of topological manifolds of dimension less than or equal to $3.$


A useful tool to compute the twisted Reidemeister torsion is the Mayer-Vietoris sequence. Suppose $K$ is a finite CW-complex and  $K_1, K_2, \dots, K_n$ are its sub-complexes. For $\{i,j\}\subset\{1,2,\dots, n\},$ let $K_{ij}=K_i\cap K_j$ if it is non-empty. Assume
\begin{enumerate}[(1)]
\item $K=K_1\cup K_2\cup\dots\cup K_n,$ and 
\item $K_i\cap K_j\cap K_k=\emptyset$ for all $\{i,j,k\}\subset\{1,\dots,n\}.$ 
\end{enumerate}
For a representation $\rho:\pi_1(K)\to \mathrm{SL}(N;\mathbb C),$ let $\rho_k$ and $\rho_{ij}$ respectively be the restriction of $\rho$ to $\pi_1(K_k)$ and $\pi_1(K_{ij}).$

\begin{lemma}\label{shortexact} The follow sequence of chain complexes 
\begin{equation}\label{se}
0\to\bigoplus_{\{i,j\}\subset\{1,\dots,n\}}\mathrm{C}_*(K_{ij};\rho_{ij})\xrightarrow{\delta}\bigoplus_{k=1}^n\mathrm{C}_*(K_k;\rho_k)\xrightarrow{\epsilon}\mathrm C_*(K;\rho)\to 0
\end{equation}
is exact, where $\epsilon$ is the sum defined by
$$\epsilon(\mathbf c_1,\dots,\mathbf c_n)=\sum_{k=1}^n\mathbf c_k$$
and $\delta$ is the alternating sum defined by
$$(\delta\mathbf c)_k=-\sum_{j=1}^{k-1}\mathbf c_{jk}+\sum_{l=k+1}^n\mathbf c_{kl}.$$
\end{lemma}
This short exact sequence can be found in for example \cite[Proposition 15.2]{BT} for untwisted complexes, and the proof for the twisted case is similar. The short exact sequence (\ref{se})  induces the following long exact sequence $\mathcal H:$
\begin{equation}\label{le}
\dots\to\mathrm H_{m+1}(K;\rho)\xrightarrow{\partial_{m+1}}\bigoplus_{\{i,j\}\subset\{1,\dots,n\}}\mathrm H_m(K_{ij};\rho_{ij})\xrightarrow{\delta_m}\bigoplus_{k=1}^n\mathrm H_m(K_k;\rho_k)\xrightarrow{\epsilon_m}\mathrm H_m(K;\rho)\to\dots,
\end{equation}
and the twisted Reidemeister torsion of $K$ can be computed by those of $\{K_{k}\},$ $\{K_{ij}\}$ and $\mathcal H.$

\begin{theorem}[Mayer-Vietoris]\label{MaV}(\cite[Proposition 0.11]{P})  Let $\mathbf h_*,$ $\{\mathbf h_{k,*}\}$ and $\{\mathbf h_{ij,*}\}$ respectively be bases of $\mathrm H_*(K;\rho),$ $\mathrm H_*(K_k;\rho_k)$ and $\mathrm H_*(K_{ij};\rho_{ij}),$ and let $\mathbf h_{**}$ be the union of  $\mathbf h_*,$ $\sqcup_k\mathbf h_{k,*}$ and $\sqcup_{\{i,j\}}\mathbf h_{ij,*}$ which is a basis of $\mathcal H.$ Then
$$
\mathrm{Tor}(K, \{\mathbf h_*\}; \rho)
= \pm \frac{\prod_{k=1}^n \mathrm{Tor}(K_k, \mathbf h_{k,*}; \rho_k)}{\prod_{\{i,j\}\subset\{1,\dots,n\}} \mathrm{Tor}(K_{ij}, \mathbf h_{ij,*}; \rho_{ij})\cdot \mathrm{Tor}(\mathcal{H}, \mathbf h_{**})}.
$$
\end{theorem}

In \cite[Proposition 0.11]{P}, Theorem \ref{MaV} is proved for the union of two sub-complexes, and the proof of the current form carries out in essentially the same way.
\subsection{Adjoint twisted Reidemeister torsions}\label{adrt}

In this section we recall results of Porti\,\cite{P} for the Reidemeister torsions of hyperbolic $3$-manifolds twisted by the adjoint action $\mathrm {Ad}_\rho=\mathrm {Ad}\circ\rho$ of an irreducible $\mathrm {PSL}(2;\mathbb C)$-representation $\rho.$ Here $\mathrm {Ad}$ is the adjoint action of $\mathrm {PSL}(2;\mathbb C)$ on its Lie algebra $\mathfrak{psl}(2;\mathbb C)\cong \mathbb C^3.$

For a closed orientable hyperbolic $3$-manifold  $M$ with the holonomy representation $\rho,$ by the Weil local rigidity theorem and the Mostow rigidity theorem,
$\mathrm H_k(M;\mathrm{Ad}_\rho)=0$ for all $k.$ Then the adjoint twisted Reidemeister torsion 
$$\mathrm{Tor}(M;\mathrm{Ad}_\rho)\in\mathbb C^*/\{\pm 1\}$$
 is defined without making any additional choice.

Now suppose  $M$ is a compact, orientable  $3$-manifold with boundary consisting of $n$ disjoint tori $T_1 \dots,  T_n$ whose interior admits a complete hyperbolic structure with  finite volume. Let $\mathrm X(M)$ be the $\mathrm{PSL}(2;\mathbb C)$-character variety of $M.$  

By \cite{CS, FG}, every irreducible component of $\mathrm X(M)$ has dimension greater than or equal to $n;$ and we denote by $\mathrm X^n(M)=\cup \mathrm X_k(M)$ the union of the irreducible components $\{\mathrm X_k(M)\}$ of $\mathrm X(M)$ that have dimension exact equal to $n.$ If $M$ is hyperbolic, then $\mathrm X^n(M)$ is non-empty because it contains the distinguished component $ \mathrm X_0(M)$ containing the character of the holomony representation of the complete hyperbolic structure of $M$\,\cite{T, NZ}.  The main reason that we consider the space $\mathrm X^n(M)$ in this article instead of $\mathrm X_0(M)$ is that: If $M_{\boldsymbol\mu}$ is a hyperbolic $3$-manifold obtained by doing a Dehn-filling along a system of simple closed curves $\boldsymbol\mu$ on $\partial M,$ then it is not clear whether the restriction of  the character of the holonomy representation of the hyperbolic structure on $M_{\boldsymbol\mu}$ to $M$ always lies in $\mathrm X_0(M);$ but it always lies in $\mathrm X^n(M)$ by a standard Mayer-Vietoris sequence argument. This fact will be used in the proof of Theorem \ref{main2}.

Below we recall two fundamental results (Theorem \ref{HM} and Theorem \ref{funT}) of Porti\,\cite{P}. Theorem \ref{funT} was originally proved for characters in $\mathrm X_0(M),$ but by essentially the same argument can be generalized to characters in $\mathrm X^n(X).$

We denote by $\mathrm X^{\text{irr}}(M)$ the Zariski-open subset of $\mathrm X(M)$ consisting of the irreducible characters.

\begin{theorem}\cite[Section 3.3.3]{P}\label{HM} For a system of simple closed curves $\boldsymbol\alpha=(\alpha_1,\dots,\alpha_n)$ on $\partial M$ with $\alpha_i\subset T_i,$ $i\in\{1,\dots,n\},$  and a character $[\rho]$ in a Zariski open subset of $\mathrm X_0(M)\cap\mathrm X^{\text{irr}}(M),$  we have:
\begin{enumerate}[(i)]
\item For $k\neq 1,2,$ $\mathrm H_k(M;\mathrm{Ad}\rho)=0.$
\item  For $i\in\{1,\dots,n\},$ up to scalar $\mathrm Ad_\rho(\pi_1(T_i))^T$ has a unique invariant vector $\mathbf I_i\in \mathbb C^3;$ and
$$\mathrm H_1(M;\mathrm{Ad}\rho)\cong \mathbb C^n$$ 
with a basis
$$\mathbf h^1_{(M,\alpha)}=\{\mathbf I_1\otimes [\alpha_1],\dots, \mathbf I_n\otimes [\alpha_n]\}$$
where $([\alpha_1],\dots,[\alpha_n])\in \mathrm H_1(\partial M;\mathbb Z)\cong 
\bigoplus_{i=1}^n\mathrm H_1(T_i;\mathbb Z).$
 
\item Let $([T_1],\dots,[T_n])\in \bigoplus_{i=1}^n\mathrm H_2(T_i;\mathbb Z)$ be the fundamental classes of $T_1,\dots, T_n.$ Then 
 $$\mathrm H_2(M;\mathrm{Ad}\rho)\cong\bigoplus_{i=1}^n\mathrm H_2(T_i;\mathrm{Ad}\rho)\cong \mathbb C^n$$ 
with  a basis 
$$\mathbf h^2_M=\{\mathbf I_1\otimes [T_1],\dots, \mathbf I_n\otimes [T_n]\}.$$
\end{enumerate}
\end{theorem}

\begin{remark}[\cite{P,NZ,HK}]\label{Rm} Important examples of the characters in Theorem \ref{HM} include the character of the holonomy representation of the complete hyperbolic structure on the interior of $M,$
 the restriction of the holonomy representation of the closed $3$-manifold $M_{\boldsymbol\mu}$ obtained from $M$ by doing the hyperbolic Dehn-filling along the system of simple closed curves $\boldsymbol\mu$ on $\partial M,$
 and the holonomy representation of a hyperbolic structure on the interior of $M$ whose completion is a conical manifold with cone angles less than $2\pi.$
\end{remark}

\begin{definition}\label{reg} Let $\boldsymbol\alpha=(\alpha_1,\dots,\alpha_n)$ be a system of simple closed curves on $\partial M$ with $\alpha_i\subset T_i,$ $i\in\{1,\dots,n\}.$  A character $[\rho]$ in $\mathrm X^n(M)\cap\mathrm X^{\text{irr}}(M)$ is \emph{$\boldsymbol\alpha$-regular} if condition (ii) in Theorem \ref{HM} is satisfied.
\end{definition}

\begin{remark}  We notice that Definition \ref{reg} does not only consider characters  in   the distinguished component $\mathrm X_0(M),$ but also considers characters in $\mathrm X^n(M).$  By \cite[Proposition 3.22]{P}, for characters in $\mathrm X_0(M),$ our definition of the $\boldsymbol\alpha$-regularity is equivalent to \cite[D\'efinition 3.21]{P}.
\end{remark}

It follows that  for any system of simple closed curves $\boldsymbol \alpha$ on $\partial M,$  the $\boldsymbol \alpha$-regular characters are smooth points of $\mathrm X(M);$ and the logarithmic holonomies of $\boldsymbol \alpha$  form a local parametrization of $\mathrm X(M)$ near each of the $\boldsymbol \alpha$-regular characters. Here for a $\mathrm{PSL}(2;\mathbb C)$-character $[\rho],$ the \emph{logarithmic holonomy} of $\alpha_i$ is defined up to sign as the logarithm of the ratio of the eigenvalues of $\rho([\alpha_i]).$

\begin{definition} \label{ATRT}
The \emph{adjoint twisted Reidemeister torsion} of $M$  \emph{with respect to $\boldsymbol\alpha$} is the function 
$$\mathbb T_{(M,\boldsymbol\alpha)}: \mathrm X^n(M)\cap\mathrm X^{\text{irr}}(M)\to\mathbb C/\{\pm 1\}$$ 
defined by
$$\mathbb T_{(M,\boldsymbol\alpha)}([\rho])=\mathrm{Tor}(M, \{\mathbf h^1_{(M,\alpha)},\mathbf h^2_M\};\mathrm{Ad}_\rho)$$
if $\rho$ is $\boldsymbol \alpha$-regular, and by $0$ if otherwise. 
\end{definition}

\begin{theorem}\cite[Theorem 4.1]{P}\label{funT}
Let $M$ be a compact, orientable  $3$-manifold with boundary consisting of $n$ disjoint tori $T_1 \dots,  T_n$ whose interior admits a complete hyperbolic structure with  finite volume. Let  $\mathbb C(\mathrm X^n(M)\cap\mathrm X^{\text{irr}}(M))$ be the ring of rational functions over $\mathrm X^n(M)\cap\mathrm X^{\text{irr}}(M).$ Then
\begin{equation*}
\begin{split}
\mathrm H_1(\partial M;\mathbb Z)&\to \mathbb C(\mathrm X^n(M)\cap\mathrm X^{\text{irr}}(M))\\
\quad\quad\boldsymbol\alpha\quad\ \  &\mapsto \quad\quad\quad\mathbb T_{(M,\boldsymbol\alpha)}
\end{split}
\end{equation*}
 up to sign defines  a  function which is a $\mathbb Z$-multilinear homomorphism with respect to the direct sum $\mathrm H_1(\partial M;\mathbb Z)\cong 
\bigoplus_{i=1}^n\mathrm H_1(T_i;\mathbb Z)$ satisfying the following properties:
\begin{enumerate}[(i)]
\item For a system of simple closed curves $\boldsymbol\alpha$ on $\partial M,$ if the component $\mathrm X_k(M)$ contains an $\boldsymbol\alpha$-regular character, then the support of $\mathbb T_{(M,\boldsymbol\alpha)}$
contains a Zariski-open subset of $\mathrm X_k(M)\cap\mathrm X^{\text{irr}}(M)$ consisting of all the $\boldsymbol\alpha$-regular characters in $\mathrm X_k(M).$ 

\item \emph{(Change of Curves Formula).} Let $\boldsymbol\beta=\{\beta_1,\dots,\beta_n\}$ and $\boldsymbol\gamma=\{\gamma_1,\dots,\gamma_n\}$ be two systems of simple closed curves on $\partial M.$ Let $( u_{\beta_1},\dots,  u_{\beta_n})$ and $( u_{\gamma_1},\dots, u_{\gamma_n})$ respectively be the logarithmic holonomies of the curves in $\boldsymbol\beta$ and $\boldsymbol\gamma.$ Then we have the equality of rational functions
\begin{equation}\label{coc}
\mathbb T_{(M,\boldsymbol\beta)}
=\pm\det\bigg( \frac{\partial ( u_{\beta_1},\dots, u_{\beta_n})}{\partial ( u_{\gamma_1},\dots, u_{\gamma_n})}\bigg)\mathbb T_{(M,\boldsymbol\gamma)}
\end{equation}
on $\mathrm X_k(M)\cap\mathrm X^{\text{irr}}(M)$ for the component $\mathrm X_k(M)$  containing a $\boldsymbol \gamma$-regular character, where $\frac{\partial ( u_{\beta_1},\dots, u_{\beta_n})}{\partial ( u_{\gamma_1},\dots, u_{\gamma_n})}$ is the Jocobian matrix of $( u_{\beta_1},\dots, u_{\beta_n})$ with respect to $( u_{\gamma_1},\dots, u_{\gamma_n}).$

\item \emph{(Surgery Formula).} For $m$ with $0\leqslant m\leqslant n,$ let $\boldsymbol\mu=(\mu_1,\dots,\mu_m)$ be a system of simple closed curves on $\partial M$ such that $\mu_i\subset T_i,$ and let $\boldsymbol\nu=(\nu_{m+1},\dots,\nu_n)$  be a system of simple closed curves on $\partial M$ such that $\nu_j\subset T_j.$ Let $M_{\boldsymbol\mu}$ be a hyperbolic $3$-manifold obtained from $M$ be doing the Dehn-filling along $\boldsymbol\mu.$ Then $\boldsymbol\nu$ can be considered as a system of simple closed curves on $\partial M_{\boldsymbol\mu}.$ Let $[\rho_{\boldsymbol\mu}]\in\mathrm X^{n-m}(M_{\boldsymbol\mu})\cap \mathrm X^{\text{irr}}(M_{\boldsymbol\mu})$ and let $[\rho]\in \mathrm X^n(M)\cap\mathrm X^{\text{irr}}(M)$ be the restriction of $[\rho_{\boldsymbol\mu}]$ on $M.$ Let $( u_{\gamma_1},\dots, u_{\gamma_m})$ be the logarithmic holonomies in $\rho$ of the core curves $\gamma_1,\dots,\gamma_m$ of the solid tori filled in. If $\rho_{\boldsymbol\mu}$ is $\boldsymbol\nu$-regular, then $\rho$ is $\boldsymbol\mu\cup\boldsymbol\nu$-regular, and 
\begin{equation}\label{sf}
\mathbb T_{(M_{\boldsymbol\mu},\boldsymbol\nu)}([\rho_{\boldsymbol\mu}])=\pm\mathbb T_{(M,\boldsymbol\mu\cup\boldsymbol\nu)}([\rho])\prod_{i=1}^m\frac{1}{4\sinh^2\frac{ u_{\gamma_i}}{2}}.
\end{equation}
\end{enumerate}
\end{theorem}

\subsection{Gram matrix function and truncated hyperideal tetrahedra}\label{Gram}

\begin{definition}\label{GMF} Let $\mathrm M_{4\times 4}(\mathbb C)$ be the space of $4\times 4$ matrices with complex entries. The \emph{Gram matrix function} 
$$\mathbb G:\mathbb C^6\to \mathrm M_{4\times 4}(\mathbb C)$$
 is defined for $\boldsymbol z=(z_{12}, z_{13}, z_{14}, z_{23}, z_{24}, z_{34})$ by 
\begin{equation*}
\begin{split}
\mathbb{G}(\boldsymbol z)=\left[
\begin{array}{cccc}
1 & -\cosh z_{12} & -\cosh z_{13} &-\cosh z_{14}\\
-\cosh z_{12}& 1 &-\cosh z_{23} & -\cosh z_{24}\\
-\cosh z_{13} & -\cosh z_{23} & 1 & -\cosh z_{34} \\
-\cosh z_{14} & -\cosh z_{24} & -\cosh z_{34}  & 1 \\
 \end{array}\right].
 \end{split}
\end{equation*}
\end{definition}

The values of $\mathbb G$ at different $\boldsymbol z$ recover the Gram matrices of a truncated hyperideal tetrahedron in the dihedral angles and in the edge lengths. Recall from  \cite{BB, F} that a truncated hyperideal tetrahedron $\Delta$ in $\mathbb H^3$ is a compact convex polyhedron that is diffeomorphic to a truncated tetrahedron in $\mathbb E^3$ with four hexagonal faces $\{H_1, H_2, H_3, H_4\}$ isometric to right-angled hyperbolic hexagons and four triangular faces $\{T_1,T_2,T_3,T_4\}$ isometric to hyperbolic triangles.  We call the four triangular faces the \emph{triangles of truncation}, and call the intersection of two hexagonal faces an \emph{edge} and the angle between these two hexagonal faces the \emph{dihedral angle} at this edge.

\begin{figure}[htbp]
\centering
\includegraphics[scale=0.3]{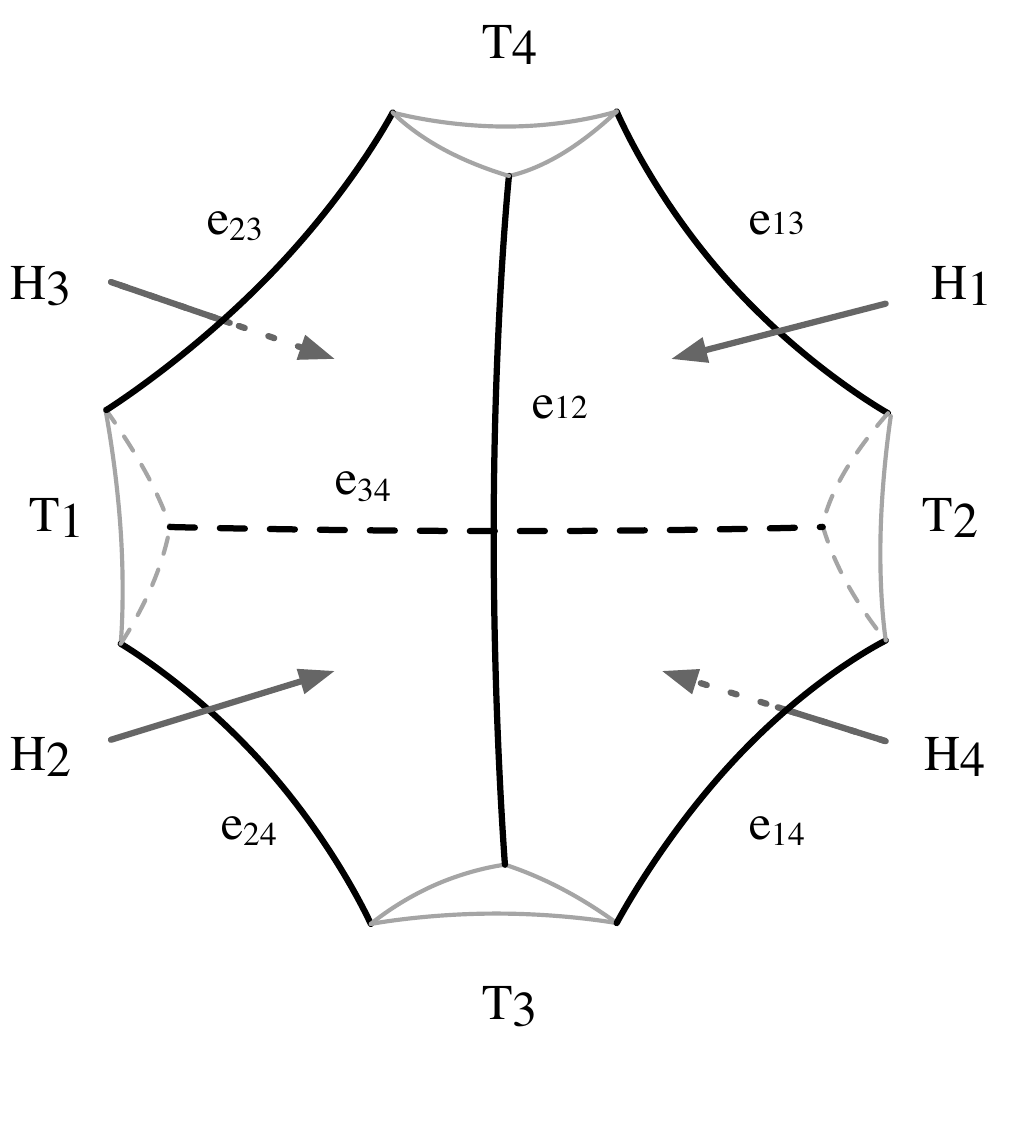}
\caption{}
\label{hyperideal} 
\end{figure}

For $\{i,j\}\subset \{1,2,3,4\},$ as in Figure \ref{hyperideal}, if we let $e_{ij}$ be the edge adjacent to the hexagonal faces $H_i$ and $H_j,$ and let $\alpha_{ij}$ and $l_{ij}$ respectively be the dihedral angle at and the length of $e_{ij},$ 
then the \emph{Gram matrix in the dihedral angles} of $\Delta$  is the  matrix
\begin{equation*}
\begin{split}
\mathrm{G}_{\boldsymbol\alpha} =\left[
\begin{array}{cccc}
1 & -\cos\alpha_{12} & -\cos\alpha_{13} &-\cos\alpha_{14}\\
-\cos\alpha_{12}& 1 &-\cos\alpha_{23} & -\cos\alpha_{24}\\
-\cos\alpha_{13} & -\cos\alpha_{23} & 1 & -\cos\alpha_{34} \\
-\cos\alpha_{14} & -\cos\alpha_{24} & -\cos\alpha_{34}  & 1 \\
 \end{array}\right].
 \end{split}
\end{equation*}
For $k,l\in\{1,2,3,4\},$ let $G_{\boldsymbol\alpha}^{kl}$ be the $kl$-th cofactor of $\mathrm{G}_{\boldsymbol\alpha}.$ Then by the hyperbolic Law of Cosine,  we have 
\begin{equation}\label{cosh}
\cosh l_{ij}=\frac{G_{\boldsymbol\alpha}^{kl}}{\sqrt{G_{\boldsymbol\alpha}^{kk}G_{\boldsymbol\alpha}^{ll}}},
\end{equation}
where $\{k,l\}=\{1,2,3,4\}\setminus\{i,j\}.$

\begin{figure}[htbp]
\centering
\includegraphics[scale=0.3]{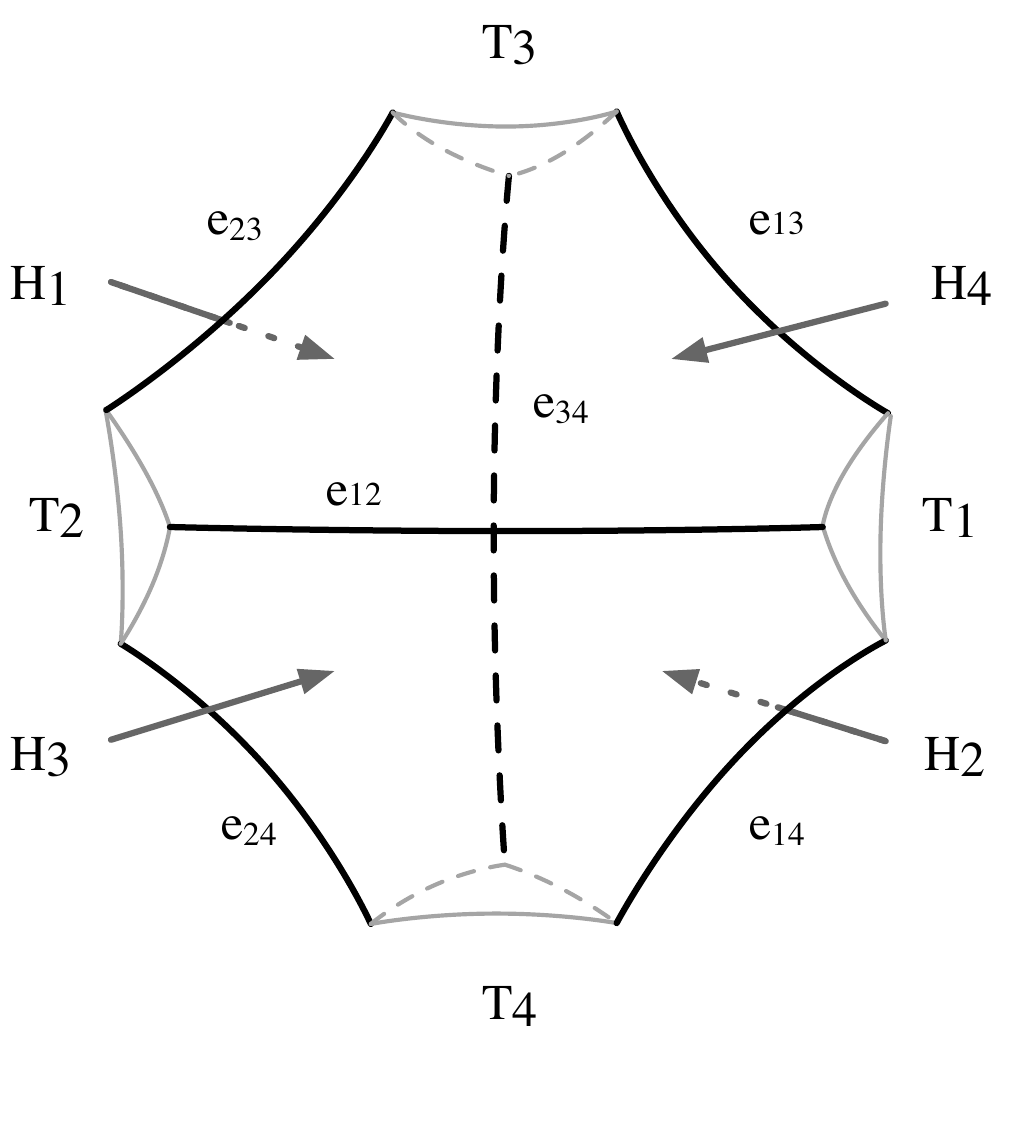}
\caption{}
\label{dualhyperideal} 
\end{figure}

For $\{i,j\}\subset \{1,2,3,4\},$ as in Figure \ref{dualhyperideal}, if we let $e_{ij}$ be the edge connecting the triangles of truncation $T_i$ and $T_j,$ and let $l_{ij}$ and $\alpha_{ij}$ respectively be the length of and the dihedral angle at $e_{ij},$ 
then the \emph{Gram matrix  in the edge lengths}  of $\Delta$ is the matrix
\begin{equation*}
\begin{split}
\mathrm{G}_{\boldsymbol l} =\left[
\begin{array}{cccc}
1 & -\cosh l_{12} & -\cosh l_{13} &-\cosh l_{14}\\
-\cosh l_{12}& 1 &-\cosh l_{23} & -\cosh l_{24}\\
-\cosh l_{13} & -\cosh l_{23} & 1 & -\cosh l_{34} \\
-\cosh l_{14} & -\cosh l_{24} & -\cosh l_{34}  & 1 \\
 \end{array}\right].
 \end{split}
\end{equation*}
For $k,l\in\{1,2,3,4\},$ let $G_{\boldsymbol l}^{kl}$ be the $kl$-th cofactor of $\mathrm{G}_{\boldsymbol l}.$ Then by the hyperbolic Law of Cosine,  we have 
\begin{equation}\label{cos}
\cos \alpha_{ij}=\frac{G_{\boldsymbol l}^{kl}}{\sqrt{G_{\boldsymbol l}^{kk}G_{\boldsymbol l}^{ll}}},
\end{equation}
where $\{k,l\}=\{1,2,3,4\}\setminus\{i,j\}.$ 

We observe that, if $\boldsymbol z=(\mathbf i\alpha_{12}, \mathbf i\alpha_{13}, \mathbf i\alpha_{14},\mathbf  i\alpha_{23},\mathbf  i\alpha_{24},\mathbf  i\alpha_{34}),$ then for $\Delta$ in Figure \ref{hyperideal},
$$\mathbb G(\boldsymbol z)=\mathrm G_{\boldsymbol \alpha};$$
and if $\boldsymbol z=(l_{12}, l_{13},l_{14}, l_{23}, l_{24}, l_{34}),$ then for $\Delta$ in Figure \ref{dualhyperideal},
$$\mathbb G(\boldsymbol z)=\mathrm G_{\boldsymbol l}.$$

\begin{remark} The way of assigning the edges $\{e_{ij}\}$ in the latter case is to consider $\Delta$ as a deeply truncated tetrahedron\,\cite{KM} that $T_1,\dots,T_4$ are the faces and $H_1,\dots,H_4$ are  the faces of truncation. In this way, $e_{ij}$ is the edge adjacent to or connecting the $i$-th and the $j$-th faces. For a general deeply truncated tetrahedron $\Delta,$ when two faces intersect we let $z_{ij}=\pm \mathbf i\alpha_{ij}$ and when two faces are disjoint we let $z_{ij}=\pm l_{ij},$ then $\mathbb G(\boldsymbol z)$ coincides with the Gram matrix of the deeply truncated tetrahedron $\Delta.$ See \cite[Section 2.1]{BY} for more details.
\end{remark}


\subsection{Fundamental shadow link complements}\label{fsl}

In this section we recall the construction and basic properties of the fundamental shadow link complements. The building blocks for a fundamental shadow link complement are truncated tetrahedra as on the left of Figure \ref{bb}. If we take $d$ building blocks $\Delta_1,\dots, \Delta_d$ and glue them together along the triangles of truncation, then we obtain a (possibly non-orientable) handlebody of genus $d+1$ with a link on its boundary consisting of the edges of the building blocks, such as the right of Figure \ref{bb}. By taking the orientable double (the orientable double
covering with the boundary quotient out by the deck involution) of this handlebody, we obtain a link $L_{\text{FSL}}$ inside $\#^{d+1}(S^2\times S^1).$ We call a link obtained this way a fundamental shadow link, and its complement $M=\#^{d+1}(S^2\times S^1)\setminus L_{\text{FSL}} $ a \emph{fundamental shadow link complement}. The fundamental shadow link complements form a universal family of $3$-manifolds in the following sense.

\begin{figure}[htbp]
\centering
\includegraphics[scale=0.25]{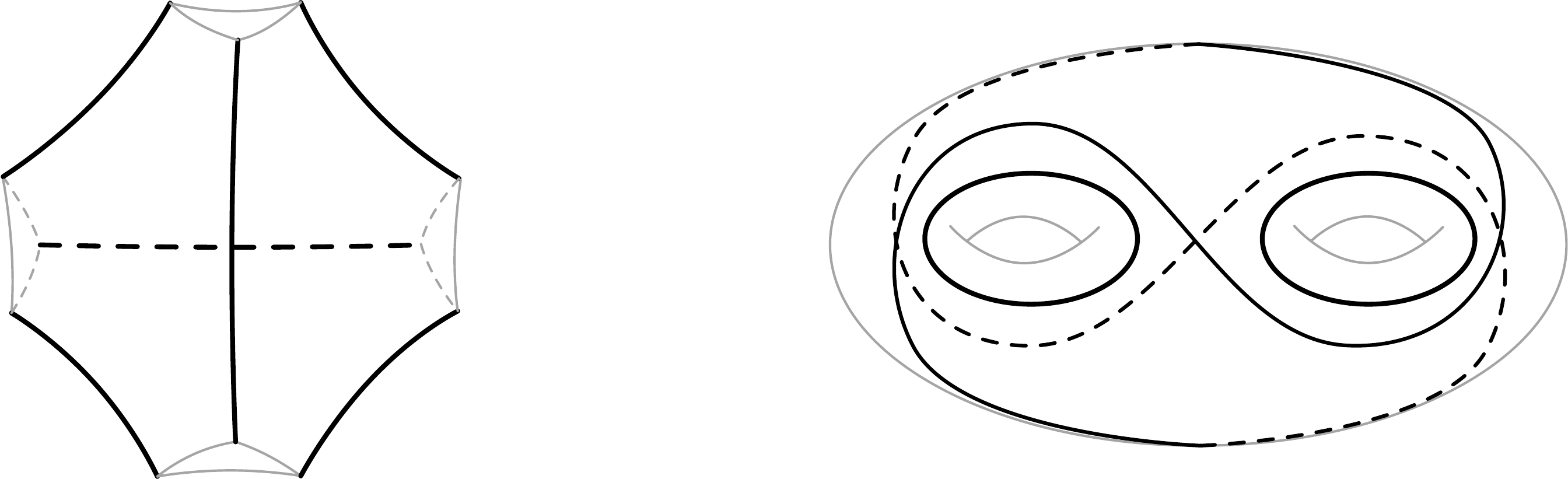}
\caption{The handlebody on the right is obtained from the truncated tetrahedron on the left by identifying the triangles on the top and the bottom by a horizontal reflection and the triangles on the left and the right by a vertical reflection.}
\label{bb}
\end{figure}

\begin{theorem}[\cite{CT}]\label{CT} Any compact orientable $3$-manifold with empty or toroidal boundary can be obtained from a fundamental shadow link complement by doing an integral Dehn-filling along suitable boundary components.
\end{theorem}

A hyperbolic cone metric on $\#^{d+1}(S^2\times S^1)$ with singular locus $L_{\text{FSL}}$ and with cone angles $2\alpha_1,$ $\dots,$ $2\alpha_n$ can be constructed as follows. For each $k\in \{1,\dots, d\},$ let $e_{k_1},\dots,e_{k_6}$ be the edges of the building block $\Delta_k,$ and let $2\alpha_{k_i}$ be the cone angle of the component of $L$ containing $e_{k_i}.$ Suppose  $\{{\alpha_{k_1},\dots,\alpha_{k_6}}\big\}$ form the set of dihedral angles of a truncated hyperideal tetrahedron, by abuse of notation still denoted by $\Delta_k.$ Then the hyperbolic metric on $M$ whose completion has singular locus $L_{\text{FSL}}$ with cone angles $2\alpha_1,\dots,2\alpha_n$ at the components is obtained by glueing the truncated hyperideal tetrahedra $\Delta_k$'s together along isometries between pairs of the triangles of truncation, then taking the orientable double. In this metric, the logarithmic holonomy  of the meridian  of a tubular neighborhood  of the $i$-th component of $L_{\text{FSL}}$ equals  $2\mathbf i\alpha_i.$ We also notice that when all the truncated hyperideal tetrahedra have edge lengths equal to zero, ie, are the regular ideal hyperbolic octahedra, we obtain the complete hyperbolic structure on $M.$

For the purpose of computing the adjoint twisted Reidemeister torsion, we need the following alternative construction of the fundamental shadow link complements. The idea is that, instead of gluing the truncated tetrahedra together along the triangles of  truncation first and then taking the orientable double, we  take the double of each tetrahedron first along the hexagonal faces and then glue the resulting pieces together along the pairs of the double of the triangles of truncation. To be precise, for each $\Delta_k,$ $k\in\{1,\dots,d\},$ we let $D_k$ be the union of $\Delta_k$ with its mirror image via the identity map between the four hexagonal faces and with the six edges  removed. In the language of \cite{C}, $D_k$ is a \emph{D-block}. The boundary of each $D_k$ is a union of four $3$-puncture spheres (coming from the double of the four triangles of truncation) and six cylinders (coming from the boundary of a tubular neighborhood of the edges). We glue these $D$-blocks together via orientation reversing homeomorphisms between pairs of $3$-puncture spheres part of the boundary, which send a triangle of truncation in one $3$-puncture sphere to a triangle of truncation in the other $3$-puncture sphere.  The quotient space is a fundamental shadow link complement, and this construction could be considered as a $3$-dimensional analog of the pair of pants decomposition of surfaces. 

We call the double of a truncated hyperideal tetrahedra a \emph{hyperbolic $D$-block}. Then a hyperbolic cone metric on $M$ can alternatively be constructed as by gluing the hyperbolic $D$-blocks together via orientation reversing isometries between the hyperbolic $3$-puncture spheres (double of the hyperbolic triangles of truncation with the three cone singularities removed) which preserve the hyperbolic triangles.

\subsection{Double of hyperbolic polyhedral $3$-manifolds}\label{dhp}

Dual to the construction of a fundamental shadow link complement is the construction of the double of a hyperbolic polyhedral $3$-manifold. As defined in \cite{Luo, LY}, a \emph{hyperbolic polyhedral $3$-manifold} $N$ is obtained from $d$  truncated hyperideal tetrahedra $\Delta_1,\dots,\Delta_d$ glued together via isometries between pairs of the hexagonal faces. The \emph{cone angle} at an edge is the sum of the dihedral angles of the truncated hyperideal tetrahedra around the edge. If all the cone angles are equal to $2\pi,$ then $N$ admits a hyperbolic metric with totally geodesic boundary and a geometric triangulation given by $\Delta_1,\dots,\Delta_d.$ It is proved in \cite[Theorem 1.2 (b)]{LY} that hyperbolic polyhedral $3$-manifolds are rigid in the sense that they are up to isometry determined and infinitesimally determined by their cone angles. 

To construct the double of $N,$ we can also take the double of each tetrahedron first along the triangles of truncation and then glue the resulting pieces together. To be precise, for each truncated tetrahedron $\Delta_k,$ $k\in\{1,\dots,d\},$ we let $D_k$ be the union of $\Delta_k$ with its mirror image via the identity map between the four triangles of truncation and with the double of the six edges removed. This is dual to the $D$-block in Section \ref{fsl}, hence we call it a \emph{dual $D$-block}. The boundary of each $D_k$ is a union of four $3$-hole spheres (coming from the double of the four hexagonal faces) and six cylinders (coming from the double of the boundary of a  tubular neighborhood of the edges). We then glue these dual $D$-blocks together via orientation reversing homeomorphisms between pairs of $3$-hole spheres, and the quotient space $M$ is the double of $N$ with the double of the edges removed. If we fill the double of edges back in, topologically we get the the double $\overline M$ of $N.$ 

Geometrically, if we let each truncated tetrahedron $\Delta_k$ be a truncated hyperideal tetrahedron, then the four $3$-hole spheres are hyperbolic $3$-hole spheres with geodesic boundary. If we require the gluing map between these hyperbolic $3$-hole spheres to be isometries, then the quotient space is the double $\overline M$  of the hyperbolic polyhedral $3$-manifold $N,$ and $M$ is obtained from $\overline M$ by removing all the double of the edges.

For $i\in\{1,\dots,n\},$ let $l_i$ be the length of the edge $e_i$ of the hyperbolic polyhedral manifold $N.$ Since $M$ comes from doubling, we can choose a \emph{preferred longitude} on the boundary of a tubular neighborhood of the double of $e_i$ (by isotoping $e_i$ into $\Delta_k$ and then doubling) whose logarithmic holonomy equals $2l_i.$


\section{Adjoint twisted Reidemeister torsion of the pairs of pants}\label{TP}

Let $P$ be a pair of pants with oriented boundary components $\gamma_1,$ $\gamma_2$ and $\gamma_3.$ Then $\pi_1(P)$  is a free group of rank $2$ generated by $[\gamma_{1}]$ and $[\gamma_{2}].$ By \cite{FK, G1}, the $\mathrm{SL}(2;\mathbb C)$-character variety of $P$ is homeomorphic to  $\mathbb C^3$ parametrized by the traces of the image of  $[\gamma_1],$ $[\gamma_2]$ and $[\gamma_3];$ and a representation  $\widetilde \rho:\pi_1(P)\to\mathrm{SL}(2;\mathbb C)$ is irreducible if and only if 
$$f_P\big(
\mathrm {Tr}\widetilde\rho([\gamma_1]),
\mathrm {Tr}\widetilde\rho([\gamma_2]),
\mathrm {Tr}\widetilde\rho([\gamma_3])\big)\neq 0,$$
where $f_P$ is the polynomial
$$f_P(x,y,z)=x^2+y^2+z^2-xyz-4.$$

 The \emph{logarithmic holonomies}  of $(\gamma_1,\gamma_2,\gamma_3)$ in a representation $\widetilde \rho:\pi_1(P)\to\mathrm{SL}(2;\mathbb C)$ are up to sign the complex numbers $(u_1,u_2,u_3)$ satisfying 
$$\big(
\mathrm {Tr}\widetilde\rho([\gamma_1]),
\mathrm {Tr}\widetilde\rho([\gamma_2]),
\mathrm {Tr}\widetilde\rho([\gamma_3])\big)=\Big(-2\cosh\frac{u_1}{2},-2\cosh\frac{u_2}{2},-2\cosh\frac{u_3}{2}\Big).$$
In this way, if $\rho_0:\pi_1(P)\to\mathrm{PSL}(2;\mathbb C)$ is the holonomy representation of the complete hyperbolic structure on $P$ and $\widetilde \rho_0:\pi_1(P)\to\mathrm{SL}(2;\mathbb C)$ is the lifting of $\rho_0$ with
 $$\big(
\mathrm {Tr}\widetilde\rho_0([\gamma_1]),
\mathrm {Tr}\widetilde\rho_0([\gamma_2]),
\mathrm {Tr}\widetilde\rho_0([\gamma_3])\big)=(-2,-2,-2),$$ then the logarithmic holonomies of $(\gamma_1,\gamma_2,\gamma_3)$ in $\widetilde\rho_0$ are $(0,0,0).$  The \emph{Gram matrix} of $\widetilde \rho$ is defined as
$$\mathbb G=\left[
\begin{array}{cccc}
1 & -\cosh\frac{u_3}{2}& -\cosh\frac{u_2}{2} \\
-\cosh\frac{u_3}{2}& 1 & -\cosh\frac{u_1}{2}\\
-\cosh\frac{u_2}{2} & -\cosh\frac{u_1}{2}  & 1  \\
 \end{array}\right].$$
 Then 
 $$f_P\big(
\mathrm {Tr}\widetilde\rho([\gamma_1]),
\mathrm {Tr}\widetilde\rho([\gamma_2]),
\mathrm {Tr}\widetilde\rho([\gamma_3])\big )=-4\det\mathbb G,$$
 and $\widetilde\rho$ is irreducible if and only if $\det\mathbb G\neq 0.$

Since $\pi_1(P)$ is a free group, every $\mathrm{PSL}(2;\mathbb C)$-representation of it lifts to  an $\mathrm{SL}(2;\mathbb C)$-representation. Hence the  $\mathrm{SL}(2;\mathbb C)$-character variety of $P$ is a branched cover of the $\mathrm{PSL}(2;\mathbb C)$-character variety of $P,$ and the latter is an irreducible algebraic variety.  For a representation $\rho:\pi_1(P)\to\mathrm{PSL}(2;\mathbb C),$ we defined the logarithmic holonomies $(u_1,u_2,u_3)$ and the Gram matrix $\mathbb G$ of $\rho$ as those of a lifting $\widetilde \rho:\pi_1(P)\to\mathrm{SL}(2;\mathbb C)$ of $\rho.$  Notice that the logarithmic holonomies depend on the choice of the liftings of $\rho,$ and a different lifting   will change $\mathbb G$ by multiplying some rows and the corresponding columns by $-1$ at the same time, which does not change its determinant. Therefore, the \emph{determinant of the Gram matrix} $\det\mathbb G$ is independent of the choice of the liftings, and  is a well defined quantity of $\rho.$

For a representation $\rho:\pi_1(P)\to \mathrm{PSL}(2;\mathbb C),$  let $\mathrm{Ad}_\rho =\mathrm{Ad}\circ\rho:\pi_1(P)\to \mathrm{SL}(3;\mathbb C )$ be its adjoint representation. Since both $\mathrm {Ad}$ and $\mathrm{Sym}^2$ are $3$-dimensional irreducible representations of $\mathrm {SL}(2;\mathbb C),$ they are equivalent by the Classification Theorem of finite dimensional irreducible representations of $\mathrm{SL}(2;\mathbb C).$ In the rest of this paper, we will use the representation $\mathrm{Sym}^2\circ\widetilde \rho$ to do all the computations where $\widetilde \rho$ is a lifting of $\rho$ to a representation into $\mathrm{SL}(2;\mathbb C);$ and to simplify the notation still denote it by $\mathrm{Ad}_\rho.$ Notice that composing with $\mathrm{Sym}^2,$ the signs $\pm$ in front of the matrices will disappear and hence $\mathrm{Sym}^2\circ\widetilde \rho$ is independent of the choice of the lifting $\widetilde \rho.$

In addition, assume for each $i\in\{1,2,3\}$ that $\rho([\gamma_i])\neq\pm I.$ Then up to sign we can canonically choose an invariant vector $\mathbf I_i$ of $\mathrm {Ad}_\rho([\gamma_i])^T$ as follows. Since $\rho([\gamma_i])$ is not the identity element in $\mathrm{PSL}(2;\mathbb C),$ there is  up to scalar a unique  invariant vector of $\mathrm {Ad}_\rho([\gamma_i])^T.$ To determine the scalar we consider the following Killing bilinear form 
$\kappa$ on the Lie algebra $\mathfrak{psl}(2;\mathbb C)\cong\mathbb C^3$ defined by 
\begin{equation}\label{Killing}
\kappa\Bigg(\begin{bmatrix}
a_1\\
a_2\\
a_3\\
 \end{bmatrix},
 \begin{bmatrix}
b_1\\
b_2\\
b_3\\
 \end{bmatrix}\Bigg)
=-2a_1b_3+a_2b_2-2a_3b_1,
\end{equation}
which is up to scalar the unique $\mathrm{Ad}$-invariant bilinear form on $\mathfrak{psl}(2;\mathbb C).$ Then in the case that  $\rho([\gamma_i])$ is not a parabolic element, we let $\mathbf I_i$ be up to sign the unique $\mathrm {Ad}_\rho([\gamma_i])^T$-invariant vector satisfying $\kappa(\mathbf I_i,\mathbf I_i)=1.$


\begin{definition}\label{Preg} Let $\boldsymbol\gamma=(\gamma_1,\gamma_2,\gamma_3).$ An irreducible representation $\rho:\pi_1(P)\to\mathrm{PSL}(2;\mathbb C)$ is \emph{$\boldsymbol\gamma$-regular} if 
 $$\mathbf h_P=\{\mathbf I_1\otimes [\gamma_1], \mathbf I_2\otimes[\gamma_2], \mathbf I_3\otimes[\gamma_3]\}$$
is a basis of $\mathrm{H}_1(P;\mathrm{Ad}_\rho),$ where $[\gamma_i]$ is the homology class of $\gamma_i$ in $\mathrm H_1(P;\mathbb Z),$ $i\in\{1,2,3\}.$
\end{definition}

Let $\mathrm X(P)$ be the $\mathrm{PSL}(2;\mathbb C)$-character variety of $P.$ A character $[\rho]\in\mathrm X(P)$ is  \emph{$\boldsymbol \gamma$-regular } if $\rho$ is a $\boldsymbol \gamma$-regular representation. Since $\pi_1(P)$ is a free group, an  Euler characteristic counting argument shows that if $[\rho]$ is $\boldsymbol\gamma$-regular, then $\mathrm{H}_k(P;\mathrm{Ad}_\rho)=0$ for $k\neq 1.$

The main result of this section is the following

\begin{proposition}\label{P} Let $\rho:\pi_1(P)\to\mathrm{PSL}(2;\mathbb C)$ be a $\boldsymbol\gamma$-regular representation, and 
for $i\in\{1,2,3\}$ let $ u_i$ be up to sign the logarithmic holonomy of $\gamma_i$ in $\rho.$
Then
$$\mathrm{Tor}(P, \mathbf h_P; \mathrm{Ad}_\rho)=\pm\frac{1}{16\sinh\frac{ u_1}{2}\sinh\frac{ u_2}{2}\sinh\frac{ u_3}{2}}.$$
\end{proposition}

\begin{remark} Proposition \ref{P} could be proved from other results in the literature, which is related to volume forms on the character variety, see \cite{W, BL, HP}. We include a different proof here for the readers' convenience and as a warm up for the computations in the next section. 
\end{remark}

To prove Proposition \ref{P}, we need the following Lemma.

\begin{lemma}\label{LI} 
The set of $\boldsymbol \gamma$-regular characters contains a Zariski-open subset $\mathrm Z(P)$ of $\mathrm X(P)$  consisting of the characters $[\rho]$ satisfying the following two conditions:
\begin{enumerate}[(1)]
\item $$\det[\mathbf I_1,\mathbf I_2,\mathbf I_3]\neq 0,$$ and
 \item  
 $$\det\Big[\mathbf I_1-\mathrm{Ad}_\rho ([\gamma_3]^{-1})^T\cdot\mathbf I_1,\ \mathbf I_2-\mathrm{Ad}_\rho ([\gamma_3]^{-1})^T\cdot\mathbf I_2,\ \mathbf I_3-\mathrm{Ad}_\rho ([\gamma_2])^T\cdot\mathbf I_3\Big]\neq 0.$$
\end{enumerate}
\end{lemma}

\begin{proof} Let us compute the determinants in conditions (1) and (2) first. We will do the computations for the holonomy representation of a hyperbolic metric on $P$ with cone singularities around $\gamma_1,$ $\gamma_2$ and $\gamma_3$ first. Then by analyticity the computation extends to the other representations. We would like to mention here that these computations  will also hold the key to the proof of Proposition \ref{P}.

Let $P$ be a hyperbolic $2$-sphere with three cone singularities $p_1,$ $p_2$ and $p_3$ removed. We let  the cone angles at $p_1,$ $p_2$ and $p_3$ respectively be $2\alpha_1,$ $2\alpha_2$ and $2\alpha_3$ all of which are less than $2\pi,$ and let $\gamma_1,$ $\gamma_2$ and $\gamma_3$  respectively  be the simple loops around $p_1,$ $p_2$ and $p_3.$ In this case we have $ u_i=\pm 2\mathbf i\alpha_i.$

Now $\rho([\gamma_i])^T$ is a rotation, hence has an eigenvector $\mathbf v_i^+$ of eigenvalue $e^{\mathbf i\alpha_i}$ and an eigenvector $\mathbf v_i^-$ of eigenvalue $e^{-\mathbf i\alpha_i}.$ If
\begin{equation*}
\mathbf v_i^+=\left[
\begin{array}{c}
a \\
 b\\
 \end{array}\right]
 \quad \text{and} \quad
 \mathbf v_i^-=\left[
\begin{array}{c}
c \\
d\\
 \end{array}\right],
\end{equation*}
then an invariant vector of $\mathrm {Ad}_\rho([\gamma_i])^T$ has the form
\begin{equation}\label{invformula}
\left[
\begin{array}{c}
ac \\
 ad+bc \\
 bd\\
 \end{array}\right].
\end{equation}
Indeed, if we identify $[a,b]^T$ with the polynomial $aX+bY$ and $[c,d]^T$ with $cX+dY,$ then the polynomial $(aX+bY)(cX+dY)=acX^2+(ad+bc)XY+bdY^2$ is invariant under $\big(\mathrm{Sym}^2\circ \widetilde\rho([\gamma_i])\big)^T.$ 

Now $P$ is the double of a hyperbolic triangle $T$ with cone angles $\alpha_1,$ $\alpha_2$ and $\alpha_3.$ For $i=1,2,3,$ let $e_i$ be the edge of $T$ opposite to $p_i$ and let $s_i$ be its lengths. To compute its holonomy representation $\rho,$ we isometrically embedded $T$ into $\mathbb H^3$ as follows. As in Figure \ref{triangle}, we place $p_1$ at $(0,0,1),$ the edge $e_2$ in the $xz$-plane and $T$ in the unit  hemisphere centered at $(0,0,0)$ such that the $y$-coordinate of $p_2$ is negative. 
This could always be done by replacing $T$ by its mirror image  if necessary.

\begin{figure}[htbp]
\centering
\includegraphics[scale=0.5]{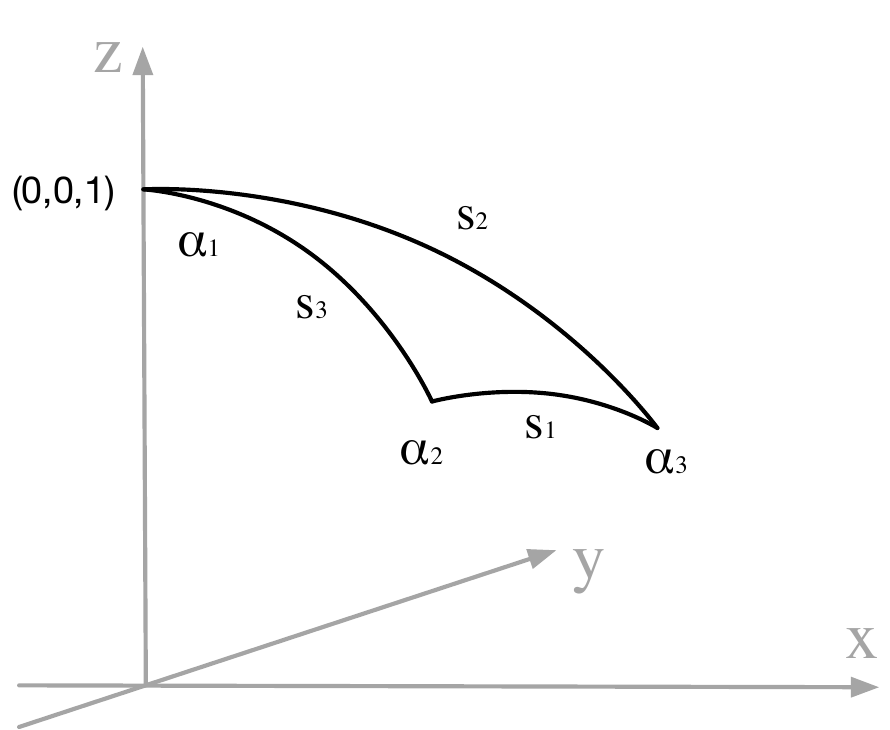}
\caption{}
\label{triangle} 
\end{figure}

To simplify the notation, we for any $z\in \mathbb C$ let 
\begin{equation*}
D_z=\left[
\begin{array}{cc}
e^{\frac{z}{2}}& 0\\
 0 &  e^{-\frac{z}{2}}\\
 \end{array}\right]
\end{equation*}
and for $i=1,2,3,$ let 
\begin{equation*}
S_i=\left[
\begin{array}{cc}
\cosh{\frac{s_i}{2}}& \sinh{\frac{s_i}{2}} \\
\sinh{\frac{s_i}{2}} & \cosh{\frac{s_i}{2}}\\
 \end{array}\right].
\end{equation*}

Suppose for each $i,$ $\gamma_i$ goes counterclockwise around $p_i.$ Then by conjugating the tangent framings at $p_2$ and $p_3$ back to $p_1=(0,0,1)$ and conjugating the tangent vectors of the axes of the rotations to $\frac{\partial}{\partial z},$ we have 
\begin{equation*}
\begin{split}
\rho([\gamma_1])&=\pm D_{2\mathbf i\alpha_1},\\
\rho([\gamma_2])&=\pm D_{\mathbf i\alpha_1}^{-1}S_3D_{2\mathbf i\alpha_2}S_3^{-1}D_{i\alpha_1}=\pm S_2D_{-\mathbf i\alpha_3}^{-1}S_1^{-1}D_{2\mathbf i\alpha_2}S_1D_{-\mathbf i\alpha_3}S_2^{-1},\\
\rho([\gamma_3])&=\pm S_2D_{2\mathbf i\alpha_3}S_2^{-1}=\pm D_{\mathbf i\alpha_1}^{-1}S_3D_{\mathbf i\alpha_2}^{-1}S_1^{-1}D_{2\mathbf i\alpha_3}  S_1D_{\mathbf i\alpha_2}S_3^{-1}D_{\mathbf i\alpha_1}.
\end{split}
\end{equation*}
Here we compute $\rho([\gamma_2])$ and $\rho([\gamma_3])$ in two ways for the purpose of computing different things later. Since both $D_z$ and $S_i$ are symmetric matrices, we have
\begin{equation}\label{holo}
\begin{split}
\rho([\gamma_1])^T&=\pm D_{2\mathbf i\alpha_1},\\
\rho([\gamma_2])^T&=\pm D_{\mathbf i\alpha_1}S_3^{-1}D_{2\mathbf i\alpha_2}S_3D_{\mathbf i\alpha_1}^{-1}=\pm S_2^{-1}D_{-\mathbf i\alpha_3}S_1D_{2\mathbf i\alpha_2}S_1^{-1}D_{-\mathbf i\alpha_3}^{-1}S_2,\\
\rho([\gamma_3])^T&=\pm S_2^{-1}D_{2\mathbf i\alpha_3}S_2=\pm D_{\mathbf i\alpha_1}S_3^{-1}D_{\mathbf i\alpha_2}S_1D_{2\mathbf i\alpha_3}  S_1^{-1}D_{\mathbf i\alpha_2}^{-1}S_3D_{\mathbf i\alpha_1}^{-1}.
\end{split}
\end{equation}
Then 
\begin{equation}\label{eigen}
\begin{split}
[\mathbf v_1^+,\mathbf v_1^-]&=I,\\
[\mathbf v_2^+,\mathbf v_2^-]&=D_{\mathbf i\alpha_1}S_3^{-1}=S_2^{-1}D_{-\mathbf i\alpha_3}S_1,\\
[\mathbf v_3^+,\mathbf v_3^-]&=S_2^{-1}=D_{\mathbf i\alpha_1}S_3^{-1}D_{\mathbf i\alpha_2}S_1.
\end{split}
\end{equation}
Using the first half of the second and third equations of (\ref{eigen}), (\ref{invformula}) and a direct computation, we have
\begin{equation}\label{linv}
\mathbf I_1=\left[
\begin{array}{c}
0 \\
1 \\
0\\
 \end{array}\right],
 \quad
 \mathbf I_2=\left[
\begin{array}{c}
-\frac{1}{2}e^{\mathbf i\alpha_1}\sinh s_3\\
\cosh s_3 \\
-\frac{1}{2}e^{-\mathbf i\alpha_1}\sinh s_3\\
 \end{array}\right]
 \quad\text{and}\quad
  \mathbf I_3=\left[
\begin{array}{c}
-\frac{1}{2}\sinh s_2\\
\cosh s_2 \\
-\frac{1}{2}\sinh s_2\\
 \end{array}\right].
\end{equation}
Since $\kappa(\mathbf I_i,\mathbf I_i)=1$ for $i\in\{1,2,3\},$ they are the canonical invariant vectors. 
Therefore
\begin{equation}\label{det}
\det[\mathbf I_1,\mathbf I_2,\mathbf I_3]=-\frac{\mathbf i}{2}\sin\alpha_1\sinh s_2\sinh s_3.
\end{equation}
This computes the determinant in (1) for the holonomy representation of a hyperbolic structure with cone singularities.

To compute the determinant in (2), we need the following auxiliary computations. For real numbers $x$ and $y,$ we let 
\begin{equation*}
X=\left[
\begin{array}{cc}
\cosh{\frac{x}{2}}& \sinh{\frac{x}{2}} \\
\sinh{\frac{x}{2}} & \cosh{\frac{x}{2}}\\
 \end{array}\right]
 \quad\text{and}\quad
Y=\left[
\begin{array}{cc}
\cosh{\frac{y}{2}}& \sinh{\frac{y}{2}} \\
\sinh{\frac{y}{2}} & \cosh{\frac{y}{2}}\\
 \end{array}\right],
\end{equation*}
and for a complex number $z$ let $D_z$ be as before. We let 
\begin{equation*}
\mathbf I_{xy}^z=\left[
\begin{array}{c}
ac \\
 ad+bc \\
 bd\\
 \end{array}\right]\quad\text{if}\quad
X^{-1}D_zY=\pm\left[
\begin{array}{cc}
a & c\\
 b & d\\
 \end{array}\right];
\end{equation*}
and let
\begin{equation*}
\mathbf I_{wxy}^z=\left[
\begin{array}{c}
ac \\
 ad+bc \\
 bd\\
 \end{array}\right]\quad\text{if}\quad
D_wX^{-1}D_zY=\pm\left[
\begin{array}{cc}
a & c\\
 b & d\\
 \end{array}\right].
\end{equation*}
We notice that $\mathbf I_{xy}^z$ and $\mathbf I_{wxy}^z$ are independent of the signs $\pm$ in front of the $2\times 2$ matrices. Then using the hyperbolic trigonometric identities $\cosh z-\cosh z'=2\sinh\frac{z+z'}{2}\sinh\frac{z-z'}{2}$ and $\sinh z-\sinh z'=2\cosh\frac{z+z'}{2}\sinh\frac{z-z'}{2}$ for any complex numbers $z$ and $z'$ and a direct computation, we have 
\begin{equation}\label{difference1}
\mathbf I_{xy}^z-\mathbf I_{xy}^{z'}=\sinh y\sinh\frac{z-z'}{2}\left[
\begin{array}{c}
\sinh\frac{z+z'}{2}\cosh x+\cosh\frac{z+z'}{2} \\
-2\sinh\frac{z+z'}{2}\sinh x \\
\sinh\frac{z+z'}{2}\cosh x-\cosh\frac{z+z'}{2} \\
 \end{array}\right],
\end{equation}
and
\begin{equation}\label{difference2}
\mathbf I_{wxy}^z-\mathbf I_{wxy}^{z'}=\sinh y\sinh\frac{z-z'}{2}\left[
\begin{array}{c}
e^w\big(\sinh\frac{z+z'}{2}\cosh x+\cosh\frac{z+z'}{2}\big)\\
-2\sinh\frac{z+z'}{2}\sinh x \\
e^{-w}\big(\sinh\frac{z+z'}{2}\cosh x-\cosh\frac{z+z'}{2}\big) \\
 \end{array}\right].
\end{equation}

By the first half of the third equation of (\ref{holo}), we have 
\begin{equation}\label{3}
\rho([\gamma_3]^{-1})^T=\pm S_2^{-1}D_{-2\mathbf i\alpha_3}S_2.
\end{equation}
By (\ref{3}) and the first equation of (\ref{eigen}), we have 
$$[\mathbf v_1^+,\mathbf v_1^-]=I=S_2^{-1}D_0S_2$$
and
$$\rho([\gamma_3]^{-1})^T\cdot [\mathbf v_1^+,\mathbf v_1^-]=\pm S_2^{-1}D_{-2\mathbf i\alpha_3}S_2.$$
Therefore, by (\ref{difference1})
\begin{equation*}
\begin{split}
\mathbf I_1-\mathrm{Ad}_\rho ([\gamma_3]^{-1})^T\cdot\mathbf I_1&=\mathbf I_{s_2s_2}^0-\mathbf I_{s_2s_2}^{-2\mathbf i\alpha_3}\\
&=\mathbf i\sinh s_2\sin \alpha_3\left[
\begin{array}{c}
-\mathbf i\sin\alpha_3\cosh s_2+\cos\alpha_3 \\
2\mathbf i\sin\alpha_3\sinh s_2\\
-\mathbf i\sin\alpha_3\cosh s_2-\cos\alpha_3 \\
 \end{array}\right].\\
\end{split}
\end{equation*}
By (\ref{3}) and the second half of the second equation of (\ref{eigen}), we have 
$$[\mathbf v_2^+,\mathbf v_2^-]=S_2^{-1}D_{-\mathbf i\alpha_3}S_1$$
and
$$\rho([\gamma_3]^{-1})^T\cdot [\mathbf v_2^+,\mathbf v_2^-]=\pm S_2^{-1}D_{-3\mathbf i\alpha_3}S_1.$$
Therefore, by (\ref{difference1})
\begin{equation*}
\begin{split}
\mathbf I_2-\mathrm{Ad}_\rho ([\gamma_3]^{-1})^T\cdot\mathbf I_2&=\mathbf I_{s_2s_1}^{-\mathbf i\alpha_3}-\mathbf I_{s_2s_1}^{-3\mathbf i\alpha_3}\\
&=\mathbf i\sinh s_1\sin \alpha_3\left[
\begin{array}{c}
-\mathbf i\sin(2\alpha_3)\cosh s_2+\cos(2\alpha_3) \\
2\mathbf i\sin(2\alpha_3)\sinh s_2\\
-\mathbf i\sin(2\alpha_3)\cosh s_2-\cos(2\alpha_3) \\
 \end{array}\right].\\
\end{split}
\end{equation*}
By the first half of the second equation of (\ref{holo}) and the second half of the third equation of (\ref{eigen}), we have
$$[\mathbf v_3^+,\mathbf v_3^-]=D_{\mathbf i\alpha_1}S_3^{-1}D_{\mathbf i\alpha_2}S_1$$
and
$$\rho([\gamma_2])^T\cdot [\mathbf v_3^+,\mathbf v_3^-]=\pm D_{\mathbf i\alpha_1}S_3^{-1}D_{3\mathbf i\alpha_2}S_1.$$
Therefore, by (\ref{difference2})
\begin{equation*}
\begin{split}
\mathbf I_3-\mathrm{Ad}_\rho ([\gamma_2])^T\cdot\mathbf I_3&=\mathbf I_{(\mathbf i\alpha_1)s_3s_1}^{\mathbf i\alpha_2}-\mathbf I_{(\mathbf i\alpha_1)s_3s_1}^{3\mathbf i\alpha_2}\\
&=-\mathbf i\sinh s_1\sin \alpha_2\left[
\begin{array}{c}
e^{\mathbf i\alpha_1}\big(\mathbf i\sin(2\alpha_2)\cosh s_3+\cos(2\alpha_2)\big) \\
-2\mathbf i\sin(2\alpha_2)\sinh s_3\\
e^{-\mathbf i\alpha_1}\big(\mathbf i\sin(2\alpha_2)\cosh s_3-\cos(2\alpha_2) \big)\\
 \end{array}\right].\\
\end{split}
\end{equation*}
We observe that the matrix
\begin{equation*}
\begin{split}
&\left[
\begin{array}{ccc}
-\mathbf i\sin\alpha_3\cosh s_2+\cos\alpha_3  & -\mathbf i\sin(2\alpha_3)\cosh s_2+\cos(2\alpha_3) & e^{\mathbf i\alpha_1}\big(\mathbf i\sin(2\alpha_2)\cosh s_3+\cos(2\alpha_2)\big) \\
2\mathbf i\sin\alpha_3\sinh s_2 & 2\mathbf i\sin(2\alpha_3)\sinh s_2 & -2\mathbf i\sin(2\alpha_2)\sinh s_3\\
-\mathbf i\sin\alpha_3\cosh s_2-\cos\alpha_3 & -\mathbf i\sin(2\alpha_3)\cosh s_2-\cos(2\alpha_3) & e^{-\mathbf i\alpha_1}\big(\mathbf i\sin(2\alpha_2)\cosh s_3-\cos(2\alpha_2) \big)\\
 \end{array}\right]\\
=&\left[\begin{array}{ccc}
-\mathbf i & 0 & 1\\
0 & 2\mathbf i & 0\\
-\mathbf i & 0 &-1\\
 \end{array}\right]\cdot\left[
\begin{array}{ccc}
\sin\alpha_3\cosh s_2 & \sin(2\alpha_3)\cosh s_2 & -\cos\alpha_1\sin(2\alpha_2)\cosh s_3-\sin\alpha_1\cos(2\alpha_2) \\
\sin\alpha_3\sinh s_2 & \sin(2\alpha_3)\sinh s_2 & -\sin(2\alpha_2)\sinh s_3\\
\cos\alpha_3 & \cos(2\alpha_3) & -\sin\alpha_1\sin(2\alpha_2)\cosh s_3+\cos\alpha_1\cos(2\alpha_2)\\
 \end{array}\right].\\
 \end{split}
\end{equation*}
Denoting  the second matrix above by $M$, we have
\begin{equation*}
\begin{split}
&\det\Big[\mathbf I_1-\mathrm{Ad}_\rho ([\gamma_3]^{-1})^T\cdot\mathbf I_1,\ \mathbf I_2-\mathrm{Ad}_\rho ([\gamma_3]^{-1})^T\cdot\mathbf I_2,\ \mathbf I_3-\mathrm{Ad}_\rho ([\gamma_2])^T\cdot\mathbf I_3\Big]\\
=&-4\mathbf i\sinh^2s_1\sinh s_2\sin\alpha_2\sin^2\alpha_3\det M.\\
 \end{split}
\end{equation*}
Computing the cofactors of $M,$ we have 
$M_{13}=-\sin\alpha_3\sinh s_2,$
$M_{23}=\sin\alpha_3\cosh s_2$ and 
$M_{33}=0.$ Then 
\begin{equation*}
\begin{split}
&\det M\\
=&-\sin\alpha_3\sinh s_2\big(-\cos\alpha_1\sin(2\alpha_2)\cosh s_3-\sin\alpha_1\cos(2\alpha_2) \big)-\sin\alpha_3\cosh s_2\sin(2\alpha_2)\sinh s_3\\
=&-\sinh s_2\sin\alpha_1\sin \alpha_3,
\end{split}
\end{equation*}
where the last equality comes from the use of the hyperbolic Law of Sine that $\sinh s_3 = \frac{\sinh s_2\sin \alpha_3}{\sin \alpha_2}$ to get a common factor $\sinh s_2,$ then the use of the hyperbolic Law of Cosine that $\cosh s_2 =\frac{\cos \alpha_2+\cos\alpha_1\cos\alpha_3}{\sin \alpha_1\sin\alpha_3}$ and $\cosh s_3 =\frac{\cos \alpha_3+\cos\alpha_1\cos\alpha_2}{\sin \alpha_1\sin\alpha_2}$ to change the quantity into a function of the angles $\alpha_1,$ $\alpha_2$ and $\alpha_3$ only, finally the use of the double angle formulas to $\sin(2\alpha_2),$ $\cos(2\alpha_2)$ and $\sin(2\alpha_3)$ to get a function of $\{\sin\alpha_k\}$ and $\{\cos\alpha_k\}$ only then followed by a simplification.

Therefore, 
\begin{equation}\label{det5}
\begin{split}
&\det\Big[\mathbf I_1-\mathrm{Ad}_\rho ([\gamma_3]^{-1})^T\cdot\mathbf I_1,\ \mathbf I_2-\mathrm{Ad}_\rho ([\gamma_3]^{-1})^T\cdot\mathbf I_2,\ \mathbf I_3-\mathrm{Ad}_\rho ([\gamma_2])^T\cdot\mathbf I_3\Big]\\
=&4\mathbf i\sin\alpha_1\sin\alpha_2\sin^3\alpha_3\sinh^2s_1\sinh ^2s_2.
 \end{split}
\end{equation}
This computes the determinant in (2) for the holonomy representation of a hyperbolic structure.

For the other characters in $\mathrm X(P),$ we observe that for the holonomy representation $\rho$ of a hyperbolic structure  with cone angles $(2\alpha_1,2\alpha_2,2\alpha_3),$ for any lifting $\widetilde\rho:\pi_1(P)\to\mathrm{SL}(2;\mathbb C)$  of $\rho,$ we have
$$\mathrm {Tr}\widetilde\rho([\gamma_i])=\pm2\cos\alpha_i$$
for $i\in\{1,2,3\}.$ Then by the trigonometry identity and   the hyperbolic Law of Cosine, we have
\begin{equation}\label{rational}
\sinh s_i=\pm\sqrt{\frac{-\det\mathrm G_{\boldsymbol \alpha}}{(1-\cos^2\alpha_j)(1-\cos^2\alpha_k)}}
\end{equation}
for $\{i,j,k\}=\{1,2,3\},$ where 
$$\mathrm G_{\boldsymbol \alpha}=\left[
\begin{array}{cccc}
1 & -\cos\alpha_{3} & -\cos\alpha_{2} \\
-\cos\alpha_{3}& 1 &-\cos\alpha_{1 }\\
-\cos\alpha_{2} & -\cos\alpha_{1} & 1  \\
 \end{array}\right]$$
is the Gram matrix in the angles of the hyperbolic triangle with angles $(\alpha_1,\alpha_2,\alpha_3).$ 
Therefore, the square of $\sinh s_i$ is a rational functions in $(\mathrm {Tr}\widetilde\rho([\gamma_1]),\mathrm {Tr}\widetilde\rho([\gamma_2]),\mathrm {Tr}\widetilde\rho([\gamma_3])).$ Since $\mathrm X(P)$ is an irreducible algebraic variety, by the analyticity of the functions on the right hand sides, (\ref{det}) and (\ref{det5}) hold for 
the other characters $[\rho]$ in $\mathrm X(P).$  

Since the square of the determinants in conditions (1) and (2) are rational functions in the coordinates $(\mathrm {Tr}\widetilde\rho([\gamma_1]),\mathrm {Tr}\widetilde\rho([\gamma_2]),\mathrm {Tr}\widetilde\rho([\gamma_3])),$ the lifting of those characters form a Zariski-open subset of $\mathbb C^3,$ and hence those characters themselves form a Zariski-open subset of $\mathrm X(P).$
\\

Next we show that the representations satisfying (1) and (2) are $\boldsymbol\gamma$-regular.  We will compute the homologies of $P$ using its spine $\Gamma,$ which is the $1$-dimensional $CW$ complex on the left of Figure \ref{spine1} consisting of two $0$-cells $x_1$ and $x_1$ and three $1$-cells $a_1,$ $a_2$ and $a_3$ all of which are oriented from $x_1$ to $x_2.$ 

\begin{figure}[htbp]
\centering
\includegraphics[scale=0.45]{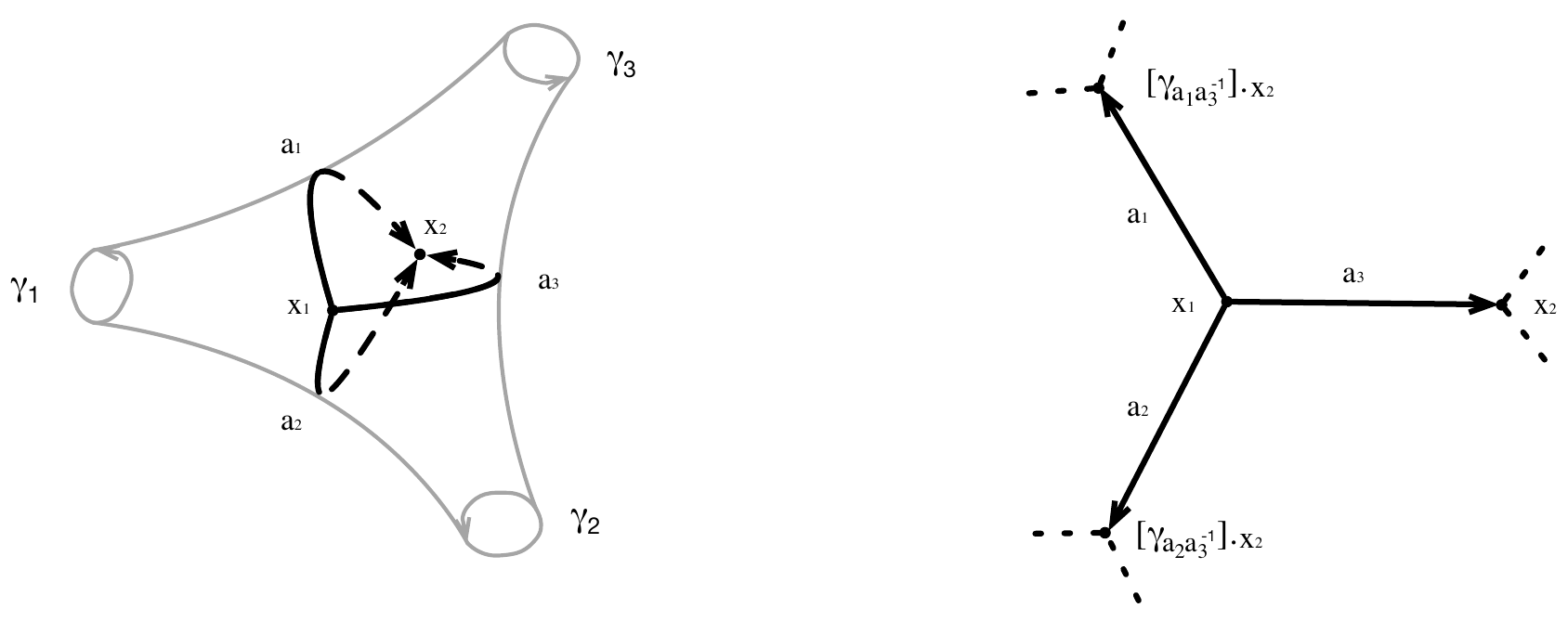}
\caption{}
\label{spine1} 
\end{figure}

Let $\{\mathbf e_1,\mathbf e_2,\mathbf e_3\}$ be the standard basis of $\mathbb C^3$ and let the choice of representatives $x_1,x_2,a_1,a_2$ and $a_3$ in the universal covering of $\Gamma$ be as drawn on the right of Figure \ref{spine1}.  Then $\mathrm C_0(P;\mathrm{Ad}_\rho)\cong \mathbb C^6$ with a natural basis $\{\mathbf e_i\otimes x_k\}$ for $i\in\{1,2,3\}$ and $k\in\{1,2\};$ $\mathrm C_1(P;\mathrm{Ad}_\rho)\cong \mathbb C^9$ with a natural basis $\{\mathbf e_i\otimes a_k\}$ for $i,k\in\{1,2,3\};$ and $\mathrm C_k(P;\mathrm{Ad}_\rho)=0$ for $k\neq 0$ or $1.$

We choose $x_1$ to be the base point of the fundamental group; and for $\{j,k\}\subset \{1,2,3\},$ let $\gamma_{a_ja_k^{-1}}$ be the curve starting from $x_1$ traveling along $a_j$ to $x_2$ then along $-a_k$ back to $x_1.$ In this way, we have $[\gamma_{a_1a_2^{-1}}]=[\gamma_1],$ $[\gamma_{a_2a_3^{-1}}]=[\gamma_2]$ and $[\gamma_{a_1a_3^{-1}}]=[\gamma_3]^{-1}.$

By condition (1),  the vectors $\mathbf I_1\otimes (a_1-a_2),$ $\mathbf I_2\otimes (a_2-a_3)$ and $\mathbf I_3\otimes (a_1-a_3)$ are linearly independent in $\mathrm C_1(P; \mathrm{Ad}_\rho).$ Next we show that they lie in the kernel of $\partial : \mathrm C_1(P;\mathrm{Ad}_\rho)\to \mathrm C_0(P;\mathrm{Ad}_\rho).$ Indeed, for the image of $\mathbf I_1\otimes (a_1-a_2),$ we have
\begin{equation*}
\begin{split}
\partial (\mathbf I_1\otimes (a_1-a_2))&=\mathbf I_1\otimes \partial (a_1 -a_2)\\
&=\mathbf I_1\otimes \Big((x_1-[\gamma_{a_1a_3^{-1}}]\cdot x_2)-(x_1-[\gamma_{a_2a_3^{-1}}]\cdot x_2)\Big)\\
&=\mathbf I_1\otimes \Big([\gamma_2]\cdot x_2- [\gamma_3]^{-1}\cdot x_2\Big)\\
&=\Big(\mathrm{Ad}_\rho ([\gamma_2])^T \cdot\mathbf I_1-\mathrm{Ad}_\rho ([\gamma_3]^{-1})^T \cdot \mathbf I_1\Big)\otimes x_2\\
&=\Big(\mathrm{Ad}_\rho ([\gamma_2])^T \mathrm{Ad}_\rho ([\gamma_1])^T\cdot \mathbf I_1-\mathrm{Ad}_\rho ([\gamma_3]^{-1})^T \cdot \mathbf I_1\Big)\otimes x_2=0,
\end{split}
\end{equation*}
where the penultimate equality comes from $\mathrm{Ad}_\rho ([\gamma_1])^T\cdot \mathbf I_1=\mathbf I_1$ and the last equation comes from $\gamma_1\cdot\gamma_2=\gamma_3^{-1}.$
For the image of the other two vectors, we have 
\begin{equation*}
\begin{split}
\partial (\mathbf I_2\otimes (a_2-a_3))&=\mathbf I_2\otimes \partial (a_2 -a_3)\\
&=\mathbf I_2\otimes \Big((x_1-[\gamma_{a_2a_3^{-1}}]\cdot x_2)-(x_1- x_2)\Big)\\
&=\mathbf I_2\otimes \Big( x_2- [\gamma_2]\cdot x_2\Big)\\
&=\Big(\mathbf I_2-\mathrm{Ad}_\rho ([\gamma_2])^T \cdot \mathbf I_2\Big)\otimes x_2=0,
\end{split}
\end{equation*}
and 
\begin{equation*}
\begin{split}
\partial (\mathbf I_3\otimes (a_1-a_3))&=\mathbf I_3\otimes \partial (a_1 -a_3)\\
&=\mathbf I_3\otimes \Big((x_1-[\gamma_{a_1a_3^{-1}}]\cdot x_2)-(x_1- x_2)\Big)\\
&=\mathbf I_3\otimes \Big( x_2- [\gamma_3]^{-1}\cdot x_2\Big)\\
&=\Big(\mathbf I_3-\mathrm{Ad}_\rho ([\gamma_3]^{-1})^T \cdot \mathbf I_3\Big)\otimes x_2=0,
\end{split}
\end{equation*}
where the last equalities respectively come from $\mathrm{Ad}_\rho ([\gamma_2])^T\cdot \mathbf I_2=\mathbf I_2$ and $\mathrm{Ad}_\rho ([\gamma_3]^{-1})^T\cdot \mathbf I_3=\mathbf I_3.$ Therefore, $\mathbf I_1\otimes (a_1-a_2),$ $\mathbf I_2\otimes (a_2-a_3)$ and $\mathbf I_3\otimes (a_1-a_3)$ represent three linearly independent elements $\mathbf I_1\otimes [\gamma_1],$ $\mathbf I_2\otimes[\gamma_2]$ and $\mathbf I_3\otimes [\gamma_3]$ in $\mathrm H_1(P;\mathrm{Ad}_\rho).$ Later we will prove that they also span, and hence form a basis of $\mathrm H_1(P;\mathrm{Ad}_\rho).$

Now we claim that $\{\mathbf I_1\otimes (a_1-a_2), \mathbf I_2\otimes (a_2-a_3), \mathbf I_3\otimes (a_1-a_3)\}$ joint with six vectors $\{\mathbf I_1\otimes a_3, \mathbf I_2\otimes a_3, \mathbf I_3\otimes a_3, \mathbf I_1\otimes a_1,  \mathbf I_2\otimes a_1, \mathbf I_3\otimes a_2\}$ form a basis of 
$\mathrm C_1(P;\mathrm{Ad}_\rho).$ Indeed, in the natural basis $\{\mathbf e_i\otimes a_k\},$ $i,k\in\{1,2,3\},$ the $9\times 9$ matrix consisting of these vectors as the columns is obtained from the one consisting of $\{\mathbf I_j\otimes a_k\},$ $j,k\in\{1,2,3\},$ as the columns by a sequence of elementary column operations of type I, III, and II with a factor $-1.$ The latter matrix is a block matrix with three $3\times3$ blocks $[\mathbf I_1,\mathbf I_2,\mathbf I_3]$ on the diagonal and $0's$ elsewhere, hence by condition (1) is non-singular and has determinant  $\det[\mathbf I_1,\mathbf I_2,\mathbf I_3]^3.$ As a consequence, the former matrix is also non-singular and up to sign has determinant  $\det[\mathbf I_1,\mathbf I_2,\mathbf I_3]^3.$

In the next step, we will study the image of the six vectors $\{\mathbf I_1\otimes a_3, \mathbf I_2\otimes a_3, \mathbf I_3\otimes a_3, \mathbf I_1\otimes a_1,  \mathbf I_2\otimes a_1, \mathbf I_3\otimes a_2\}$ under the boundary map $\partial,$ and show that they span $\mathrm C_0(P;\mathrm{Ad}_\rho).$ We have for $j=1,2,3,$
$$\partial( \mathbf I_j\otimes a_3)= \mathbf I_j\otimes \partial a_3=\mathbf I_j\otimes (x_1-x_2)=\mathbf I_j\otimes x_1-\mathbf I_j\otimes x_2;$$
for $k=1,2,$
$$\partial (\mathbf I_k\otimes a_1)= \mathbf I_k\otimes \partial a_1=\mathbf I_k\otimes (x_1-[\gamma_{a_1a_3^{-1}}]\cdot x_2)=\mathbf I_k\otimes x_1-\Big(\mathrm{Ad}_\rho ([\gamma_3]^{-1})^T\cdot\mathbf I_k\Big)\otimes x_2;$$
and
$$\partial (\mathbf I_3\otimes a_2)= \mathbf I_3\otimes \partial a_2=\mathbf I_3\otimes (x_1-[\gamma_{a_2a_3^{-1}}]\cdot x_2)=\mathbf I_3\otimes x_1-\Big(\mathrm{Ad}_\rho ([\gamma_2])^T\cdot\mathbf I_3\Big)\otimes x_2.$$
Therefore, in the natural basis $\{\mathbf e_i\otimes x_k\},$ $i\in\{1,2,3\},$ $k\in\{1,2\},$ the $6\times 6$ matrix consisting of $\{\partial(\mathbf I_1\otimes a_3), \partial(\mathbf I_2\otimes a_3), \partial(\mathbf I_3\otimes a_3), \partial(\mathbf I_1\otimes a_1),  \partial(\mathbf I_2\otimes a_1), \partial(\mathbf I_3\otimes a_2)\}$ as the columns has four $3\times 3$ blocks, where on the top it has two copies of $[\mathbf I_1,\mathbf I_2,\mathbf I_3],$ on the bottom left is has $[-\mathbf I_1,-\mathbf I_2,-\mathbf I_3]$ and on the bottom right 
$$\Big[-\mathrm{Ad}_\rho ([\gamma_3]^{-1})^T\cdot\mathbf I_1,\ -\mathrm{Ad}_\rho ([\gamma_3]^{-1})^T\cdot\mathbf I_2,\ -\mathrm{Ad}_\rho ([\gamma_2])^T\cdot\mathbf I_3\Big].$$
This matrix is row equivalent to (by adding the top blocks to the bottom) the one with two copies of $[\mathbf I_1,\mathbf I_2,\mathbf I_3]$ on the top, $0's$ on the bottom left and
$$\Big[\mathbf I_1-\mathrm{Ad}_\rho ([\gamma_3]^{-1})^T\cdot\mathbf I_1,\ \mathbf I_2-\mathrm{Ad}_\rho ([\gamma_3]^{-1})^T\cdot\mathbf I_2,\ \mathbf I_3-\mathrm{Ad}_\rho ([\gamma_2])^T\cdot\mathbf I_3\Big]$$
on the bottom right. The determinant of both of the $6\times 6$ matrices is
$$\det[\mathbf I_1,\mathbf I_2,\mathbf I_3]\cdot\det\Big[\mathbf I_1-\mathrm{Ad}_\rho ([\gamma_3]^{-1})^T\cdot\mathbf I_1,\ \mathbf I_2-\mathrm{Ad}_\rho ([\gamma_3]^{-1})^T\cdot\mathbf I_2,\ \mathbf I_3-\mathrm{Ad}_\rho ([\gamma_2])^T\cdot\mathbf I_3\Big].$$
By conditions (1) and (2),  the product above is nonzero and hence $\{\partial(\mathbf I_1\otimes a_3), \partial(\mathbf I_2\otimes a_3), \partial(\mathbf I_3\otimes a_3), \partial(\mathbf I_1\otimes a_1),  \partial(\mathbf I_2\otimes a_1), \partial(\mathbf I_3\otimes a_2)\}$ span $\mathrm C_0(P;\mathrm{Ad}_\rho).$ 
This implies that $\mathrm H_0(P;\mathrm{Ad}_\rho)=0.$ Since there are no cells of dimension higher than or equal to $2,$ $\mathrm H_k(P;\mathrm{Ad}_\rho)=0$ for $k\geqslant 2.$ 

Finally,  since $\{\partial(\mathbf I_1\otimes a_3), \partial(\mathbf I_2\otimes a_3), \partial(\mathbf I_3\otimes a_3), \partial(\mathbf I_1\otimes a_1),  \partial(\mathbf I_2\otimes a_1), \partial(\mathbf I_3\otimes a_2)\}$ span $\mathrm C_0(P;\mathrm{Ad}_\rho)\cong \mathbb C^6,$ by dimension counting the kernel of $\partial: \mathrm C_1(P;\mathrm{Ad}_\rho) \to \mathrm C_0(P;\mathrm{Ad}_\rho)$ has dimension at most $3.$ Hence $\mathbf I_1\otimes (a_1-a_2),$ $\mathbf I_2\otimes (a_2-a_3)$ and $\mathbf I_3\otimes (a_1-a_3)$ span the kernel of $\partial.$ This shows that the elements they represent  $\mathbf h_P=\{\mathbf I_1\otimes [\gamma_1], \mathbf I_2\otimes[\gamma_2], \mathbf I_3\otimes [\gamma_3]\}$ form a basis of $\mathrm H_1(P;\mathrm{Ad}_\rho),$ and $\mathrm H_1(P;\mathrm{Ad}_\rho)\cong \mathbb C^3.$ This completes the proof. 
\end{proof}

\begin{proof} [Proof of Proposition \ref{P}]

Since the adjoint twisted Reidemeister torsion is invariant under subdivisions, elementary expansions and elementary collapses of CW-complexes  by \cite{M,T}, we can do the computation using the spine $\Gamma$ of $P$ as on the left of Figure \ref{spine1}.

The  adjoint twisted Reidemeistor torsion equals, up to sign, the determinant of the $9\times 9$ matrix consisting of $\{\mathbf I_1\otimes (a_1-a_2), \mathbf I_2\otimes (a_2-a_3), \mathbf I_3\otimes (a_1-a_3), \mathbf I_1\otimes a_3, \mathbf I_2\otimes a_3, \mathbf I_3\otimes a_3, \mathbf I_1\otimes a_1,  \mathbf I_2\otimes a_1, \mathbf I_3\otimes a_2\}$ as the columns divided by the determinant of the $6\times 6$ matrix consisting of $\{\partial(\mathbf I_1\otimes a_3), \partial(\mathbf I_2\otimes a_3), \partial(\mathbf I_3\otimes a_3), \partial(\mathbf I_1\otimes a_1),  \partial(\mathbf I_2\otimes a_1), \partial(\mathbf I_3\otimes a_2)\}$ as the columns. 

By (\ref{det}) and (\ref{det5}), for the holonomy representation of a hyperbolic structure we have
\begin{equation*}
\begin{split}
&\mathrm{Tor}(P, \mathbf h_P;\mathrm{Ad}_\rho)\\
=&\pm\frac{\det[\mathbf I_1,\mathbf I_2,\mathbf I_3]\cdot \det[\mathbf I_1,\mathbf I_2,\mathbf I_3]\cdot \det[\mathbf I_1,\mathbf I_2,\mathbf I_3]}{\det[\mathbf I_1,\mathbf I_2,\mathbf I_3]\cdot\det\Big[\mathbf I_1-\mathrm{Ad}_\rho ([\gamma_3]^{-1})^T\cdot\mathbf I_1,\ \mathbf I_2-\mathrm{Ad}_\rho ([\gamma_3]^{-1})^T\cdot\mathbf I_2,\ \mathbf I_3-\mathrm{Ad}_\rho ([\gamma_2])^T\cdot\mathbf I_3\Big]}\\
=&\pm\frac{\mathbf i}{16\sin\alpha_1\sin\alpha_2\sin\alpha_3}\\
=&\pm\frac{1}{16\sinh\frac{ u_1}{2}\sinh\frac{ u_2}{2}\sinh\frac{ u_3}{2}},
\end{split}
\end{equation*}
where the second equality comes from the hyperbolic Law of Sine that $\frac{\sinh s_3}{\sinh s_1}=\frac{\sin \alpha_3}{\sin\alpha_1}.$ 

Finally, by Lemma \ref{LI} and the analyticity, the result holds for all $\boldsymbol\gamma$-regular characters in $\mathrm X(P).$
\end{proof}


\section{Adjoint twisted Reidemeister torsion of the $D$-blocks}\label{TD}

Let $\Delta$ be a truncated  tetrahedron with triangles of truncation $T_1,$ $T_2,$ $T_3,$ $T_4$ and hexagonal faces $H_1,$ $H_2,$ $H_3,$ $H_4$ such that $T_k$ is opposite to $H_k.$ Recall that an \emph{edge} is the intersection of two hexagonal faces; and we call the intersection of a triangle of truncation and a hexagonal face a \emph{short edge}. Let $D$ be the union of $\Delta$ with its mirror image via the identity map between the four hexagonal faces  $H_1,\dots, H_4$ and with the six edges emoved. This is a \emph{D-block} as defined in \cite{C} and recalled in Section \ref{fsl}. For $\{j,k\}\subset \{1,2,3,4\}$ we let $e_{jk}$ be the edge adjacent to $H_j$ and $H_k.$ For $\{j,k\}\subset\{1,2,3,4\},$ let $\gamma_{jk}$ be a simple  loop around $e_{jk}.$

The fundamental group  $\pi_1(D)$  is a free group of rank $3$ generated by $[\gamma_{12}],$ $[\gamma_{13}]$ and $[\gamma_{14}].$ 
By \cite{G1}, the $\mathrm{SL}(2;\mathbb C)$-character variety of $D$ is homeomorphic to a hypersurface in $\mathbb C^7$ parametrized by the traces  of the image of $[\gamma_{12}],$ $[\gamma_{13}],$ $[\gamma_{14}],$ $[\gamma_{12}\cdot\gamma_{13}],$ $[\gamma_{12}\cdot\gamma_{14}],$ $[\gamma_{13}\cdot\gamma_{14}]$ and $[\gamma_{12}\cdot\gamma_{13}\cdot\gamma_{14}],$ which is a double branched cover of $\mathbb C^6$ parametrized by the first six components. A representation $\widetilde \rho:\pi_1(D)\to\mathrm{SL}(2;\mathbb C)$ is not in the branch locus if and only if 
$$f_D\big(\mathrm {Tr}\widetilde\rho([\gamma_{12}]),
\mathrm {Tr}\widetilde\rho([\gamma_{13}]),
\mathrm {Tr}\widetilde\rho([\gamma_{14}]),\mathrm {Tr}\widetilde\rho([\gamma_{12}\cdot\gamma_{13}]),
\mathrm {Tr}\widetilde\rho([\gamma_{12}\cdot\gamma_{14}]),
\mathrm {Tr}\widetilde\rho([\gamma_{13}\cdot\gamma_{14}])\big)\neq 0,$$
where $f_D$ is the polynomial
\begin{equation*}
\begin{split}
f_D(t_1,t_2,t_3&,t_{12},t_{13},t_{23})=\big(t_{12}t_3+t_{13}t_2+t_{23}t_1-t_1t_2t_3\big)^2\\
&-4\big(t_1^2+t_2^2+t_3^2+t_{12}^2+t_{13}^2+t_{23}^2-t_1t_2t_{12}-t_1t_3t_{13}-t_2t_3t_{23}+t_{12}t_{13}t_{23}-4\big).
\end{split}
\end{equation*}

The \emph{logarithmic holonomies} of $(\gamma_{12},\gamma_{13},\gamma_{14},\gamma_{23},\gamma_{24},\gamma_{34})$ in $\widetilde \rho$ are up to sign the complex numbers $(u_{12}, u_{13},u_{14},u_{23},u_{24},u_{34})$ satisfying 
\begin{equation*}
\begin{split}
\big(\mathrm {Tr}\widetilde\rho([\gamma_{12}]),
&\mathrm {Tr}\widetilde\rho([\gamma_{13}]),
\mathrm {Tr}\widetilde\rho([\gamma_{14}]),\mathrm {Tr}\widetilde\rho([\gamma_{23}]),
\mathrm {Tr}\widetilde\rho([\gamma_{24}]),
\mathrm {Tr}\widetilde\rho([\gamma_{34}])\big)\\
=&\Big(-2\cosh\frac{u_{12}}{2},-2\cosh\frac{u_{13}}{2},-2\cosh\frac{u_{14}}{2},-2\cosh\frac{u_{23}}{2},-2\cosh\frac{u_{24}}{2},-2\cosh\frac{u_{34}}{2}\Big).
\end{split}
\end{equation*}
In this way, if $D$ is with the hyperbolic structure obtained by doubling the regular ideal octahedron, $\rho_0:\pi_1(D)\to\mathrm{PSL}(2;\mathbb C)$ is the holonomy representation of this hyperbolic structure on $D$ and $\widetilde \rho_0:\pi_1(D)\to\mathrm{SL}(2;\mathbb C)$ is the lifting of $\rho_0$ with
\begin{equation*}
\begin{split}\big(\mathrm {Tr}\widetilde\rho_0([\gamma_{12}]),
\mathrm {Tr}\widetilde\rho_0([\gamma_{13}]),
\mathrm {Tr}\widetilde\rho_0([\gamma_{14}]),\mathrm {Tr}\widetilde\rho_0([\gamma_{23}]),
\mathrm {Tr}\widetilde\rho_0([\gamma_{24}]),
&\mathrm {Tr}\widetilde\rho_0([\gamma_{34}])\big)\\
=&(-2,-2,-2,-2,-2,-2),
\end{split}
\end{equation*}
 then the logarithmic holonomies of $(\gamma_{12},\dots,\gamma_{34})$  in $\widetilde\rho_0$ are $(0,\dots,0).$ We notice that the complete hyperbolic structure on a fundamental shadow link complement is obtained by gluing such $D$-blocks together by isometries along the faces. Therefore, this hyperbolic structure can be considered as ``the complete hyperbolic structure'' on $D.$   

The \emph{Gram matrix} of a representation $\widetilde \rho:\pi_1(D)\to\mathrm{SL}(2;\mathbb C)$ is  the value of the Gram matrix function $\mathbb G$ defined in Definition \ref{GMF} at $\big(\frac{u_{12}}{2},\dots, \frac{u_{34}}{2}\big),$ ie, 
$$\mathbb G=\mathbb G\Big(\frac{u_{12}}{2},\frac{u_{13}}{2},\frac{u_{14}}{2},\frac{u_{23}}{2},\frac{u_{24}}{2},\frac{u_{34}}{2}\Big)=\left[
\begin{array}{cccc}
1 & -\cosh \frac{u_{12}}{2} & -\cosh  \frac{u_{13}}{2} &-\cosh  \frac{u_{14}}{2}\\
-\cosh  \frac{u_{12}}{2}& 1 &-\cosh  \frac{u_{23}}{2} & -\cosh  \frac{u_{24}}{2}\\
-\cosh  \frac{u_{13}}{2} & -\cosh  \frac{u_{23}}{2} & 1 & -\cosh  \frac{u_{34}}{2} \\
-\cosh  \frac{u_{14}}{2} & -\cosh  \frac{u_{24}}{2} & -\cosh  \frac{u_{34}}{2}  & 1 \\
 \end{array}\right].$$
 By the trace identity of the matrices in $\mathrm{SL}(2;\mathbb C),$ for $\{j,k\}\subset\{2,3,4\},$ 
 $$\mathrm {Tr}\widetilde\rho([\gamma_{jk}])=\mathrm {Tr}\widetilde\rho([\gamma_{1j}\cdot\gamma_{1k}^{-1})]=\mathrm {Tr}\widetilde\rho([\gamma_{1j}])\mathrm {Tr}\widetilde\rho([\gamma_{1k}])-\mathrm {Tr}\widetilde\rho([\gamma_{1j}\cdot\gamma_{1k}]).$$
Then by a direct computation, we have
$$
f_D\big(\mathrm {Tr}\widetilde\rho([\gamma_{12}]),
\mathrm {Tr}\widetilde\rho([\gamma_{13}]),
\mathrm {Tr}\widetilde\rho([\gamma_{14}]),\mathrm {Tr}\widetilde\rho([\gamma_{12}\cdot\gamma_{13}]),
\mathrm {Tr}\widetilde\rho([\gamma_{12}\cdot\gamma_{14}]),
\mathrm {Tr}\widetilde\rho([\gamma_{13}\cdot\gamma_{14}])\big)\\
=16\det \mathbb G,$$
  and $\widetilde\rho$ is not in the branch locus of the double branched cover of the $\mathrm{SL}(2;\mathbb C)$-character variety of $D$ over $\mathbb C^6$  if and only if $\det\mathbb G\neq 0.$

Since $\pi_1(D)$ is a free group, every $\mathrm{PSL}(2;\mathbb C)$-representation of it lifts to  $\mathrm{SL}(2;\mathbb C)$-representation Hence the  $\mathrm{SL}(2;\mathbb C)$-character variety of $D$ is a branched cover of the $\mathrm{PSL}(2;\mathbb C)$-character variety of  $D,$ and the latter is an irreducible algebraic variety.  For a representation $\rho:\pi_1(D)\to\mathrm{PSL}(2;\mathbb C),$ we defined the logarithmic holonomies $(u_{12},\dots,u_{34})$ and the Gram matrix $\mathbb G$ of $\rho$  as those of a lifting $\widetilde \rho:\pi_1(D)\to\mathrm{SL}(2;\mathbb C)$ of $\rho.$ Notice that the logarithmic holonomies depend on the choice of the liftings of $\rho,$ and a different lifting will change $\mathbb G$ by multiplying some rows and the corresponding columns by $-1$ at the same time, which does not change its determinant. Therefore, the \emph{determinant of the Gram matrix} $\det\mathbb G$ is independent of the choice of the liftings, and  is a well defined quantity of $\rho.$

Let $\rho:\pi_1(D)\to \mathrm{PSL}(2;\mathbb C)$ be a  representation, and let $\mathrm{Ad}_\rho:\pi_1(D)\to \mathrm{SL}(3;\mathbb C)$ be its adjoint representation. In addition, we assume for each $\{j,k\}\subset\{1,2,3,4\}$ that $\rho([\gamma_{jk}])\neq\pm I.$ Then in the case that $\rho([\gamma_{jk}])$ is not a parabolic element, we let $\mathbf I_{jk}$ be up to sign the unique invariant vector of $\mathrm {Ad}_\rho([\gamma_{jk}])^T$ with $\kappa(\mathbf I_{jk},\mathbf I_{jk})=1,$ where $\kappa$ is the Killing bilinear form on $\mathfrak{psl}(2;\mathbb C)$ defined in (\ref{Killing}).


\begin{definition}\label{Dreg} Let $\boldsymbol\gamma=(\gamma_{12},\gamma_{13},\gamma_{14},\gamma_{23},\gamma_{24},\gamma_{34}).$ An irreducible representation $\rho:\pi_1(D)\to\mathrm{PSL}(2;\mathbb C)$ is \emph{$\boldsymbol\gamma$-regular} if 
 $$\mathbf h_D=\big\{\mathbf I_{jk}\otimes [\gamma_{jk}]\}\ \big|\ \{j,k\}\subset \{1,2,3,4\}\big\}$$
is a basis of $\mathrm{H}_1(D;\mathrm{Ad}_\rho),$ where $[\gamma_{jk}]$ is the homology class of $\gamma_{jk}$ in $\mathrm H_1(D;\mathbb Z).$
\end{definition}

Let $\mathrm X(D)$ be the $\mathrm{PSL}(2;\mathbb C)$-character variety of $D.$ A character $[\rho]\in\mathrm X(D)$ is  \emph{$\boldsymbol \gamma$-regular } if $\rho$ is a $\boldsymbol \gamma$-regular representation. Since $\pi_1(D)$ is a free group, an Euler characteristic counting argument shows that if $[\rho]$ is  $\boldsymbol\gamma$-regular, then $\mathrm{H}_k(D;\mathrm{Ad}_\rho)=0$ for $k\neq 1.$

The main result of this section is the following Proposition \ref{D}. 

\begin{proposition}\label{D} Let $\rho:\pi_1(D)\to\mathrm{PSL}(2;\mathbb C)$ be a $\boldsymbol\gamma$-regular representation, and 
for $\{j,k\}\subset\{1,2,3,4\}$ let $ u_{jk}$ be up to sign the \emph{logarithmic holonomy} of $\gamma_{jk}$ in $\rho.$ Then
$$\mathrm{Tor}(D, \mathbf h_D; \mathrm{Ad}_\rho)=\pm\frac{\sqrt{\det\mathbb G\Big(\frac{u_{12}}{2}, \frac{u_{13}}{2}, \frac{u_{14}}{2}, \frac{u_{23}}{2}, \frac{u_{24}}{2},\frac{u_{34}}{2}\Big)}}{32\sinh\frac{ u_{12}}{2}\sinh\frac{ u_{13}}{2}\sinh\frac{ u_{14}}{2}\sinh\frac{ u_{23}}{2}\sinh\frac{ u_{24}}{2}\sinh\frac{ u_{34}}{2}}.$$
\end{proposition}


To prove Proposition \ref{D}, we need the following Lemma.

\begin{lemma} \label{LMD}
The set of $\boldsymbol \gamma$-regular characters contains a Zariski-open subset $\mathrm Z(D)$ of $\mathrm X(D)$ consisting of the characters $[\rho]$ satisfying the following two conditions:
\begin{enumerate}[(1)]
\item 
\begin{equation*}
\begin{split}
\det[\mathbf I_{12},\mathbf I_{13},\mathbf I_{14}]&\neq 0,\\
\det[\mathbf I_{12},\mathbf I_{23},\mathbf I_{24}]&\neq 0,\\
\det[\mathbf I_{13},\mathbf I_{23},\mathbf I_{34}]&\neq 0,\\
\det[\mathbf I_{14},\mathbf I_{24},\mathbf I_{34}]&\neq0,\\
\end{split}
\end{equation*}
and
\item \begin{equation*}
\det\Big[\mathbf I_{12}-\mathrm{Ad}_\rho ([\gamma_{13}])^T\cdot\mathbf I_{12},\ \mathbf I_{14}-\mathrm{Ad}_\rho ([\gamma_{13}])^T\cdot\mathbf I_{14},\mathbf I_{24}-\mathrm{Ad}_\rho ([\gamma_{23}])^T\cdot\mathbf I_{24}\Big]\neq 0.
\end{equation*}
\end{enumerate}
\end{lemma}

\begin{proof} Let us compute the determinants in conditions (1) and (2) first.  Similar to the proof of Lemma \ref{LI} we will do the computations for the holonomy representation of a hyperbolic metric on $D$ with cone singularities around the edges $e_{jk}$'s first. Then by analyticity the computation extends to the other representations. 
 
Now let $\Delta$ be a truncated hyperideal tetrahedron and let  $D$ be the union of $\Delta$ with its mirror image via the identity map between the four hexagonal faces  $H_1,\dots, H_4$ and with the six edges $e_{12},\dots, e_{34}$ removed. This is a hyperbolic D-block defined in Section \ref{dhp}. For $\{j,k\}\subset \{1,2,3,4\}$ we let $l_{jk}$ and $\alpha_{jk}$ respectively be the length of and the dihedral angle at the edge $e_{jk}.$ We let $s_{jk}$ be the length of the short edge adjacent to $T_j$ and $H_k,$ and notice that   $s_{jk}$ and $s_{kj}$ are the lengths of different short edges.

Let $\rho:\pi_1(D)\to \mathrm{PSL}(2;\mathbb C)$ be the holonomy representation of $D$ and let $\mathrm{Ad}_\rho:\pi_1(D)\to \mathrm{SL}(3;\mathbb C)$ be its adjoint representation. For $\{j,k\}\subset\{1,2,3,4\},$ let $\gamma_{jk}$ be a simple  loop around $e_{jk}.$ Since $\rho(\gamma_{jk})$ is  an elliptic element in $\mathrm{PSL}(2;\mathbb C)$ which is not the identity matrix, $\mathrm {Ad}_\rho([\gamma_{jk}])^T$ has up to sign the canonical invariant vector  $\mathbf I_{jk}.$

To compute the holonomy representation $\rho$ of $D,$ we isometrically embedded $\Delta$ into $\mathbb H^3$ as follows. As in Figure \ref{tetra1}, we place the intersection point of $H_1,$ $H_2$ and $T_4$ at $(0,0,1),$ the edge $e_{12}$ along the $z$-axis such that the intersection point of $H_1,$ $H_2$ and $T_3$ is above $(0,0,1),$ the hexagonal face $H_1$ in the $xz$-plane and $T_4$ in the unit hemisphere centered at $(0,0,0)$ such that the $y$-coordinate of all the interior points of $\Delta$ are negative. This could always be done by using the mirror image of $\Delta$ if necessary.

\begin{figure}[htbp]
\centering
\includegraphics[scale=0.5]{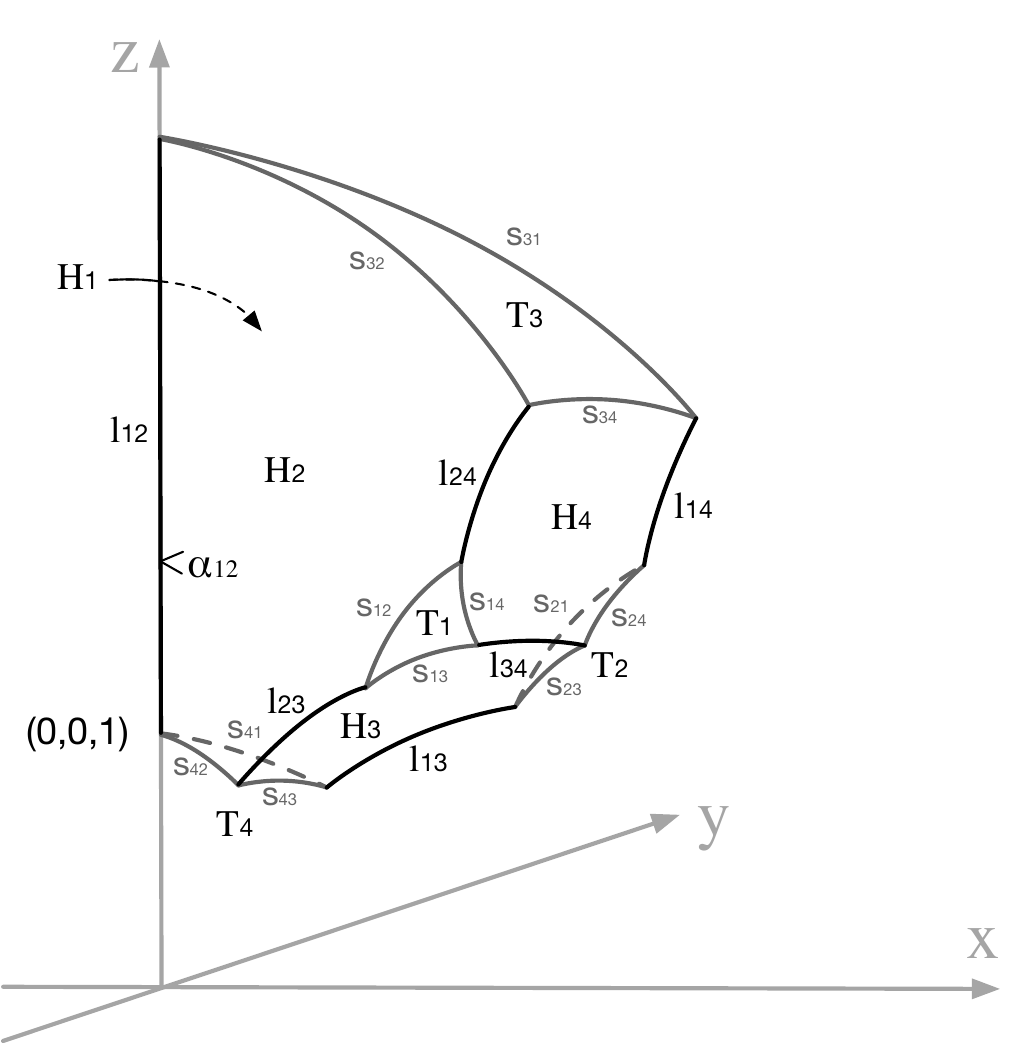}
\caption{}
\label{tetra1} 
\end{figure}

For any complex number $z$ let
\begin{equation*}
D_z=\left[
\begin{array}{cc}
e^{\frac{z}{2}}& 0\\
 0 &  e^{-\frac{z}{2}}\\
 \end{array}\right],
\end{equation*}
and for $\{j,k\}\subset\{1,2,3,4\}$ let 
\begin{equation*}
S_{jk}=\left[
\begin{array}{cc}
\cosh{\frac{s_{jk}}{2}}& \sinh{\frac{s_{jk}}{2}} \\
\sinh{\frac{s_{jk}}{2}} & \cosh{\frac{s_{jk}}{2}}\\
 \end{array}\right].
\end{equation*}

Suppose  $\gamma_{12},$ $\gamma_{14},$ $\gamma_{23}$ and $\gamma_{24}$ go counterclockwise and  $\gamma_{13}$  goes clockwise around the corresponding edges observed from the perspective above $T_3.$ By conjugating the tangent framings back to $p_1=(0,0,1)$ and conjugating the tangent vectors of the axes of the rotations to $\frac{\partial}{\partial z},$ we have 
\begin{equation*}
\begin{split}
\rho([\gamma_{12}])&=\pm D_{2\mathbf i\alpha_{12}},\\
\rho([\gamma_{13}])&=\pm S_{41}D_{-2\mathbf i\alpha_{13}}S_{41}^{-1},\\
\rho([\gamma_{14}])&=\pm D_{l_{12}}S_{31}D_{2\mathbf i\alpha_{14}}S_{31}^{-1}D_{l_{12}}^{-1}=\pm S_{41}D_{l_{13}}S_{21}^{-1}D_{-2\mathbf i\alpha_{14}}S_{21}D_{l_{13}}^{-1}S_{41}^{-1},\\
\rho([\gamma_{23}])&=\pm D_{\mathbf i\alpha_{12}}^{-1}S_{42}D_{2\mathbf i\alpha_{23}}S_{42}^{-1}D_{\mathbf i\alpha_{12}},\\
\rho([\gamma_{24}])&=\pm D_{\mathbf i\alpha_{12}}^{-1}S_{42}D_{l_{23}}S_{12}^{-1}D_{-2\mathbf i\alpha_{24}}S_{12}D_{l_{23}}^{-1}S_{42}^{-1}D_{\mathbf i\alpha_{12}}.
\end{split}
\end{equation*}
Here we write $\rho([\gamma_{14}])$  in two ways for the purpose of computing different things later. Since both $D_z$ and $S_{jk}$ are symmetric matrices, we have
\begin{equation}\label{holo2}
\begin{split}
\rho([\gamma_{12}])^T&=\pm D_{2\mathbf i\alpha_{12}},\\
\rho([\gamma_{13}])^T&=\pm S_{41}^{-1}D_{-2\mathbf i\alpha_{13}}S_{41},\\
\rho([\gamma_{14}])^T&=\pm D_{l_{12}}^{-1}S_{31}^{-1}D_{2\mathbf i\alpha_{14}}S_{31}D_{l_{12}}=\pm S_{41}^{-1}D_{l_{13}}^{-1}S_{21}D_{-2\mathbf i\alpha_{14}}S_{21}^{-1}D_{l_{13}}S_{41},\\ 
\rho([\gamma_{23}])^T&=\pm D_{\mathbf i\alpha_{12}}S_{42}^{-1}D_{2\mathbf i\alpha_{23}}S_{42}D_{\mathbf i\alpha_{12}}^{-1},\\
\rho([\gamma_{24}])^T&=\pm D_{\mathbf i\alpha_{12}}S_{42}^{-1}D_{l_{23}}^{-1}S_{12}D_{-2\mathbf i\alpha_{24}}S_{12}^{-1}D_{l_{23}}S_{42}D_{\mathbf i\alpha_{12}}^{-1}.
\end{split}
\end{equation}
Since $\rho([\gamma_{jk}])^T$ is a rotation of angle $2\alpha_{jk},$ it has an eigenvector  $\mathbf v_{jk}^+$ with eigenvalue $e^{\mathbf i\alpha_{jk}}$ and an eigenvector  $\mathbf v_{jk}^-$ with eigenvalue $e^{-\mathbf i\alpha_{jk}}.$ By (\ref{holo2}) we have
\begin{equation}\label{eigen2}
\begin{split}
[\mathbf v_{12}^+,\mathbf v_{12}^-]&=I,\\
[\mathbf v_{13}^+,\mathbf v_{13}^-]&=S_{41}^{-1},\\
[\mathbf v_{14}^+,\mathbf v_{14}^-]&=D_{l_{12}}^{-1}S_{31}^{-1}=S_{41}^{-1}D_{l_{13}}^{-1}S_{21},\\
[\mathbf v_{24}^+,\mathbf v_{24}^-]&= D_{\mathbf i\alpha_{12}}S_{42}^{-1}D_{l_{23}}^{-1}S_{12},
\end{split}
\end{equation}
and by (\ref{invformula}), the first half of the third equation of (\ref{eigen2}) and a direct computation we have
\begin{equation}\label{inv2}
\mathbf I_{12}=\left[
\begin{array}{c}
0 \\
1 \\
0\\
 \end{array}\right],
 \quad
   \mathbf I_{13}=\left[
\begin{array}{c}
-\frac{1}{2}\sinh s_{41}\\
\cosh s_{41} \\
-\frac{1}{2}\sinh s_{41}\\
 \end{array}\right]
 \quad\text{and}\quad
 \mathbf I_{14}=\left[
\begin{array}{c}
-\frac{1}{2}e^{-l_{12}}\sinh s_{31}\\
\cosh s_{31} \\
-\frac{1}{2}e^{l_{12}}\sinh s_{31}\\
 \end{array}\right].
\end{equation}
Since $\kappa(\mathbf I_{12},\mathbf I_{12})=\kappa(\mathbf I_{13},\mathbf I_{13})=\kappa(\mathbf I_{14},\mathbf I_{14})=1,$ they are the canonical invariant vectors. 
Therefore, 
\begin{equation}\label{det2}
\det[\mathbf I_{12},\mathbf I_{13},\mathbf I_{14}]=-\frac{1}{2}\sinh l_{12}\sinh s_{31}\sinh s_{41}.
\end{equation}

Here we notice that by the hyperbolic Law of Sine for $H_1,$ the quantity $\sinh l_{12}\sinh s_{31}\sinh s_{41}$ remains the same if we choose any edge and two adjacent short edges of $H_1,$ hence is an intrinsic quantity of $H_1.$

For any $i\neq 1,$ applying an orientation preserving isometry $\phi_i$ of $\mathbb H^3$ we can place $H_i$ in $\mathbb H^3$ in the same way as $H_1;$ and the invariant vector $\mathbf I_{ij},$ $i\neq j,$ will be changed by $\mathrm{Ad}_{\phi_i},$ which is a matrix in $\mathrm{SL}(3;\mathbb C).$  Therefore, following the same computation as we did for (\ref{det2}), we have
\begin{equation}\label{det3}
\begin{split}
\det[\mathbf I_{12},\mathbf I_{23},\mathbf I_{24}]&=\frac{1}{2}\sinh l_{12}\sinh s_{32}\sinh s_{42},\\
\det[\mathbf I_{13},\mathbf I_{23},\mathbf I_{34}]&=-\frac{1}{2}\sinh l_{13}\sinh s_{23}\sinh s_{43},\\
\det[\mathbf I_{14},\mathbf I_{24},\mathbf I_{34}]&=\frac{1}{2}\sinh l_{14}\sinh s_{24}\sinh s_{34}.\\
\end{split}
\end{equation}
This computes the determinants in (1) for the holonomy representation of a hyperbolic $D$-block.

To compute the determinant in (2), by the second equation of (\ref{holo2}) and the first equation of (\ref{eigen2}), we have 
$$[\mathbf v_{12}^+,\mathbf v_{12}^-]=I=S_{41}^{-1}D_0S_{41}$$
and
$$\rho([\gamma_{13}])^T\cdot [\mathbf v_{12}^+,\mathbf v_{12}^-]=\pm S_{41}^{-1}D_{-2\mathbf i\alpha_{13}}S_{41}.$$
Therefore, by (\ref{difference1}) and the notation therein, 
\begin{equation*}
\begin{split}
\mathbf I_{12}-\mathrm{Ad}_\rho ([\gamma_{13}])^T\cdot\mathbf I_{12}&=\mathbf I_{s_{41}s_{41}}^0-\mathbf I_{s_{41}s_{41}}^{-2\mathbf i\alpha_{13}}\\
&=\mathbf i\sinh s_{41}\sin \alpha_{13}\left[
\begin{array}{c}
-\mathbf i\sin\alpha_{13}\cosh s_{41}+\cos\alpha_{13} \\
2\mathbf i\sin\alpha_{13}\sinh s_{41}\\
-\mathbf i\sin\alpha_{13}\cosh s_{41}-\cos\alpha_{13} \\
 \end{array}\right].\\
\end{split}
\end{equation*}
By the second equation of (\ref{holo2}) again  and the second half of the third equation of (\ref{eigen2}), we have 
$$[\mathbf v_{14}^+,\mathbf v_{14}^-]=S_{41}^{-1}D_{-l_{13}}S_{21}$$
and
$$\rho([\gamma_{13}])^T\cdot [\mathbf v_{14}^+,\mathbf v_{14}^-]=\pm S_{41}^{-1}D_{-l_{13}-2\mathbf i\alpha_{13}}S_{21}.$$
Therefore, by (\ref{difference1})
\begin{equation*}
\begin{split}
\mathbf I_{14}-\mathrm{Ad}_\rho ([\gamma_{13}])^T\cdot\mathbf I_{14}&=\mathbf I_{s_{41}s_{21}}^{-l_{13}}-\mathbf I_{s_{41}s_{21}}^{-l_{13}-2\mathbf i\alpha_{13}}\\
&=\mathbf i\sinh s_{21}\sin \alpha_{13}\left[
\begin{array}{c}
-\sinh(l_{13}+\mathbf i\alpha_{13})\cosh s_{41}+\cosh(l_{13}+\mathbf i\alpha_{13}) \\
2\sinh(l_{13}+\mathbf i\alpha_{13})\sinh s_{41}\\
-\sinh(l_{13}+\mathbf i\alpha_{13})\cosh s_{41}-\cosh(l_{13}+\mathbf i\alpha_{13}) \\
 \end{array}\right].\\
\end{split}
\end{equation*}
Finally, by the fourth equation of (\ref{holo2}) and (\ref{eigen2}), we have
$$[\mathbf v_{24}^+,\mathbf v_{24}^-]= D_{\mathbf i\alpha_{12}}S_{42}^{-1}D_{-l_{23}}S_{12}$$
and
$$\rho([\gamma_{23}])^T\cdot[\mathbf v_{24}^+,\mathbf v_{24}^-]=\pm D_{\mathbf i\alpha_{12}}S_{42}^{-1}D_{-l_{23}+2\mathbf i\alpha_{23}}S_{12}.
$$
Therefore, by (\ref{difference2})
\begin{equation*}
\begin{split}
\mathbf I_{24}-\mathrm{Ad}_\rho ([\gamma_{23}])^T\cdot\mathbf I_{24}&=\mathbf I_{(\mathbf i\alpha_{12})s_{42}s_{12}}^{-l_{23}}-\mathbf I_{(\mathbf i\alpha_{12})s_{42}s_{12}}^{-l_{23}+2\mathbf i\alpha_{23}}\\
&=-\mathbf i\sinh s_{12}\sin \alpha_{23}\left[
\begin{array}{c}
e^{\mathbf i\alpha_{12}}\big(-\sinh(l_{23}-\mathbf i\alpha_{23})\cosh s_{42}+\cosh(l_{23}-\mathbf i\alpha_{23})\big) \\
2\sinh(l_{23}-\mathbf i\alpha_{23})\sinh s_{42}\\
e^{-\mathbf i\alpha_{12}}\big(-\sinh(l_{23}-\mathbf i\alpha_{23})\cosh s_{42}-\cosh(l_{23}-\mathbf i\alpha_{23})\big)\\
 \end{array}\right].\\
\end{split}
\end{equation*}
Putting all together, we have
\begin{equation*}
\begin{split}
&\det\Big[\mathbf I_{12}-\mathrm{Ad}_\rho ([\gamma_{13}])^T\cdot\mathbf I_{12},\ \mathbf I_{14}-\mathrm{Ad}_\rho ([\gamma_{13}])^T\cdot\mathbf I_{14},\mathbf I_{24}-\mathrm{Ad}_\rho ([\gamma_{23}])^T\cdot\mathbf I_{24}\Big]\\
=&\mathbf i\sin^2\alpha_{13}\sin\alpha_{23}\sinh s_{12}\sinh s_{21}\sinh s_{41}\cdot\det \left[\begin{array}{ccc}
-1 & 0 & 1\\
0 & 2 & 0\\
-1 & 0 &-1\\
 \end{array}\right]\cdot \det M,
\end{split}
\end{equation*}
where $M$ is the following matrix
\begin{equation*}
\begin{split}
&\left[
\begin{array}{ccc}
\mathbf i\sin\alpha_{13}\cosh s_{41}& \sinh(l_{13}+\mathbf i\alpha_{13})\cosh s_{41}&  \cos\alpha_{12}\sinh(l_{23}-\mathbf i\alpha_{23})\cosh s_{42}-\mathbf i\sin\alpha_{12}\cosh(l_{23}-\mathbf i\alpha_{23}) \\
\mathbf i\sin\alpha_{13}\sinh s_{41}& \sinh(l_{13}+\mathbf i\alpha_{13})\sinh s_{41} & \sinh(l_{23}-\mathbf i\alpha_{23})\sinh s_{42}\\
\cos\alpha_{13} & \cosh(l_{13}+\mathbf i\alpha_{13})  &-\mathbf i\sin\alpha_{12}\sinh(l_{23}-\mathbf i\alpha_{23})\cosh s_{42}+\cos\alpha_{12}\cosh(l_{23}-\mathbf i\alpha_{23}) \\
 \end{array}\right].\\
\end{split}
\end{equation*}
Computing the cofactors of $M$using the hyperbolic angle sum formula, we have 
$M_{13}=-\sinh l_{13}\sinh s_{41},$
$M_{23}=\sinh l_{13}\cosh s_{41}$ and 
$M_{33}=0.$ Then 
\begin{equation*}
\begin{split}
\det M=&-\sinh l_{13}\sinh s_{41}\Big( \cos\alpha_{12}\sinh(l_{23}-\mathbf i\alpha_{23})\cosh s_{42}-\mathbf i\sin\alpha_{12}\cosh(l_{23}-\mathbf i\alpha_{23})\Big)\\
&+\sinh l_{13}\cosh s_{41}\sinh(l_{23}-\mathbf i\alpha_{23})\sinh s_{42}\\
=&\frac{\sin\alpha_{12}\sinh l_{13}\sinh l_{23}\sinh s_{42}}{\sin \alpha_{13}},
\end{split}
\end{equation*}
where the last equality comes from the use of the hyperbolic Law of Sine that $\sinh s_{41} = \frac{\sinh s_{42}\sin \alpha_{23}}{\sin \alpha_{13}}$ to get a common factor $\sinh s_{41},$ the use of the hyperbolic Law of Cosine in $T_4$ to write $\cosh s_{41}$ and $\cosh s_{42}$ into trig-functions of the angles $\alpha_{12},$ $\alpha_{13}$ and $\alpha_{23}$ and the use of the angle sum formula to expand $\sinh(l_{23}-i\alpha_{23})$ and $\cosh(l_{23}-i\alpha_{23})$ into trig- and hyperbolic trig-fuctions of $\alpha_{23}$ and $l_{23}.$ Then after a final simplification, the imaginary part vanishes and the real part becomes the quantity above.

Therefore, 
\begin{equation}\label{det4}
\begin{split}
&\det\Big[\mathbf I_{12}-\mathrm{Ad}_\rho ([\gamma_{13}])^T\cdot\mathbf I_{12},\ \mathbf I_{14}-\mathrm{Ad}_\rho ([\gamma_{13}])^T\cdot\mathbf I_{14},\mathbf I_{24}-\mathrm{Ad}_\rho ([\gamma_{23}])^T\cdot\mathbf I_{24}\Big]\\
=&4\mathbf i\sin\alpha_{12}\sin\alpha_{13}\sin\alpha_{23}(\sinh l_{13}\sinh s_{21}\sinh s_{41})(\sinh l_{23}\sinh s_{42}\sinh s_{12}).
 \end{split}
\end{equation} 
This computes the determinant in (2) for the holonomy representation of a hyperbolic $D$-block.

For the other characters in $\mathrm X(D),$ we observe that for the holonomy representation $\rho$ of a hyperbolic D-block  with cone angles $(2\alpha_{12},\dots,2\alpha_{34}),$ for any lifting $\widetilde\rho:\pi_1(D)\to\mathrm{SL}(2;\mathbb C)$ of $\rho,$ we have
$$\mathrm {Tr}\widetilde\rho([\gamma_{jk}])=\pm2\cos\alpha_{jk}$$
for $\{j,k\}\subset\{1,2,3,4\}.$ Notice that $l_{ij},$ $s_{ki}$  and $s_{li}$ are the lengths of an edge and the two adjacent short edges around the face $H_i,$ and all the determines in conditions (1) and (2) have factors products of the form $\sinh l_{ij}\sinh s_{ki}\sinh s_{li}.$ We claim that 
\begin{equation}\label{symmetry}
\sinh l_{ij}\sinh s_{ki}\sinh s_{li}=\pm\sqrt{\frac{-\det \mathrm G_{\boldsymbol\alpha}}{(1-\cos^2\alpha_{ij})(1-\cos^2\alpha_{ik})(1-\cos^2\alpha_{il})}}
\end{equation}
for $\{i,j,k,l\}\subset\{1,2,3,4\},$ where $\mathrm G_{\boldsymbol\alpha}$ is the Gram matrix in the dihedral angles of the truncated hyperideal tetrahedron $\Delta$ recalled in Section \ref{Gram}. As a consequence, the square of $\cosh l_{ij}\sinh s_{ki}\sinh s_{li}$ is a rational function in $(\mathrm {Tr}\widetilde\rho([\gamma_{12}]),\dots,\mathrm {Tr}\widetilde\rho([\gamma_{34}])).$ Indeed, to see (\ref{symmetry}), using the hyperbolic Law of Cosine to the face $H_i,$ we have
\begin{equation*}
\begin{split}
\sinh^2 l_{ij}\sinh^2 s_{ki}\sinh^2 s_{li}=&\bigg(\Big(\frac{\cosh s_{ji}+\cosh s_{ki}\cosh s_{li}}{\sinh s_{ki}\sinh s_{li}}\Big)^2-1\bigg)\sinh^2 s_{ki}\sinh^2 s_{li}\\
=&\, 2\cosh s_{ji}\cosh s_{ki}\cosh s_{li}+\cosh^2 s_{ji} +\cosh^2 s_{ki} + \cosh^2 s_{li}-1;
\end{split}
\end{equation*}
and using the hyperbolic Law of Cosine to the triangles of truncation $T_j,$ $T_k$ and $T_l,$ we have 
$$\cosh s_{ji}=\frac{\cos\alpha_{kl} +\cos\alpha_{ik}\cos\alpha_{il}}{\sin\alpha _{ik}\sin\alpha_{il}},$$
$$\cosh s_{ki}=\frac{\cos\alpha_{jl} +\cos\alpha_{ij}\cos\alpha_{il}}{\sin\alpha_{ij}\sin\alpha_{il}},$$
and
$$\cosh s_{li}=\frac{\cos\alpha_{jk} +\cos\alpha_{ij}\cos\alpha_{ik}}{\sin\alpha_{ij}\sin\alpha_{ik}}.$$
Plugging these into the previous identity, we have (\ref{symmetry}). Since $\mathrm X(D)$ is an irreducible algebraic variety, by analyticity, (\ref{det2}), (\ref{det3}) and (\ref{det4}) hold for the other characters $[\rho]$ in $\mathrm X(D).$  

Since the square of the determinants in conditions (1) and (2) are rational functions in the coordinates $(\mathrm {Tr}\widetilde\rho([\gamma_{12}]),\dots, \mathrm {Tr}\widetilde\rho([\gamma_{34}])),$ the lifting of those characters form a Zariski-open subset of the $\mathrm{SL}(2;\mathbb C)$ character variety of $D,$ and hence those characters themselves form a Zariski-open subset of $\mathrm X(D).$
\\

Next we show that the representations satisfying (1) and (2) are $\boldsymbol\gamma$-regular.  We will compute the homologies of $D$ using its spine $\Gamma,$ which is the $1$-dimensional $CW$ complex consisting of two $0$-cells $x_1$ and $x_2$ (one dual to each copy of $\Delta$) and four $1$-cells $a_1,$ $a_2,$ $a_3$ and $a_4$ (one dual to each hexagonal face $H_j$) all of which are oriented from $x_1$ to $x_2.$

Let $\{\mathbf e_1,\mathbf e_2,\mathbf e_3\}$ be the standard basis of $\mathbb C^3$ and let the choice of representatives $x_1,x_2,a_1,a_2, a_3$ and $a_4$ in the universal covering of $\Gamma$ as drawn in Figure \ref{spine2}.  Then $\mathrm C_0(D;\mathrm{Ad}_\rho)\cong \mathbb C^6$ with a natural basis $\{\mathbf e_i\otimes x_k\}$ for $i\in\{1,2,3\}$ and  $k\in\{1,2\};$ $\mathrm C_1(D;\mathrm{Ad}_\rho)\cong \mathbb C^{12}$ with a natural basis $\{\mathbf e_i\otimes a_k\}$ for $i\in\{1,2,3\}$ and $k\in\{1,2,3,4\};$ and $\mathrm C_k(D;\mathrm{Ad}_\rho)=0$ for $k\neq 0$ or $1.$ 

\begin{figure}[htbp]
\centering
\includegraphics[scale=0.6]{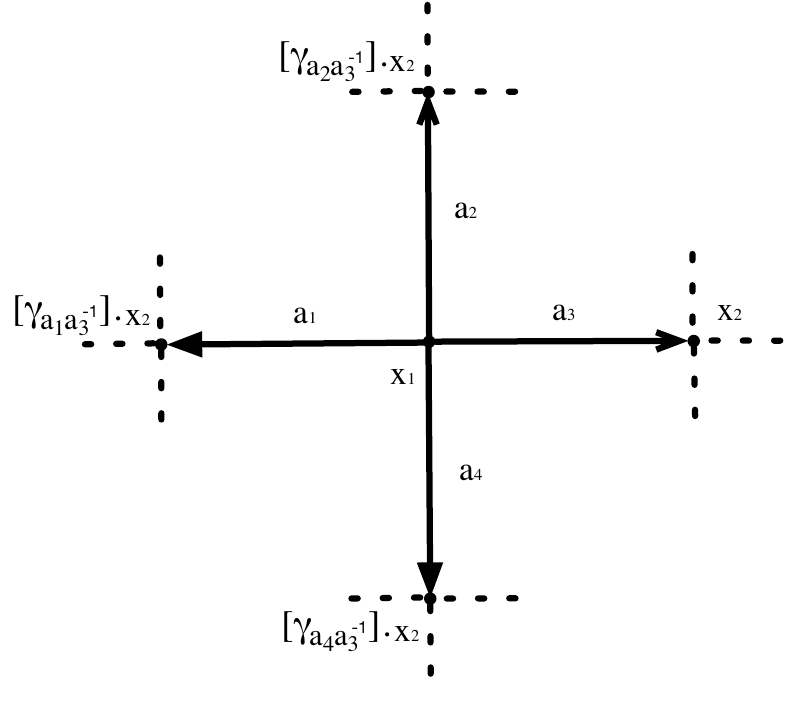}
\caption{}
\label{spine2} 
\end{figure}

We choose $x_1$ to be the base point of the fundamental group; and for $\{j,k\}\subset \{1,2,3\},$ let $\gamma_{a_ja_k^{-1}}$ be the curve starting from $x_1$ traveling along $a_j$ to $x_2$ then along $-a_k$ back to $x_1.$ In this way, we have $[\gamma_{a_ka_j^{-1}}]=[\gamma_{jk}]^{\pm 1}.$ Checking the orientation carefully we have $[\gamma_{a_1a_2^{-1}}]=[\gamma_{12}],$ $[\gamma_{a_2a_3^{-1}}]=[\gamma_{23}]$ and $[\gamma_{a_1a_3^{-1}}]=[\gamma_{13}].$

By condition (1), we see that the vectors \{$\mathbf I_{jk}\otimes (a_j-a_k)\},$ $\{j,k\}\subset\{1,2,3,4\},$  are linearly independent in $\mathrm C_1(D; \mathrm{Ad}_\rho).$ To show that they lie in the kernel of $\partial : \mathrm C_1(D;\mathrm{Ad}_\rho)\to \mathrm C_0(D;\mathrm{Ad}_\rho),$ we have
\begin{equation*}
\begin{split}
\partial (\mathbf I_{jk}\otimes (a_j-a_k))&=\mathbf I_{jk}\otimes \partial (a_j -a_k)\\
&=\mathbf I_{jk}\otimes \Big((x_1-[\gamma_{a_ja_3^{-1}}]\cdot x_2)-(x_1-[\gamma_{a_ka_3^{-1}}]\cdot x_2)\Big)\\
&=\mathbf I_{jk}\otimes \Big([\gamma_{a_ka_3^{-1}}]\cdot x_2- [\gamma_{a_ja_3^{-1}}]\cdot x_2\Big)\\
&=\Big(\mathrm{Ad}_\rho ([\gamma_{a_ka_3^{-1}}])^T \cdot\mathbf I_{jk}-\mathrm{Ad}_\rho ([\gamma_{a_ja_3^{-1}}])^T \cdot \mathbf I_{jk}\Big)\otimes x_2\\
&=\Big(\mathrm{Ad}_\rho ([\gamma_{a_ka_3^{-1}}])^T \mathrm{Ad}_\rho ([\gamma_{a_ja_k^{-1}}])^T\cdot \mathbf I_{jk}-\mathrm{Ad}_\rho([\gamma_{a_ja_3^{-1}}])^T \cdot \mathbf I_{jk}\Big)\otimes x_2=0,
\end{split}
\end{equation*}
where the penultimate equality comes from $\mathrm{Ad}_\rho ([\gamma_{a_ja_k^{-1}}])^T\cdot \mathbf I_{jk}=\mathrm{Ad}_\rho ([\gamma_{jk}]^{\pm 1})^T\cdot \mathbf I_{jk}=\mathbf I_{jk}$ and the last equation comes from $\gamma_{a_ja_k^{-1}}\cdot\gamma_{a_ka_3^{-1}}=\gamma_{a_ja_3^{-1}}.$ Therefore, $\{\mathbf I_{jk}\otimes (a_j-a_k)\},$ $\{j,k\}\subset\{1,2,3,4\},$ represent six linearly independent elements $\{\mathbf I_{jk}\otimes [\gamma_{jk}]\}$ in $\mathrm H_1(D;\mathrm{Ad}_\rho).$ Later we will prove that they also span, and hence form a basis of $\mathrm H_1(D;\mathrm{Ad}_\rho).$

Now we claim that these six vectors $\{\mathbf I_{12}\otimes (a_1-a_2), \mathbf I_{13}\otimes (a_1-a_3), \mathbf I_{14}\otimes (a_1-a_4), \mathbf I_{23}\otimes (a_2-a_3), \mathbf I_{24}\otimes (a_2-a_4), \mathbf I_{34}\otimes (a_3-a_4)\}$  joint with the other six vectors $\{\mathbf I_{13}\otimes a_3, \mathbf I_{23}\otimes a_3, \mathbf I_{34}\otimes a_3, \mathbf I_{12}\otimes a_1,  \mathbf I_{14}\otimes a_1, \mathbf I_{24}\otimes a_2\}$ form a basis of 
$\mathrm C_1(D;\mathrm{Ad}_\rho).$ Indeed, in the natural basis $\{\mathbf e_i\otimes a_k\}$ for $i\in\{1,2,3\}$ and $k\in\{1,2,3,4\},$ the $12\times 12$ matrix consisting of these vectors as the columns is obtained from the one consisting of $\{\mathbf I_{jk}\otimes a_k\},$ $k\in\{1,2,3,4\}$ and $j\neq k,$ as the columns by a sequence of elementary column operations of type I, III, and II with a factor $-1.$ The latter matrix is a block matrix with four $3\times3$ blocks $[\mathbf I_{12},\mathbf I_{13},\mathbf I_{14}],$ $[\mathbf I_{12},\mathbf I_{23},\mathbf I_{24}],$ $[\mathbf I_{13},\mathbf I_{23},\mathbf I_{34}]$ and  $[\mathbf I_{14},\mathbf I_{24},\mathbf I_{34}]$ on the diagonal and $0's$ elsewhere, hence has determinant  
$$\det[\mathbf I_{12},\mathbf I_{13},\mathbf I_{14}]\cdot\det[\mathbf I_{12},\mathbf I_{23},\mathbf I_{24}]\cdot\det[\mathbf I_{13},\mathbf I_{23},\mathbf I_{34}]\cdot\det[\mathbf I_{14},\mathbf I_{24},\mathbf I_{34}]$$
and by condition  (1) is non-singular. As a consequence, the former matrix is also non-singular and up to sign has the same determinant.

Next we will study the image of the six vectors $\{\mathbf I_{13}\otimes a_3, \mathbf I_{23}\otimes a_3, \mathbf I_{34}\otimes a_3, \mathbf I_{12}\otimes a_1,  \mathbf I_{14}\otimes a_1, \mathbf I_{24}\otimes a_2\}$ under the boundary map $\partial,$ and show that they span $\mathrm C_0(D;\mathrm{Ad}_\rho).$ We have for $j=1,2,4,$
$$\partial( \mathbf I_{j3}\otimes a_3)= \mathbf I_{j3}\otimes \partial a_3=\mathbf I_{j3}\otimes (x_1-x_2)=\mathbf I_{j3}\otimes x_1-\mathbf I_{j3}\otimes x_2;$$
for $k=2,4,$
$$\partial (\mathbf I_{1k}\otimes a_1)= \mathbf I_{1k}\otimes \partial a_1=\mathbf I_{1k}\otimes (x_1-[\gamma_{a_1a_3^{-1}}]\cdot x_2)=\mathbf I_{1k}\otimes x_1-\Big(\mathrm{Ad}_\rho ([\gamma_{13}])^T\cdot\mathbf I_{1k}\Big)\otimes x_2;$$
and
$$\partial (\mathbf I_{24}\otimes a_2)= \mathbf I_{24}\otimes \partial a_2=\mathbf I_{24}\otimes (x_1-[\gamma_{a_2a_3^{-1}}]\cdot x_2)=\mathbf I_{24}\otimes x_1-\Big(\mathrm{Ad}_\rho ([\gamma_{23}])^T\cdot\mathbf I_{24}\Big)\otimes x_2.$$
Therefore, in the natural basis $\{\mathbf e_i\otimes x_k\},$ $i\in\{1,2,3\},$ $k\in\{1,2\},$ the $6\times 6$ matrix consisting of $\{\partial(\mathbf I_{13}\otimes a_3), \partial(\mathbf I_{23}\otimes a_3), \partial(\mathbf I_{34}\otimes a_3), \partial(\mathbf I_{12}\otimes a_1),  \partial(\mathbf I_{14}\otimes a_1), \partial(\mathbf I_{24}\otimes a_2)\}$ as the columns has four $3\times 3$ blocks, where on the top left it has $[\mathbf I_{13},\mathbf I_{23},\mathbf I_{34}]$ and on the bottom left it has $[-\mathbf I_{13},-\mathbf I_{23},-\mathbf I_{34}];$ on the top right it has $[\mathbf I_{12},\mathbf I_{14},\mathbf I_{24}]$ and on the bottom right it has
$$\Big[-\mathrm{Ad}_\rho ([\gamma_{13}])^T\cdot\mathbf I_{12},\ -\mathrm{Ad}_\rho ([\gamma_{13}])^T\cdot\mathbf I_{14},\ -\mathrm{Ad}_\rho ([\gamma_{23}])^T\cdot\mathbf I_{24}\Big].$$
This matrix is row equivalent to (by adding the top blocks to the bottom) the one with $[\mathbf I_{13},\mathbf I_{23},\mathbf I_{34}]$ on the top left, $0's$ on the bottom left and
$$\Big[\mathbf I_{12}-\mathrm{Ad}_\rho ([\gamma_{13}])^T\cdot\mathbf I_{12},\ \mathbf I_{14}-\mathrm{Ad}_\rho ([\gamma_{13}])^T\cdot\mathbf I_{14},\mathbf I_{24}-\mathrm{Ad}_\rho ([\gamma_{23}])^T\cdot\mathbf I_{24}\Big]$$
on the bottom right. Hence the determinant of both of the $6\times 6$ matrices are equal to
$$\det[\mathbf I_{13},\mathbf I_{23},\mathbf I_{34}]\cdot\det\Big[\mathbf I_{12}-\mathrm{Ad}_\rho ([\gamma_{13}])^T\cdot\mathbf I_{12},\ \mathbf I_{14}-\mathrm{Ad}_\rho ([\gamma_{13}])^T\cdot\mathbf I_{14},\mathbf I_{24}-\mathrm{Ad}_\rho ([\gamma_{23}])^T\cdot\mathbf I_{24}\Big].$$

By conditions (1) and (2),  the product above is nonzero, and hence  $\{\partial(\mathbf I_{13}\otimes a_3), \partial(\mathbf I_{23}\otimes a_3), \partial(\mathbf I_{34}\otimes a_3), \partial(\mathbf I_{12}\otimes a_1),  \partial(\mathbf I_{14}\otimes a_1), \partial(\mathbf I_{24}\otimes a_2)\}$ span $\mathrm C_0(D;\mathrm{Ad}_\rho).$ This implies that $\mathrm H_0(D;\mathrm{Ad}_\rho)=0.$ Since there are no cells of dimension higher than or equal to $2,$ $\mathrm H_k(D;\mathrm{Ad}_\rho)=0$ for $k\geqslant 2.$

Now since  $\{\partial(\mathbf I_{13}\otimes a_3), \partial(\mathbf I_{23}\otimes a_3), \partial(\mathbf I_{34}\otimes a_3), \partial(\mathbf I_{12}\otimes a_1),  \partial(\mathbf I_{14}\otimes a_1), \partial(\mathbf I_{24}\otimes a_2)\}$ span $\mathrm C_0(D;\mathrm{Ad}_\rho)\cong \mathbb C^6,$ by dimension counting the kernel of $\partial: \mathrm C_1(D;\mathrm{Ad}_\rho) \to \mathrm C_0(D;\mathrm{Ad}_\rho)$ has dimension at most $6.$ Hence $\{\mathbf I_{12}\otimes (a_1-a_2), \mathbf I_{13}\otimes (a_1-a_3), \mathbf I_{14}\otimes (a_1-a_4), \mathbf I_{23}\otimes (a_2-a_3), \mathbf I_{24}\otimes (a_2-a_4), \mathbf I_{34}\otimes (a_3-a_4)\}$ span the kernel of $\partial.$ This shows that the elements they represent  $\mathbf h_D=\{\mathbf I_{jk}\otimes [\gamma_{jk}]\},$ $\{j,k\}\subset \{1,2,3,4\},$ form a basis of $\mathrm H_1(D;\mathrm{Ad}_\rho),$ and $\mathrm H_1(D;\mathrm{Ad}_\rho)\cong \mathbb C^6.$ This completes the proof. 
\end{proof}

\begin{proof}[Proof of Proposition \ref{D}] Since the adjoint twisted Reidemeister torsion is invariant under subdivisions, elementary expansions and elementary collapses of CW-complexes  by \cite{M,T}, we can do the computation using the spine $\Gamma$ of $D.$

The adjoint twisted Reidemeistor torsion equals, up to sign, the determinant of the $12\times 12$ matrix consisting of $\{\mathbf I_{12}\otimes (a_1-a_2), \mathbf I_{13}\otimes (a_1-a_3), \mathbf I_{14}\otimes (a_1-a_4), \mathbf I_{23}\otimes (a_2-a_3), \mathbf I_{24}\otimes (a_2-a_4), \mathbf I_{34}\otimes (a_3-a_4), \mathbf I_{13}\otimes a_3, \mathbf I_{23}\otimes a_3, \mathbf I_{34}\otimes a_3, \mathbf I_{12}\otimes a_1,  \mathbf I_{14}\otimes a_1, \mathbf I_{24}\otimes a_2\}$ as the columns divided by the determinant of the $6\times 6$ matrix consisting of $\{\partial(\mathbf I_{13}\otimes a_3), \partial(\mathbf I_{23}\otimes a_3), \partial(\mathbf I_{34}\otimes a_3), \partial(\mathbf I_{12}\otimes a_1),  \partial(\mathbf I_{14}\otimes a_1), \partial(\mathbf I_{24}\otimes a_2)\}$ as the columns.

By (\ref{det2}), (\ref{det3}) and (\ref{det4}), we have for the holonomy representation of a hyperbolic D-block,
\begin{equation*}
\begin{split}
&\mathrm{Tor}(D, \mathbf h_D;\mathrm{Ad}_\rho)\\
=&\pm\frac{\det[\mathbf I_{12},\mathbf I_{13},\mathbf I_{14}]\cdot\det[\mathbf I_{12},\mathbf I_{23},\mathbf I_{24}]\cdot\det[\mathbf I_{13},\mathbf I_{23},\mathbf I_{34}]\cdot\det[\mathbf I_{14},\mathbf I_{24},\mathbf I_{34}]}{\det[\mathbf I_{13},\mathbf I_{23},\mathbf I_{34}]\cdot\det\Big[\mathbf I_{12}-\mathrm{Ad}_\rho ([\gamma_{13}])^T\cdot\mathbf I_{12},\ \mathbf I_{14}-\mathrm{Ad}_\rho ([\gamma_{13}])^T\cdot\mathbf I_{14},\mathbf I_{24}-\mathrm{Ad}_\rho ([\gamma_{23}])^T\cdot\mathbf I_{24}\Big]}\\
=&\pm\frac{\mathbf i \sinh l_{14}\sinh s_{24}\sinh s_{34}}{32\sin\alpha_{12}\sin\alpha_{13}\sin\alpha_{23}}\\
=&\pm\frac{\sqrt{\det \mathrm G_{\boldsymbol\alpha}}}{32\sin\alpha_{12}\sin\alpha_{13}\sin\alpha_{14}\sin\alpha_{23}\sin\alpha_{24}\sin\alpha_{34}}\\
=&\pm\frac{\sqrt{\det\mathbb G\Big(\frac{u_{12}}{2}, \frac{u_{13}}{2}, \frac{u_{14}}{2}, \frac{u_{23}}{2}, \frac{u_{24}}{2},\frac{u_{34}}{2}\Big)}}{32\sinh\frac{ u_{12}}{2}\sinh\frac{ u_{13}}{2}\sinh\frac{ u_{14}}{2}\sinh\frac{ u_{23}}{2}\sinh\frac{ u_{24}}{2}\sinh\frac{ u_{34}}{2}},
\end{split}
\end{equation*}
where the last  equality comes from (\ref{symmetry}).

Finally, by Lemma \ref{LMD} and the analyticity of the involved functions, the result holds for all $\boldsymbol\gamma$-regular characters  in $\mathrm X(D).$
 \end{proof}



\section{Reidemeister torsion of the Mayer-Vietoris sequence}\label{TMV}

Let $M$ be the complement of a fundamental shadow link with $n$ components, and let $\rho:\pi_1(M)\to\mathrm{PSL}(2;\mathbb C)$ be an irreducible representation. We insert a thickened pair of pants if necessary so that no $D$-block self-intersects. Suppose there are in total $c$ thickened pairs of pants inserted, and the $3$-dimensional objects ($D$-blocks  and the thickened pairs of pants) intersect at $p$ pairs of pants, then we have $p=c+2d.$ Order the $c$ thickened pair of pants together with the $d$ $D$-blocks by $D_1,\dots, D_{c+d},$ and order  the $p$ pairs of pants by $P_1,\dots, P_p.$ Then by Lemma \ref{shortexact} there is the following short exact sequence of chain complexes
\begin{equation*}
0\to \bigoplus_{j=1}^p \mathrm C_*(P_j;\mathrm{Ad}_\rho) \xrightarrow{\delta} \bigoplus_{k=1}^{c+d} \mathrm C_*(D_k;\mathrm{Ad}_\rho)\xrightarrow{\epsilon}  \mathrm C_*(M;\mathrm{Ad}_\rho) \to 0
\end{equation*}
with $\epsilon$ defined by the sum
\begin{equation}\label{sum}
\epsilon(\mathbf c_1,\dots,\mathbf c_{c+d})=\sum_{k=1}^{c+d}\mathbf c_k
\end{equation}
and $\delta$ defined by the alternating sum 
\begin{equation}\label{alt}
(\delta \mathbf c)_k=-\sum_j\mathbf c_j+\sum_l\mathbf c_l,
\end{equation}
where $j$ runs over the indices such that $P_j=D_{k'}\cap D_k$ for some $k'<k$ and $l$ runs over the indices such that $P_l=D_k\cap D_{k''}$ for some $k<k''.$ 

For each $i\in\{1,\dots, n\},$ let $T_i=\partial N(L_i)$ be the boundary of a tubular neighborhood of the $i$-th component of $L_{\text{FSL}},$ $ m_i$ be the meridian of $N(L_i)$ and $\boldsymbol m=( m_1,\dots,  m_n).$  Suppose $\rho$ is an $\boldsymbol m$-regular representation whose restriction to each pair of pants $P_j$ is  $\boldsymbol \gamma$-regular as defined  in Definition \ref{Preg}, and to 
each $D$-block $D_k$  is $\boldsymbol \gamma$-regular as defined  in Definition \ref{Dreg}, then the induced Mayer-Vietoris exact sequence $\mathcal H$ has four nonzero terms, ie,
\begin{equation}\label{MV}
0\to \mathrm H_2(M;\mathrm{Ad}_\rho) \xrightarrow{\partial} \bigoplus_{j=1} ^p\mathrm H_1(P_j;\mathrm{Ad}_\rho)\xrightarrow{\delta} \bigoplus_{k=1}^{c+d} \mathrm H_1(D_k;\mathrm{Ad}_\rho)\xrightarrow{\epsilon}  \mathrm H_1(M;\mathrm{Ad}_\rho) \to 0.
\end{equation}
Let $\mathbf I_i$ be up to sign the unique invariant vector of $\mathrm{Ad}_\rho([ m_i])^T$ with $\kappa(\mathbf I_i,\mathbf I_i)=1.$ Then by a diagram chasing,  $\mathrm H_1(M;\mathrm{Ad}_\rho)$ has a basis $\mathbf h^1_{(M,\boldsymbol m)}=\{ \mathbf I_1\otimes [ m_1],\dots, \mathbf I_n\otimes [ m_n]\}$ and $\mathrm H_1(M;\mathrm{Ad}_\rho)$ has a basis $\mathbf h^2_{M}=\{ \mathbf I_1\otimes [ T_1],\dots, \mathbf I_n\otimes [ T_n]\}.$ 


\begin{proposition}\label{MVS} Let $\mathbf h_{P_j}$ be the basis of $\mathrm H_1(P_j;\mathrm{Ad}_\rho)$ in Definition \ref{Preg} and let $\mathbf h_{D_k}$  be the basis of  $\mathrm H_1(D_k;\mathrm{Ad}_\rho)$  in Definition  \ref{Dreg}.
Let $\mathbf h_{**}$ be the union of $\mathbf h^1_{(M,\boldsymbol m)},$ $\mathbf h^2_M,$ $\sqcup_j\mathbf h_{P_j}$ and $\sqcup_k\mathbf h_{D_k}.$ Then
\begin{equation} 
\mathrm{Tor}(\mathcal H, \mathbf h_{**})=\pm 1.
\end{equation}
\end{proposition}

\begin{proof} By \cite[Proposition 3.22, Corollary 3.23]{P}, Lemma \ref{LI} and Lemma \ref{LMD} and the fact that a thickened pair of pants is simple homotopic to a pair of pants, with the chosen bases $\mathbf h^1_{(M,\boldsymbol m)},$ $\mathbf h^2_M,$ $\sqcup_j\mathbf h_{P_j}$ and $\sqcup_k\mathbf h_{D_k},$ we have $$\mathrm H_2(M;\mathrm{Ad}_\rho)\cong \mathbb C^n,$$ $$\bigoplus_{j=1} ^p\mathrm H_1(P_j;\mathrm{Ad}_\rho)\cong \mathbb C^{3p},$$ $$\bigoplus_{k=1}^{c+d} \mathrm H_1(D_k;\mathrm{Ad}_\rho)\cong \mathbb C^{3c+6d}$$ and $$\mathrm H_1(M;\mathrm{Ad}_\rho)\cong \mathbb C^n.$$

In the rest of the proof, we will fix these isomorphisms and identify the linear maps $\partial,$ $\delta$ and $\epsilon$ with the left multiplications of the corresponding matrices. In particular, $\partial$ corresponds to a $3p \times n$ matrix, $\delta$ corresponds to a $ (3c+6d)\times 3p$ square matrix and $\epsilon$ corresponds to an $n\times (3c+6d)$ matrix.

For $C_3=\mathrm H_2(M;\mathrm{Ad}_\rho),$  we choose the lifting base $\widetilde{\mathbf b}_2$ to be $\mathbf h^2_M.$ Then 
\begin{equation}\label{D3}
[\widetilde{\mathbf b}_2 ;\mathbf h^2_M]=1.
\end{equation}

For $C_2= \bigoplus_{j=1} ^p\mathrm H_1(P_j;\mathrm{Ad}_\rho),$ we first order the vectors in $\widetilde{\mathbf b}_2=\mathbf h^2_M$ by $\{ \mathbf u_1,\dots, \mathbf u_n\}.$ Then $\mathbf b_2=\{\partial( \mathbf u_1),\dots,\partial ( \mathbf u_n)\}.$ We also order the vectors in $\sqcup_j\mathbf h_{P_j}$ by $\{\mathbf v_1,\dots,\mathbf v_{3p}\},$ and choose the lifting basis $\widetilde{\mathbf b}_1$ as follows. Since the sequence (\ref{MV}) is exact, $\delta$ has rank $3c+6d-n=3p-n.$ Suppose a basis of the column space of $\delta$ consists of the columns $\{\mathbf w_{j_1},\dots,\mathbf w_{j_{3c+6d-n}}\}$ of $\delta,$ then we let $\widetilde{\mathbf b}_1=\{\mathbf v_{j_1},\dots,\mathbf v_{j_{3p-n}}\}.$ Next we compute $\det[\mathbf b_2\sqcup\widetilde{\mathbf b}_1;\sqcup_j\mathbf h_{P_j}].$ Recall that there is a one-to-one correspondence between $\{ \mathbf u_1,\dots, \mathbf u_n\}$ and the boundary components $\{T_1,\dots, T_k\}$ of $M$ and a one-to-one correspondence between $\{\mathbf v_1,\dots,\mathbf v_{3p}\}$ and the boundary components of the disjoint union $\sqcup P_j$ of $\{P_j\}.$ Then a diagram chasing show that 
$$\partial( \mathbf u_k)=\sum_{s=1}^{n_k}\pm\mathbf v_{i_s},$$
where $n_k$ is the number of the boundary components of $\sqcup P_j$ intersecting $T_k,$ $\mathbf v_{i_1},\dots,\mathbf v_{i_{n_k}}$ are the vectors corresponding to those boundary components of $\sqcup_jP_j$ and the signs $\pm$ are determined as follows. Fix an orientation of the longitude $l_k$ of $T_k,$ and suppose $P_{i_s}=D_r\cap D_t$ and $D_r$ comes immediately before $D_t$ along $l_k$ in the chosen orientation. Then the sign in front of $\mathbf v_{i_s}$ is $+$ if $r>t,$ and is $-$ if otherwise. Since each boundary component of $\sqcup_jP_j$ intersects exactly one boundary component of $M,$ each row of the $n\times 3p$ matrix $\partial$ has exactly one nonzero entry, which equals either $1$ or $-1.$ Therefore, rows $j_1,\dots,j_{3p-n}$ of the matrix $\mathbf b_2\sqcup\widetilde{\mathbf b}_1$ have exactly two nonzero entries, one from $\mathbf b_2$ and one from $\widetilde{\mathbf b}_1;$ and the other rows of $\mathbf b_2\sqcup\widetilde{\mathbf b}_1$ have exactly one nonzero entry. Let $M$ be the $(3p-n)\times(3p-n)$ matrix consisting of the rows $j_1,\dots,j_{3p-n}$ of the columns $\mathbf v_{j_1},\dots,\mathbf v_{j_{3p-n}}$ of $\mathbf b_2\sqcup\widetilde{\mathbf b}_1,$ and let $N$ be the $n\times n$ matrix obtained from $\mathbf b_2\sqcup\widetilde{\mathbf b}_1$ by removing those rows and columns. Then each row of $M$ and $N$ contains exactly one nonzero entry, which equals $1$ or $-1,$ hence $\det M=\pm 1,$ $\det N=\pm 1$ and $\det[\mathbf b_2\sqcup\widetilde{\mathbf b}_1]=\pm\det M\cdot\det N=\pm1.$ Therefore, 
\begin{equation}\label{D2}
[\mathbf b_2\sqcup\widetilde{\mathbf b}_1;\sqcup_j\mathbf h_{P_j}]=\pm1.
\end{equation}

For $C_1=\bigoplus_{k=1}^{c+d} \mathrm H_1(D_k;\mathrm{Ad}_\rho),$ we have $\mathbf b_1=\{\delta(\mathbf v_{j_1}),\dots,\delta(\mathbf v_{j_{3p-n}})\}=\{\mathbf w_{j_1},\dots,\mathbf w_{j_{3c+6d-n}}\}.$ We choose the lifting basis $\widetilde{\mathbf b}_0$ as follows. Since each $P_j$ is adjacent to two of $\{D_1,\dots,D_{c+d}\}$ without redundancy and each edge of $D_k$ connects two of $\{P_1,\dots,P_p\}$ without redundancy, 
by (\ref{alt}) each row of $\delta$ has exactly two nonzero entries each of which equals $1$ or $-1,$ and each column of $\delta$ has exactly two nonzero entries, one equals $1$ and the other equals $-1.$ For $t\notin \{j_1,\dots,j_{3c+6d-n}\},$ let $\mathbf x_t\in\mathbb C^{3c+6d}$ be the vector obtained from the column $\mathbf w_t$ of $\delta$ by replacing the entry $-1$ by $0.$ Then we let $\widetilde{\mathbf b}_0=\big\{\mathbf x_t\ |\ t\in \{1,\dots,3c+6d\}\setminus\{j_1,\dots,j_{3c+6d-n}\}\big\}.$ Now we claim that $\{\mathbf x_t\}$ are linearly independent and $\epsilon(\mathbf x_t)\neq 0$ for each $t$ so that $\mathbf b_1\sqcup\widetilde{\mathbf b}_0$ form a basis of $C_1.$ Indeed, since each $\mathbf x_t$ contains only one nonzero component, to prove the linear independence it suffices to prove that no two nonzero entries of $\{\mathbf x_t\}$ are in the same row. Suppose otherwise that $\mathbf x_{t_1}$ and $\mathbf x_{t_2}$ have nonzero components in  row $k,$ then due to the fact that each row of $\delta$ has only two nonzero entries, the $k$-th component of all the comlmns $\mathbf w_{j_1},\dots,\mathbf w_{j_{3c+6d-n}}$ are $0.$ This contradicts the fact that $\{\mathbf w_{j_1},\dots,\mathbf w_{j_{3c+6d-n}}\}$ is a basis of the column space of $\delta$ since $\mathbf w_{t_1}$ and $\mathbf w_{t_2}$ have the $k$-th component equal to $1$ and neither of them can be written as a linear combination of $\{\mathbf w_{j_1},\dots,\mathbf w_{j_{3c+6d-n}}\}.$ Also, since each edge of $D_k$ belongs to exactly one boundary component of $M,$ by (\ref{sum}) $\epsilon(\mathbf x_t)$ has exactly one nonzero component which equals $1,$ hence is nonzero. This finishes the proof of the claim. Next, we compute $\det[\mathbf b_1\sqcup\widetilde{\mathbf b}_0].$ We observe that the matrix $[\mathbf b_1\sqcup\widetilde{\mathbf b}_0]$ satisfies the following three properties: 
\begin{enumerate}[(I)]
\item It is nonsingular.
\item Each column has either exactly one nonzero component which equals $\pm 1;$ or has exactly two nonzero components, one equals $1$ and the other equals $-1.$
\item There is at least one column containing exactly one nonzero component. 
\end{enumerate}
We let $t_1,\dots,t_n$ be the rows where some $\mathbf x_t$ has nonzero components. Let $M_1$ be the $n\times n$ matrix consisting of the rows $t_1,\dots,t_n$ of the vectors $\{\mathbf x_t\},$ and let $N_1$ be the $(3c+6d-n)\times(3c+6d-n)$ matrix obtained from $\mathbf b_1$ by removing those rows. Since each column of $M_1$ contains exactly one $1$ and no two $1'$s are in the same row, $\det M_1=\pm 1.$ As a consequence, we have $\det[\mathbf b_1\sqcup\widetilde{\mathbf b}_0]=\pm\det M_1\cdot\det N_1=\pm \det N_1.$ We claim that $N_1$ also satisfies the properties (I), (II) and (II). Indeed, (I) comes from the equality right above and (II) comes from the construction of $N_1.$ For (III), suppose otherwise that all the columns of $N_1$ has one $1$ and one $-1,$ then all rows of $N_1$ add up to zero and $N_1$ is singular, which contradicts (I). Therefore, we can collect all the columns of $N_1$ containing only one nonzero components, and let $M_2$ be the square matrix consisting of the rows that contain those nonzero components, and let $N_2$ be the square matrix consisting of the other columns with those rows removed. Then $\det N_1=\det M_2\cdot\det N_2.$ Since $\det N_1\neq 0,$ we have $\det M_2\neq 0.$ This implies that no two nonzero components of $M_2$ are in the same row. Together with the fact that all the columns of $M_2$ has only one nonzero entry $\pm 1,$ we have $\det M_2=\pm 1.$ This implies that $\det N_1=\pm\det N_2.$ By the same argument, we have that $N_2$ satisfies properties (I), (II) and (III), and we can recursively construct smaller square matrices $M_3,N_3,\dots, M_k, N_k, \dots$ that $M_k$ consists of the rows containing those nonzero entries of the columns of $N_{k-1}$ containing exactly one nonzero entry and $N_k$ consists of the other columns of $N_{k-1}$ with those rows removed, so that $\det M_k=\pm 1,$ $\det N_{k-1}=\pm \det M_k\cdot \det N_k=\pm \det N_k$ and $N_k$ satisfies (I), (II) and (III). This algorithm stops at some $k$ when all columns of $N_k$ contain exactly one nonzero entry $\pm 1,$  and we have $\det[\mathbf b_1\sqcup\widetilde{\mathbf b}_0]=\pm\det N_1=\dots=\pm \det N_k=\pm 1.$ Therefore,
\begin{equation}\label{D1}
[\mathbf b_1\sqcup\widetilde{\mathbf b}_0;\sqcup_k\mathbf h_{P_k}]=\pm1.
\end{equation}

For $C_0= \mathrm H_1(M;\mathrm{Ad}_\rho),$ we have $\mathbf b_0=\big\{ \epsilon(\mathbf x_t)\ |\ t\in \{1,\dots,3c+6d\}\setminus\{j_1,\dots,j_{3c+6d-n}\}\big\}.$ Since $\mathbf b_1\sqcup\widetilde{\mathbf b}_0$ form a basis of $C_1$ and $\mathbf b_1$ lies in the kernel of $\epsilon,$ $\mathbf b_0$ is a basis of $C_0.$ In the previous paragraph, we show that each $\epsilon(\mathbf x_t)$ contains exactly nonzero entry $1,$ hence $\det[\mathbf b_0]=\pm 1,$ which is the same as
\begin{equation}\label{D0}
[\mathbf b_0;\mathbf h^1_{(M,\boldsymbol m)}]=\pm1.
\end{equation}

Therefore, by (\ref{D3}), (\ref{D2}), (\ref{D1}) and (\ref{D0}), we have
$$\mathrm{Tor}(\mathcal H;\mathbf h_{**})=\frac{[\widetilde{\mathbf b}_2;\mathbf h^2_M]\cdot[\mathbf b_1\sqcup\widetilde{\mathbf b}_0;\sqcup_k\mathbf h_{P_k}]}{[\mathbf b_2\sqcup\widetilde{\mathbf b}_1;\sqcup_j\mathbf h_{P_j}]\cdot[\mathbf b_0;\mathbf h^1_{(M,\boldsymbol m)}]}=\pm 1.$$
 \end{proof}



\section{Proof of Theorems \ref{main1}, \ref{main3} and \ref{main2}}

\begin{proof}[Proof of Theorem \ref{main1}]  

For (1), let $M$ be a fundamental shadow link complement. Recall that $M$ is the union of $D$-blocks by orientation reversing homeomorphisms between the $3$-puncture spheres (which is homeomorphic to a pair of pants). For each pair of pants $P$ and $i\in\{1,2,3\},$ let $\gamma_i$ be the simple closed curve around the puncture $p_i;$ and for each D-block $D$ and $\{j,k\}\subset\{1,2,3,4\},$ let $\gamma_{jk}$ be the simple closed curve around the edge $e_{jk}.$ Then $(\gamma_1,\gamma_2,\gamma_3)$ is the restriction of the meridians $\boldsymbol m$ of $M$  to $P,$ and $(\gamma_{12},\dots,\gamma_{34})$ is the restriction  of $\boldsymbol m$ to $D.$
Let $\rho:\pi_1(M)\to\mathrm{PSL}(2;\mathbb C)$ be an $\boldsymbol m$-regular representation, and we will consider the following three cases:
\begin{enumerate}[\text{Case} I.]
\item  The restriction of $[\rho]$  to each pair of pants $P_j$ is  $\boldsymbol \gamma$-regular as defined  in Definition \ref{Preg}, and to 
each $D$-block $D_k$  is $\boldsymbol \gamma$-regular as defined  in Definition \ref{Dreg}.
\item $[\rho]$ is not in Case I, and $\mathrm {Tr}\rho([m_i])\neq\pm 2$ for all $i\in\{1,\dots, n\}.$
\item Otherwise.
\end{enumerate}

If $[\rho]$ is in Case I, then by Theorem \ref{MaV}, Propositions \ref{P}, \ref{D} and \ref{MVS},  we have
$$\mathbb T_{(M,\boldsymbol m)}([\rho])=\mathrm{Tor}(M; \{\mathbf h^1_{(M,\boldsymbol m)}, \mathbf h^2_M\}; \mathrm {Ad}_\rho)=\pm2^{3d}\prod_{k=1}^d \sqrt{\det \mathbb G_k}.$$
This completes the proof of (1) for $[\rho]$ in Case I.
\\


Next we show that each  $[\rho]$ in  Case II and Case III  is in the closure of the set of characters in Case I in the classical (Hausdorff) topology, and the continuity of adjoint twisted Reidemeister torsion and  the determinants of the Gram matrix functions will complete the proof. 
\\

For Case II, we first recall \cite[Proposition 5.13]{P2} that, if $\rho$ is $\boldsymbol m$-regular and $\mathrm {Tr}\rho([m_i])\neq\pm 2$ for all $i\in\{1,\dots, n\},$ ie, is in Case II, then the logarithmic holonomies  $(u_1,\dots, u_n)$ form a local coordinates of $\mathrm X(M)$ near $[\rho].$ Since the restriction of $[\rho]$ to each $P_j$ and $D_k$ will possibly identity the  traces of certain curves in $\boldsymbol \gamma,$ we consider the following 
subsets of $\mathrm X(P_j)$ and $\mathrm X(D_k).$ For an equivalence relation $\sim$ on the index set $I_P=\{1,2,3\}$ with the set of equivalence classes $\overline {I_P},$ let 
$$\mathrm X_{\overline {I_P}}(P)=\big\{ [\rho]\in\mathrm X(P)\ \big|\ \text{for any lifting }\widetilde\rho\text{ of }\rho, \mathrm{Tr}\widetilde\rho([\gamma_a])=\pm\mathrm{Tr}\widetilde\rho([\gamma_b]) \text{ for } a, b\in I_P \text{ with } a\sim b\big\};$$
and for an equivalence relation $\sim$ on the index set $I_D=\{12,\dots, 34\}$ with  the set of equivalence classes $\overline {I_D},$ let 
$$\mathrm X_{\overline {I_D}}(D)=\big\{ [\rho]\in\mathrm X(D)\ \big|\  \text{for any lifting }\widetilde\rho\text{ of }\rho, \mathrm{Tr}\widetilde\rho([\gamma_{c}])=\pm\mathrm{Tr}\widetilde\rho([\gamma_{d}]) \text{ for } c,d\in I_D\text{ with }  c\sim d\big\}.$$
Then the restriction of $[\rho]$ to each $P_j$ is in $\mathrm X_{\overline {I_P}}(P_j)$ for some  $\overline {I_P};$ and the restriction of $[\rho]$ to each $D_k$ is in $\mathrm X_{\overline {I_D}}(D_k)$ for some $\overline {I_D}.$ Let
$$\mathrm Z_{\overline {I_P}}(P)=\mathrm Z(P)\cap \mathrm X_{\overline {I_P}}(P)$$
and let
$$\mathrm Z_{\overline {I_D}}(D)=\mathrm Z(D)\cap \mathrm X_{\overline {I_D}}(D).$$
Then by formulas (\ref{det}), (\ref{det5}) and (\ref{rational}), for any quotient set $\overline {I_P},$ $\mathrm Z_{\overline {I_P}}(P)$ is dense in $\mathrm X_{\overline {I_P}}(P)$ in the classical topology; and by (\ref{det2}), (\ref{det3}), (\ref{det4}) and (\ref{symmetry}), for any quotient set $\overline {I_D},$ $\mathrm Z_{\overline {I_D}}(D)$ is dense in $\mathrm X_{\overline {I_D}}(D)$  in the classical topology.  (Indeed, the numerators in the square root of the right hand side of both (\ref{rational}) and (\ref{symmetry}) have a constant term $-1$ which always stays under the identifications of the variables, hence the relevant analytic functions in the logarithmic holonomies never become the zero function.)  As a consequence, any character  in Case II is in the closure of the set of characters in Case I in the classical topology. This completes the proof of (1) for $[\rho]$ in Case II.
\\

For a character $[\rho]$ in Case III, we show that it can be smoothly  perturbed  into Case I or Case II. Recall that the Killing form $\kappa$ on $\mathfrak{psl}(2;\mathbb C)$ defines a non-degenerate bi-linear form $\langle\ ,\ \rangle: \mathrm H_1(M,\mathrm{Ad}_\rho)\times \mathrm H^1(M,\mathrm{Ad}_\rho)\to\mathbb C,$ and the basis $\mathbf h^1_{(M,\boldsymbol m)}$ of $\mathrm H_1(M,\mathrm{Ad}_\rho)$ gives an isomorphism between $\mathrm H_1(M,\mathrm{Ad}_\rho)$ and  $\mathrm H^1(M,\mathrm{Ad}_\rho).$ For each $i\in \{1,\dots,n\},$ let $\mathbf v_i$ be the element in  $\mathrm H^1(M,\mathrm{Ad}_\rho)$ dual to $\mathbf I_i\otimes[m_i]$  under this isomorphism, ie, $\langle\mathbf v_i, \mathbf I_j\otimes [m_j]\rangle=\delta_{ij},$ the Kronecker symbol.  Let $I\subset \{1,\dots,n\}$ be the subset of the indices $i$ such that $\mathrm {Tr}\rho([m_i])=\pm 2,$ and let 
$$\mathbf v=\sum_{i\in I}\mathbf v_i.$$
 We consider $\mathbf v$  as a Zariski-tangent vector of $\mathrm X(M)$ at $[\rho].$ Since $[\rho]$ is $\boldsymbol m$-regular, it is a smooth point of $\mathrm X(M).$ As a consequence, $\mathbf v$ can be realized as the tangent vector of a deformation $[\rho_t],$ $t\in[0,\epsilon).$ Then $[\rho_t]$ is the desired perturbation of $[\rho],$ as for $t\neq 0,$ 
 $$\mathrm {Tr} \rho_t([m_i])\neq \mathrm {Tr} \rho([m_i])=\pm 2$$ 
for $i\in I,$ and $$\mathrm {Tr} \rho_{t}([m_j])=\mathrm {Tr} \rho([m_j])\neq \pm 2$$
for $j\notin I.$ This shows that any representation  in Case III is in the closure of the set of the representations in Cases I and II in the classical topology, and completes the proof of (1) for $[\rho]$ in Case III.
\\

(2) is a direct consequence of (1) and Theorem \ref{funT} (ii).
\end{proof}

\begin{proof}[Proof of Theorem \ref{main3}]  Let $\boldsymbol m$ be the system of meridians of $M.$ If the restriction $[\rho]$ of $[\rho_{\boldsymbol \mu}]$ to $M$ is $\boldsymbol m$-regular, then the result follows directly from Theorem \ref{main1} and Theorem \ref{funT} (iii).

If $[\rho]$ is not $\boldsymbol m$-regular, then by Theorem \ref{funT} (i) that $\boldsymbol m$-regular characters are dense in the distinguished component of $\mathrm X(M),$ $[\rho]$ is a limit point of $\boldsymbol m$-regular characters. Then by the analyticity of the adjoint twisted Reidemeister torsion, the formula has a removable singularity at $[\rho]$ and hence can be evaluated by taking the limit of the values at the nearby $\boldsymbol m$-regular characters.
\end{proof}

\begin{proof}[Proof of Theorem \ref{main2}] From Section \ref{dhp}, we see that $M$ is homeomorphic to a fundamental shadow link complement with the meridians (as of the fundamental shadow link complement) the preferred longitude $\boldsymbol l.$ Let $\boldsymbol m=(m_1,\dots,m_n)$ be the simple closed curves around the edges, and let $(\gamma_1,\dots,\gamma_n)$ be the double of the edges. Then the holonomy representation $\rho$ of the hyperbolic cone metric has the logarithmic holonomies $u_i=u_{\gamma_i}=2l_i$ and $u_{m_i}=2\mathbf i\theta_i$ for $i\in\{1,\dots,n\}.$ Since a truncated hyperideal tetrahedron is determined and infinitesimally determined by its six edge lengths,  $\rho$ is $\boldsymbol l$-regular; and by \cite[Theorem 1.2 (b)]{LY}, $\rho$ is determined and infinitesimally determined by its cone angles $(\theta_1,\dots,\theta_n),$ hence is $\boldsymbol m$-regular. Then  (1) and (2) respectively follow from Theorem \ref{main1} (1) and (2), and (3) follows from Theorem \ref{funT} (iii). 
\end{proof}



\noindent
Ka Ho Wong\\
Department of Mathematics\\  Texas A\&M University\\
College Station, TX 77843, USA\\
(daydreamkaho@math.tamu.edu)
\\

\noindent
Tian Yang\\
Department of Mathematics\\  Texas A\&M University\\
College Station, TX 77843, USA\\
(tianyang@math.tamu.edu)


\begin{thebibliography}{99}

\bibitem{BB} X. Bao and F. Bonahon, {\em Hyperideal polyhedra in hyperbolic $3$-space}, Bull. Soc. Math. France 130 (2002), no. 3, 457--491.

\bibitem{BY} G. Belletti and T. Yang, {\em Discrete Fourier transform, quantum 6j-symbols and deeply truncated tetrahedra}, preprint, arXiv:2009.03684.

\bibitem{BL} J. Bismut and F. Labourie, {\em Symplectic geometry and the Verlinde formulas}, Surveys in differential geometry: differential geometry inspired by string theory, 97--311, Surv. Differ. Geom., 5, Int. Press, Boston, MA, 1999.


\bibitem{BT} R. Bott and L. Tu, {\em Differential forms in algebraic topology}, Graduate Texts in Mathematics, 82. Springer-Verlag, New York-Berlin, 1982. xiv+331 pp. ISBN: 0-387-90613-4

\bibitem{CY} Q. Chen and T. Yang, {\em Volume Conjectures for the Reshetikhin-Turaev and the Turaev-Viro Invariants}, Quantum
Topol. 9 (2018), no. 3, 419--460.

\bibitem {C} F. Costantino, {\em $6j$-symbols, hyperbolic structures and the volume conjecture}, Geom. Topol. 11 (2007), 1831--1854.



\bibitem{CT} F. Costantino and D. Thurston, {\em 3-manifolds efficiently bound 4-manifolds}, J. Topol. 1 (2008), no. 3, 703--745.

\bibitem {CS} M. Culler and P. Shalen, {\em Varieties of group representations and
splittings of 3-manifolds},  Ann. of Math. (2) 117 (1983), 109--146.

\bibitem{DG} T. Dimofte and S. Garoufalidis, {\em The quantum content of the gluing equations}, Geom. Topol., 17(3):1253--1315,
2013.

\bibitem{FG} E. Falbel and A. Guilloux, {\em Dimension of character varieties for $3$-manifolds},  Proc. Amer. Math. Soc. 145 (2017), no. 6, 2727--2737.

\bibitem{FK} R., Fricke and F. Klein, {\em Vorlesungen fiber die Theorie der automorphen Funktionen}, vol. 1. Stuttgart:
B.G. Teubner 1897.


\bibitem{F} M. Fujii, {\em Hyperbolic $3$-manifolds with totally geodesic boundary}, Osaka J. Math. 27 (1990), no. 3, 539--553.

\bibitem{GRY} D. Gang, M. Romo and M. Yamazaki, {\em All-order volume conjecture for closed $3$-manifolds from complex Chern-Simons theory}. Comm. Math. Phys. 359 (2018), no. 3, 915--936.

\bibitem{G1} W. Goldman. {\em Trace coordinates on Fricke spaces of some simple hyperbolic surfaces}, Handbook of Teichmüller theory. Vol. II, 611--684,
IRMA Lect. Math. Theor. Phys., 13, Eur. Math. Soc., Zürich, 2009.


\bibitem{G} S. Gukov, {\em Three-dimensional quantum gravity, Chern-Simons theory, and the A-polynomial}, Comm. Math. Phys. 255 (2005) 577--627.


\bibitem{HP}  M. Heusener, Michael and J. Porti, {\em Holomorphic volume forms on representation varieties of surfaces with boundary},  Ann. H. Lebesgue 3 (2020), 341--380. 


\bibitem{HK} C. Hodgson and S. Kerckhoff, {\em Rigidity of hyperbolic cone-manifolds and hyperbolic Dehn surgery}, J. Differential Geom. 48 (1998), no. 1, 1--59. 

\bibitem{HK2}  C. Hodgson and S. Kerckhoff, {\em Universal bounds for hyperbolic Dehn surgery}, Ann. of Math. (2) 162 (2005), no. 1, 367--421. 

\bibitem{K} R. Kashaev, {\em The hyperbolic volume of knots from the quantum dilogarithm}, Lett. Math. Phys. 39 (1997), no. 3, 269--275.

\bibitem{KM} A. Kolpakov and J. Murakami, {\em Volume of a doubly truncated hyperbolic tetrahedron}, Aequationes Math. 85 (2013), no. 3, 449--463.

\bibitem{Luo} F. Luo {\em A combinatorial curvature flow for compact 3-manifolds with boundary}, Electron. Res. Announc. Amer. Math. Soc. 11 (2005), 12--20.

\bibitem{LY} F. Luo and T. Yang, {\em Volume and rigidity of hyperbolic polyhedral 3-manifolds}, J. Topol. 11 (2018), no. 1, 1--29.

\bibitem{M} J. Milnor, {\em Whitehead Torsion}, Bull. Amer. Math. Soc. 72 (1966), 358--426.

\bibitem{MM} H. Murakami and J. Murakami, {\em The colored Jones polynomials and the simplicial volume of a knot}, Acta Math. 186 (2001), no. 1, 85--104.


\bibitem{NZ} W. Neumann and D. Zagier, {\em Volumes of hyperbolic three-manifolds}, Topology 24 (1985), no. 3, 307--332.


\bibitem{O} T. Ohtsuki, {\em On the asymptotic expansion of the quantum $SU(2)$ invariant at $q=\exp(4\pi\sqrt{-1}/N)$ for closed hyperbolic 3-manifolds obtained by integral surgery along the figure-eight knot}, Algebr. Geom. Topol. 18 (2018), no. 7, 4187--4274. 


\bibitem{P} J. Porti, {\em Torsion de Reidemeister pour les vari\'et\'es hyperboliques}, Mem. Amer. Math. Soc., 128 (612):x+139, 1997.

\bibitem{P2} J. Porti, {\em Reidemeister torsion, hyperbolic three-manifolds, and character varieties}, Handbook of group actions. Vol. IV, 447--507, Adv. Lect. Math. (ALM), 41, Int. Press, Somerville, MA, 2018.

\bibitem{T} W. Thurston, {\em The geometry and topology of $3$-manifolds}, Princeton Univ. Math.
Dept. (1978). Available from http://msri.org/publications/books/gt3m/.


\bibitem{T} V. Turaev, {\em Introduction to combinatorial torsions}, Lectures in Mathematics ETH Z\"urich. Birkh\"auser Verlag, Basel, 2001. Notes taken by Felix Schlenk.

\bibitem{W} E. Witten {\em On quantum gauge theories in two dimensions}, 
Comm. Math. Phys. 141 (1991), no. 1, 153--209.

\bibitem{WY2} K. H. Wong and T. Yang, {\em Relative Reshetikhin-Turaev invariants, hyperbolic cone metrics and discrete Fourier transforms I }, arXiv:2008.05045.

\bibitem {WY} K. H. Wong and T. Yang,  {\em Asymptotic expansion of relative quantum invariants}, arXiv:2103.15056

\bibitem{Y} T. Yang, {\em A relative version of the Turaev-Viro invariants and the volume of hyperbolic polyhedral 3-manifolds,} arXiv: 2009.04813.
\end{thebibliography}
\end{document}